\documentclass[10pt]{article}

\usepackage{amssymb,latexsym,amsmath,enumerate,verbatim,amsfonts,amsthm}
\usepackage{color,dsfont} 
\usepackage[active]{srcltx}
\usepackage{bm}
\usepackage{pifont}     
\usepackage{hyperref}
\usepackage{mathtools}
\hypersetup{colorlinks=true, citecolor=black, linkcolor=black, urlcolor=blue}
\usepackage{url}
\usepackage{todonotes}
\usepackage{multirow}
\usepackage{makecell}
\usepackage{microtype}
\usepackage{cite}

\allowdisplaybreaks[4]

\newtheorem{theorem}{Theorem}[section]
\newtheorem{lemma}[theorem]{Lemma}
\newtheorem{definition}{Definition}[section]
\newtheorem{fact}[theorem]{Fact}

\newtheorem{corollary}[theorem]{Corollary}
\newtheorem{proposition}[theorem]{Proposition}
\newtheorem{remark}[theorem]{Remark}
\newtheorem{example}[theorem]{Example}
\newtheorem{assumption}{Assumption}[section]

\newtheorem{problem}{Problem}[section]

\def\cF{{\mathcal{F}}}

\def\cL{{\mathcal{L}}}
\def\cK{{\mathcal{K}}}
\def\bx{{\bm{x}}}

\def\bz{{\bm{z}}}
\def\bd{{\bm{d}}}
\def\R{{\mathbb{R}}}

\newcommand{\inProd}[2]{\langle #1 , #2 \rangle }

\DeclareMathOperator{\ri}{ri}
\DeclareMathOperator{\epi}{epi}
\DeclareMathOperator{\spanVec}{span}
\numberwithin{equation}{section}
\def\iEBR{{\rm Ind}_{\rm EBR}}

\title{\sf Error bounds for perspective cones of a class of nonnegative Legendre functions}
\author{
Xiaozhou Wang\thanks{South China Normal University, Guangzhou, People's Republic of China ({\tt xiaozhou.wang@m.scnu.edu.cn}). The research of this author was partially supported  by the National Natural Science Foundation of China (12501427).}
\and
Bruno F. Louren\c{c}o\thanks{Department of Fundamental Statistical Mathematics, Institute of Statistical Mathematics, Tokyo, Japan ({\tt bruno@ism.ac.jp}). This author was supported partly by the JSPS Grant-in-Aid for Early-Career Scientists  23K16844.}
\and
Ting Kei Pong\thanks{The Hong Kong Polytechnic University, Hong Kong, People's Republic of China ({\tt tk.pong@polyu.edu.hk}). The research of this author was partially supported by the Hong Kong Research Grants Council PolyU 15300122 and the PolyU internal grant 4-ZZPJ.}
}
\date{\today}

\begin{document}
  \maketitle


\begin{abstract}
Error bounds play a central role in the study of conic optimization problems, including the analysis of convergence rates for numerous algorithms.
Curiously, those error bounds are often H\"olderian with exponent $1/2$.
In this paper,  we try to explain the prevalence of the $1/2$ exponent by investigating {generic properties} of error bounds for conic feasibility problems where the underlying cone is a {perspective cone}  constructed from a nonnegative Legendre function on $\R$. 
Our analysis relies on the facial reduction technique and the computation of \emph{one-step facial residual functions} ($\mathds{1}$-FRFs).
Specifically, under appropriate assumptions on the Legendre function, we show that  $\mathds{1}$-FRFs can be taken to be H\"olderian of exponent  $1/2$ {almost everywhere}  with respect to the two-dimensional Hausdorff measure.
This enables us to further establish that having a {uniform H\"olderian error bound} with exponent $1/2$ is a generic property for a class of feasibility problems involving these cones. 
  \end{abstract}


\section{Introduction}
Consider the following conic feasibility problem
\begin{align}\tag{{\bf Feas}}\label{Feas}
{\rm find}\quad  {\bm z}\in (\mathcal{L}+{\bm a}) \cap \mathcal{K},
\end{align}
where $\mathcal{E}$ is a finite-dimensional real Euclidean space, $\mathcal{K}\subseteq \mathcal{E}$ is a closed convex cone,  $\mathcal{L}\subseteq \mathcal{E}$ is a subspace and ${\bm a}\in\mathcal{E}$ is a vector.

We recall that the problem of minimizing/maximizing a linear function over a set of the form $ (\mathcal{L}+{\bm a}) \cap \mathcal{K}$ is called \emph{conic linear programming} (CLP). CLP is a cornerstone of convex optimization and it includes special cases such as linear programming, semidefinite programming and other important classes of problems that are relevant to modern applications.

The jump from linear programming to CLP over an arbitrary convex cone is significant. One important source of difficulty can be traced back to the analysis of feasibility problems such as \eqref{Feas}.
When solving a CLP numerically, typically one makes use of an iterative method that generates points ${\bm z}^k$ that ``approximately'' solve \eqref{Feas}.
And herein lies a critical point: typically, we can only  estimate the individual distances between ${\bm z}^k$  and the sets $\mathcal{L}+{\bm a}$ and $\mathcal{K}$.
The actual distance between ${\bm z}^k$ and the intersection $(\mathcal{L}+{\bm a})\cap \mathcal{K}$ is far harder to be computed directly\footnote{For example, projecting some point onto a nonnegative orthant $\R^n_+$ entails ``zeroing'' the negative entries.
The projection operator onto an affine space $\mathcal{L}+{\bm a}$ also has a closed form expression.
	However, the projection onto $(\mathcal{L}+{\bm a})\cap \R^n_+ $ does not typically have a closed form.}.

Inequalities that upper bound the distance to an intersection of sets through the individual distances are known as \emph{error bounds} and are used
in a myriad of ways in optimization theory.
A basic application of error bound theory is in the study of  convergence rates of algorithms in optimization.
This is quite natural as error bounds provide a glue that links the analysis of the individual components of a problem (i.e., each individual set) to the problem taken as a whole (i.e., the intersection of sets).

The error bound study of \eqref{Feas} for the case where $\mathcal{K}$ is polyhedral is relatively well-understood.
In this case, it can be shown that if a sequence $\{{\bm z}^k\}$ approaches $\mathcal{L}+{\bm a}$ and $\mathcal{K}$ simultaneously, then the distance from ${\bm z}^k$ to $(\mathcal{L}+{\bm a})\cap \mathcal{K}$ goes to zero at least with the same order as the individual distances
from ${\bm z}^k$ to $\mathcal{L}+{\bm a}$ and $\mathcal{K}$ go to zero.
 That is because the celebrated Hoffman's lemma \cite{Hoffman-1952} implies the existence of a constant $\kappa > 0$ such that
 \begin{align}\label{hoffman}
 	{\rm d}({\bm x},(\mathcal{L}+{\bm a})\cap \mathcal{K})\leq \kappa \max\{{\rm d}({\bm x},\mathcal{L} + {\bm a}), {\rm d}({\bm x},\mathcal{K}) \}\quad \forall\, {\bm x \in \mathcal{E}}.
 \end{align}
In particular, this type of error bound is true for the feasible region of linear programs, i.e.,  for the case where $\mathcal{K}$ is a nonnegative orthant.
Unfortunately, when $\mathcal{K}$ is not polyhedral we cannot expect an inequality such as \eqref{hoffman} to hold and we need more general types of error bounds as we will recall later in Section~\ref{section-pre}.

Here, we focus on error bounds for \eqref{Feas} that can be obtained without extra assumptions.
Following a pioneering work by Sturm for the semidefinite programming case \cite{Sturm-2000}, a new framework was developed in a series of papers \cite{Lindstrom-Lourenco-Pong-2023,Lindstrom-Lourenco-Pong-2025,Lourencco-2021} that made it possible to compute error bounds for \eqref{Feas} for many different types of $\mathcal{K}$.
An important part of this framework is the computation of so-called \emph{one-step facial residual functions} ($\mathds{1}$-FRFs) for the faces of the cone $\mathcal{K}$.

After computing $\mathds{1}$-FRFs for symmetric cones \cite{Lourencco-2021}, exponential cones \cite{Lindstrom-Lourenco-Pong-2023}, $p$-cones \cite{Lindstrom-Lourenco-Pong-2025}, power cones \cite{Lin-et-2024} and log-determinant cones \cite{Lin-Lindstrom-et-2024}, a pattern has emerged.
Discarding certain trivial cases, in all examples computed so far, the dominant term in $\mathds{1}$-FRFs is ``almost always'' a power term of exponent $1/2$, see Table~\ref{table-eb}.
This observation leads naturally to the following question.


\begin{quote}
Is the prevalence of $\mathds{1}$-FRFs with exponent $1/2$ a general phenomenon? How to quantify this in more precise terms?
\end{quote}
This paper is an attempt to answer these questions.
 Our main result is that for three-dimensional cones constructed from certain perspective functions, $\mathds{1}$-FRFs with exponent $1/2$ are valid almost-everywhere with respect to the two-dimensional Hausdorff measure, see Theorem \ref{the-almost-1frf}.

 The importance of having $\mathds{1}$-FRFs with exponent $1/2$ (or larger) is that this can be leveraged to describe the cases in which a H\"olderian error bound of exponent $1/2$ holds for the problem \eqref{Feas}.
 This, by its turn, can be further connected to fast convergence of certain algorithms.
 This connection can either be direct or indirect.
 For examples of the latter, see \cite{BNPS17,BLT17,Liu-Lourenco-2024}, where error bound conditions between sets are directly translated to convergence rates of algorithms.
 Examples of more indirect connections can be seen in \cite{Lindstrom-Lourenco-Pong-2025,YLP22}.
In \cite{YLP22}, it is shown how to compute the Kurdyka-{\L}ojasiewicz (KL) exponent of a convex function by looking at the error bounds of its conic representation.
In this context, an error bound of exponent $1/2$ can often be translated to a KL exponent of $1/2$, which, by its turn, often means linear convergence of certain algorithms, see, for example, \cite{Attouch-Bolte-2009,Attouch-Bolte-Redont-Soubeyran-2010,LP18}.
An example of this approach can also be found in \cite[Section~5.1]{Lindstrom-Lourenco-Pong-2025}, where the KL exponent of a linear regression problem with $p$-norm regularization is obtained via the error bounds of its conic representation.


 \subsection{Related works}

%
%
As discussed previously,  if  $\mathcal{K}$ is polyhedral, then \eqref{Feas} satisfies an error bound as
in \eqref{hoffman}, as a consequence of Hoffman's pioneering work \cite{Hoffman-1952}.
Even if $\mathcal{K}$ is not polyhedral, if it happens that the affine space $\mathcal{L} + {\bm a}$ passes through the relative interior of $\mathcal{K}$, then we also have a similar error bound that holds locally over each bounded set, see \cite[Theorem~7]{Bauschke-Borwein-1996}.
In our terminology, this corresponds to a Lipschitzian error bound as in Definition~\ref{def-error-bound}.

The requirement that $\mathcal{L} + {\bm a}$ passes through the relative interior of $\mathcal{K}$ is an example of \emph{constraint qualification} (CQ), which we cannot assume to hold always.
Making use of the  facial reduction technique \cite{Borwein-Wolkowicz-1981, Pataki-2013, Waki-Muramatsu-2013}, several works examined the error bound properties of \eqref{Feas} with non-polyhedral $\mathcal{K}$ without assuming any CQs; see \cite{Lin-et-2024, Lin-Lindstrom-et-2024, Lindstrom-Lourenco-Pong-2023, Lindstrom-Lourenco-Pong-2025, Lourencco-2021}.
In these works, a key technical component is the computation of the $\mathds{1}$-FRFs for a given cone.
Once all the $\mathds{1}$-FRFs are computed, they are composed in a special manner in order to obtain an error bound inequality for \eqref{Feas}, see 
Theorem~\ref{the-error-bound}. 
In all examples discussed in \cite{Lin-et-2024, Lin-Lindstrom-et-2024, Lindstrom-Lourenco-Pong-2023, Lindstrom-Lourenco-Pong-2025, Lourencco-2021},
the most common type of  $\mathds{1}$-FRF is the H\"{o}lderian $\mathds{1}$-FRF  of exponent $1/2$ (see Definition \ref{def-holder-1frf}).

An important case  is when $\cK$ is a cone of real symmetric positive semidefinite matrices.
Then, the corresponding \eqref{Feas} becomes the so-called
\emph{semidefinite feasibility problem}.
In this case, \eqref{Feas} satisfies a H\"{o}lderian error bound, as shown by Sturm \cite{Sturm-2000}.
The exponent that appears in the H\"{o}lderian error bound is controlled by the so-called \emph{singularity degree}.
When the singularity degree is $1$, the exponent becomes $1/2$.
Related to that,  Drusvyatskiy, Li and Wolkowicz \cite[Proposition 3.2]{Drusvyatskiy-Li-Wolkowicz-2017} established that the singularity degree is at most one for almost all (in the sense of Lebesgue measure) semidefinite feasibility problems.
See also \cite{Li-2010, Li-2013} for the related topic of error bounds for convex polynomials.



In addition, other authors have explored generic properties (in the sense of Hausdorff measure) of conic linear programs, see \cite{Silva-Tuncel-2019, Pataki-Levent-2001}.
For more details on various generic properties of optimization problems, see also \cite{Bolte-Daniilidis-Lewis-2011, Drusvyatskiy-Lewis-2011, Dur-Jargalsaikhan-Still-2017, Lee-Pham-2017, Schurr-Tits-Oleary-2007}.

\subsection{Main results}\label{sec:int_main}
Let
$f : \R \to \R_+ \cup\{\infty\}$ be a nonnegative, proper, closed and convex function
and denote by $f^{\pi}:\R \times \R\to \R\cup\{\infty\}$  the closed perspective function of $f$, which is  defined as
\begin{equation}\label{eq:pers}
\begin{aligned}
f^{\pi}(x,t) \coloneqq \begin{cases}
  t f(x/t) & {\rm if}\ t > 0,\\
  f_\infty(x) & {\rm if}\ t = 0,\\
\infty & {\rm if}\ t < 0,
\end{cases}
\end{aligned}
\end{equation}
with $f_\infty$ being the recession function of $f$ (see \eqref{def-f-infty} for the definition).
Then, ${\rm epi}\, f^\pi$ is a closed convex cone
and as we will see in Section~\ref{sec:pers-cone}, many useful three-dimensional cones can be written in this form.

Cones arising as the epigraph of perspective functions will be referred as
\emph{perspective cones}.
In this paper, under a mild assumption (see Assumption \ref{assu-blanket}), we show the following two generic properties.
\begin{itemize}
	\item   The $\mathds{1}$-FRF for the perspective cone ${\rm epi}\, f^\pi$ is H\"{o}lderian of exponent $1/2$ (see Definition \ref{def-holder-1frf}) almost everywhere with respect to the two-dimensional Hausdorff measure; see Theorem \ref{the-almost-1frf}.
	\item Using the generic property for $\mathds{1}$-FRFs concerning the perspective cone ${\rm epi}\, f^\pi$, we further prove a generic property concerning the error bound property for \eqref{Feas} with $\mathcal{K}= {\rm epi}\, f^\pi$ in Corollary \ref{coro-almost-eb}.
More specifically, for a special class of \eqref{Feas} with $\mathcal{K}= {\rm epi}\, f^\pi$ and $\mathcal{L}$ being a two-dimensional subspace, almost every (with respect to the two-dimensional Hausdorff measure) feasibility problem $\{ {\rm epi}\, f^\pi, \mathcal{L}+{\bm a}\}$ satisfies a uniform H\"{o}lderian error bound with exponent $1/2$ (see Definition \ref{def-error-bound}).
\end{itemize}

The paper is organized as follows. In Section \ref{section-pre}, we present preliminaries and the assumptions on the univariate functions we analyze. Section \ref{section-face} is devoted to characterizing the facial structures of ${\rm epi}\, f^\pi$, which include exposed faces and non-exposed faces of ${\rm epi}\, f^\pi$.
In Section \ref{section-frf}, we show that many $\mathds{1}$-FRFs for ${\rm epi}\, f^\pi$ can be taken to be H\"{o}lderian of exponent $1/2$.
In Section \ref{section-measure}, we further show that both  H\"{o}lderian $\mathds{1}$-FRFs of exponent $1/2$ for ${\rm epi}\, f^\pi$ and uniform H\"{o}lderian error bounds with exponent $1/2$ for a class of feasibility problems 
are generic with respect to the two-dimensional Hausdorff measure.

\section{Preliminaries}\label{section-pre}
In this section we review some notions in convex analysis; more details can be found in \cite{Rockafellar-1970}.
We assume that the Euclidean spaces ${\cal E}$ involved are equipped with some inner product $\inProd{\cdot}{\cdot}$ and we denote the corresponding
induced norm by $\|\cdot\|$; we will use the standard inner product on $\R^3$ in Sections~\ref{section-face}, \ref{section-frf} and \ref{section-measure}.
We use $B(\eta)$ to denote a ball with center at the origin and radius $\eta\geq 0$.

Throughout the paper, vectors will be boldfaced while scalars use normal typeface. For example, for ${\bm p}\in \R^3$, we write ${\bm p}=(p_1,p_2,p_3)$ with $p_1$, $p_2$ and $p_3\in \R$.
For a nonempty convex set $C\subseteq {\cal E}$, we use ${\rm int}(C)$, ${\rm ri}(C)$, ${\rm bdry}(C)$, ${\rm cl}(C)$ and $\spanVec(C)$ to represent the interior, relative interior, boundary, closure and linear span of $C$, respectively. The support function of the set $C$ is
\[
\sigma_C({\bm x})\coloneqq\sup\{ \langle {\bm x}, {\bm y}\rangle : {\bm y}\in C  \}.
\]
The \emph{normal cone of $C$ at ${\bm x} \in C$} is denoted by $\mathcal{N}_{C}({\bm x})$ and
satisfies $\mathcal{N}_{C}({\bm x}) = \{{\bm y}  : \sigma_C(\bm y) = \inProd{{\bm x}}{\bm{y}} \}$.
The distance of ${\bm x}$ to the set $C$ is
\[
{\rm d}({\bm x},C)\coloneqq\inf\{\|{\bm x}-{\bm y}\| : {\bm y}\in C\}.
\]
In addition, if $C$ is closed, we use  ${\rm P}_C({\bm x})$ to denote the projection of a vector ${\bm x}$ onto $C$, so that ${\rm d}({\bm x},C) = \|{\bm x} -  {\rm P}_C({\bm x})\|$ holds.

A  function $g:{\cal E}\rightarrow (-\infty, \infty]$ is said to be \emph{proper} if its domain ${\rm dom }\, g\coloneqq \{{\bm x}\in{\cal E} : g({\bm{x}})<\infty \}$ is nonempty. A proper function is said to be \emph{closed} if it is lower semicontinuous.
We use ${\rm epi}\,g  \subseteq{\cal E}\times \R$ to denote the epigraph of $g$, i.e.,
${\rm epi}\,g\coloneqq\{({\bm x},r)\in {\cal E}\times \R: r\geq g({\bm x}) \}.$

Let $h:\R\rightarrow (-\infty, \infty]$ be a proper closed convex function.
The subdifferential $\partial h$ at $x\in \R$ is defined as
\[
\partial h(x) \coloneqq \{ \xi\in\R : h(y)\geq h(x)+\xi(y-x)~\,\,\, \forall\, y\in \R \}.
\]
The domain of $\partial h$ is defined
as ${\rm dom}\, \partial h = \{ x \in \R : \partial h(x) \neq \emptyset\}$.
The  conjugate function $h^*$ of $h$ is given by
\[
h^*(x^*) \coloneqq \sup\{ x x^*-h(x) : x\in \R  \}.
\]
We also recall the Fenchel-Young inequality \cite[pg.~105]{Rockafellar-1970}:
\begin{align}\label{young}
	h(x)+h^*(x^*)- xx^*\geq 0  \quad \forall\, x,  x^*\in \R;
\end{align}
moreover, it holds that
\begin{align}\label{young-1}
	h(x)+h^*(x^*)- x x^*=0 \Longleftrightarrow x^*\in \partial h(x)\Longleftrightarrow x \in \partial h^*(x^*),
\end{align}
see \cite[Theorem~23.5]{Rockafellar-1970}.
Let $y \in {\rm dom}\,h$ be fixed. The recession function of $h$ is defined as
\begin{align} \label{def-f-infty}
	h_\infty(d) \coloneqq \sup_{t> 0}\frac{h(y + td) - h(y)}{t}=\lim_{t \rightarrow \infty}\frac{h(y + td) - h(y)}{t} \ \ \ \ \forall\, d\in \R,
\end{align}
and it is known that the definition of $h_\infty$ is independent of the choice of $y$ in ${\rm dom}\,h$.
Support functions, recession functions and conjugates of $h$ are related as follows (see \cite[Theorem 2.5.4]{Auslender-Alfred-Teboulle-2006}):
\begin{align}\label{infty-sigma}
	(h^*)_\infty(x)=\sigma_{{\rm dom}\, h}(x)\quad {\rm and}\quad h_\infty(x)=\sigma_{{\rm dom}\, h^*}(x)\ \ \ \forall x\in \R.
\end{align}

The following definitions are important in this paper.
\begin{definition}[essential smoothness,  strict convexity and local strong convexity {\cite{Goebel-Rockafellar-2008}}]\label{def-smooth-conexity}
	Let $h:\R\rightarrow \R\cup\{\infty\}$ be a proper convex function.
	\begin{enumerate}[\rm (i)]
		\item We say that $h$ is \emph{essentially smooth} if ${\rm int}({\rm dom}\,h)\neq \emptyset$ and it is differentiable on ${\rm int}({\rm dom}\,h)$, and
		$|h'(x^k)| \rightarrow \infty$ whenever $\{x^k\}$ belongs to ${\rm int}({\rm dom}\,h)$ and converges to a point in ${\rm bdry}({\rm dom}\,h)$.\footnote{These functions are called essentially differentiable functions in \cite{Goebel-Rockafellar-2008}. Here, we follow \cite[Section~26]{Rockafellar-1970} and adopt the terminology ``essential smoothness".}
		
		\item We say that  $h$  is \emph{essentially strictly convex} if it is strictly convex on every convex subset of ${\rm dom}\, \partial h$.
		
		\item We say that $h$  is \emph{essentially locally strongly  convex} if it is strongly convex on every compact and convex  subset of ${\rm dom}\, \partial h$.
	\end{enumerate}
\end{definition}
Clearly, essential  local strong convexity implies essential  strict convexity.
A proper, closed and convex function is called a \emph{Legendre function} if  it is  essentially smooth and essentially strictly convex.

If $h$ is  a Legendre function (which implies, in view of \cite[Theorem~26.3]{Rockafellar-1970}, that $h^*$ is also Legendre), then it follows from \eqref{young} that for any $x^*\in  {\rm bdry}({\rm dom}\, h^*)$,
\begin{align}\label{young-2}
h(x)+h^*(x^*)- x x^*> 0  \quad \forall\, x\in \R,
\end{align}
where the strict inequality follows from  \eqref{young-1} and the  fact that $\partial h^*(x^*)=\emptyset$ (which follows from Theorem 23.4 and Theorem 26.1 in \cite{Rockafellar-1970}).

\subsection{Faces and convex geometry}
Let $\mathcal{K}\subseteq \R^n$ be a nonempty closed convex cone. Its dual cone is $\mathcal{K}^*\coloneqq\{{\bm x}\in \R^n: \langle {\bm x},{\bm y} \rangle\geq 0 \,\,\,\forall \,{\bm y}\in \mathcal{K}\}$.
A nonempty convex cone $\mathcal{F}\subseteq \mathcal{K}$ is said to be a \emph{face} of $\mathcal{K}$  (denoted by $\mathcal{F}\unlhd \mathcal{K}$) if
\[
\forall \, {\bm x},{\bm y}\in \mathcal{K},  {\bm x}+{\bm y}\in \mathcal{F} \Longrightarrow {\bm x},{\bm y} \in \mathcal{F}.
\]
A face $\mathcal{F}\subseteq \mathcal{K}$ is \emph{nontrivial} if $\mathcal{F}\neq \mathcal{K}$ and $\mathcal{F} \neq \mathcal{K}  \cap -\mathcal{K}$.
A cone $\mathcal{K}$ is said to be \emph{pointed} if $ \mathcal{K}\cap -\mathcal{K}=\{{\bm 0}\}$.
If ${\rm int}{(\mathcal{K})} \neq \emptyset $ then
$\mathcal{K}$ is said to be \emph{solid} or \emph{full-dimensional}.

$\mathcal{F}$ is called an \emph{exposed face} of $\mathcal{K}$
if there exists  ${\bm z}\in \mathcal{K}^*$ such that $\mathcal{F}=\mathcal{K}\cap \{{\bm z}\}^{\bot}$. In addition,  according to \cite[Lemma 2.2]{Lourencco-Muramatsu-Tsuchiya-2021}, it is known that
\begin{align}\label{pre-1}
	{\bm z}\in {\rm ri}(\mathcal{K}^*) \Longrightarrow \mathcal{K}\cap \{{\bm z}\}^\bot= \mathcal{K}  \cap -\mathcal{K}.
\end{align}

\subsection{Perspective functions and cones}\label{sec:pers-cone}
Throughout this paper we will analyze 
functions satisfying the following assumption. 
\begin{assumption}\label{assu-blanket}
	The function $f : \R \to \R_+ \cup\{\infty\}$ is nonnegative, proper, closed and convex.
	Furthermore, $f$ is essentially smooth,\footnote{This implies that ${\rm dom}\, \partial f = {\rm int}({\rm dom}\, f)$; see \cite[Theorem 26.1]{Rockafellar-1970}.} and there exists a finite set $W\subset {\rm int}({\rm dom}\,f)$ such that
	$f$ is strongly convex on every compact and convex subset of ${\rm int}({\rm dom}\,f)\setminus W$ and $f'$ is locally Lipschitz continuous on every convex subset of ${\rm int}({\rm dom}\,f)\setminus W$.
\end{assumption}
Under Assumption~\ref{assu-blanket}, $f$ is a Legendre function,  see Proposition \ref{prop-duality}.
In particular, if $W=\emptyset$, then $f$ must also be  essentially locally strongly convex with locally Lipschitz derivative.
We mention in passing that it was shown in \cite[Corollary 4.4]{Goebel-Rockafellar-2008} that  a proper, closed and convex function is essentially smooth with locally Lipschitz gradient and essentially locally strongly convex if and only if its conjugate function is so.


In this paper, we focus on the feasibility problem \eqref{Feas} with $\mathcal{K} = {\rm epi}\, f^\pi$, where $f$ satisfies Assumption~\ref{assu-blanket} and $f^\pi$ denotes its perspective function as in \eqref{eq:pers}:
\begin{align}\label{prob-fea}
	{\rm Find}\quad  {\bm z}\in (\mathcal{L}+{\bm a}) \cap {\rm epi}\,f^\pi.
\end{align}
In view of \cite[Page~67 and Theorem~8.2]{Rockafellar-1970},  ${\rm epi}\,f^\pi$  can be written as follows:
\begin{align}
	{\rm epi}\,f^\pi &= \overline{\bigcup_{t>0}\{(x,t,r):\; (x,r)\in t\cdot {\rm epi}\,f\}} \notag\\
	& = \underbrace{\{(x,0,s):\; (x,s)\in {\rm epi}\,f_\infty\}}_{=:\mathcal{K}^\theta}\cup  \underbrace{\{(x,t,r):\; (x,r)\in t\cdot {\rm epi}\,f, t>0\}}_{\eqqcolon \mathcal{K}^t}. \label{per-cone}
\end{align}
We also observe that $\epi f^\pi$ is pointed and solid, see Proposition~\ref{prop-pointed}.

Many important cones can be written as the epigraph of perspective functions, as illustrated in the following examples.
\begin{example}[the exponential cone]\label{example-exp}
	Let $f(x) =\exp(x)$.
	Then, $f$ satisfies Assumption \ref{assu-blanket} with $W=\emptyset$. Moreover, we have
\begin{align*}
f^*(x)=
\begin{cases}
x \ln x-x & \mbox{if $x>0$}, \\
0 & \hbox{if $x=0$}, \\
\infty & \mbox{if $x<0$},
\end{cases}\quad \mbox{and}\quad
f_\infty(x)=
\begin{cases}
0  & \mbox{if $x\leq 0$}, \\
\infty & \mbox{if $x>0$}.
\end{cases}
\end{align*}
The corresponding perspective cone is the exponential cone as follows (see \cite{Lindstrom-Lourenco-Pong-2023}):
\begin{align}\label{exp-cone}
{\rm epi}\,f^\pi=\mathcal{K}^t\cup \mathcal{K}^\theta=\{ (x,t,r) : r \geq t \exp(x/t), t>0\}\cup \{(x,0,s) : x\leq 0, s\geq 0\}\subset \R^3.
\end{align}
\end{example}

\begin{example}[the exp-exponential cone]\label{example-exp-exp}
	Let $f(x)=\exp(\exp(x))$.
	Then, $f$ satisfies Assumption \ref{assu-blanket} with $W=\emptyset$.
	We have $f_\infty(d)=0$ if $d\leq 0$ and  $f_\infty(d)=\infty$ if $d>0.$
	The perspective cone  is
	\[
	{\rm epi}\, f^\pi=\mathcal{K}^\theta\cup \mathcal{K}^t=\{(x,0,r): x\leq 0, r\geq 0\}\cup \{(x,t,r) : r\geq t \exp(\exp(x/t)) : t>0\}.
	\]
\end{example}

\begin{example}[the power cone and beyond]\label{example-power}
	This example is taken from \cite[Theorem 2]{Hildebrand-2014}.
	Let
	\begin{align*}
		f(x) =
		\begin{cases}
			\alpha^{-p}|x|^p & \mbox{if $x< 0$}, \\
			x^p & \mbox{if $x\geq 0$},
		\end{cases}
	\end{align*}
	where $p\in [2,\infty)$ and $\alpha\in(0,1]$. Direct computation shows that
	\begin{align*}
		&f'(x)=
		\begin{cases}
			-\alpha^{-p}p|x|^{p-1} &  \mbox{if $x<0$}, \\
			p x^{p-1}  &  \mbox{if $x\geq 0$},
		\end{cases} \quad
		f{''}(x)
		=
		\begin{cases}
			\alpha^{-p}p(p-1)|x|^{p-2}  & \mbox{if $x<0$}, \\
			p(p-1)x^{p-2}  & \hbox{if $x> 0$},
		\end{cases}\\
		&f_\infty(d)=
		\begin{cases}
			0  & \mbox{if $d=0$}, \\
			\infty & \hbox{otherwise}.
		\end{cases}
	\end{align*}
	For $p=2$, $f(x)$ satisfies Assumption \ref{assu-blanket} with $W=\emptyset$, and $f$ is actually strongly convex  with Lipschitz derivative. For $p\in (2, \infty)$, $f$ satisfies Assumption \ref{assu-blanket} with $W=\{0\}$.
	
	From \eqref{per-cone},  the perspective cone is as follows
	\begin{align*}
		{\rm epi}\, f^\pi=\mathcal{K}^\theta\cup \mathcal{K}^t=\{(x,t,r): -\alpha r^{\frac{1}{p}}t^{\frac{1}{q}} \leq x\leq r^{\frac{1}{p}}t^{\frac{1}{q}}, r\geq 0, t\geq 0\},
	\end{align*}
	where   $1/p+1/q=1$.
	When $\alpha=1$, then ${\rm epi}\, f^\pi$ is exactly the three-dimensional power cone \cite{Lin-et-2024}:
	\begin{align}\label{power-cone}
	\mathcal{P}_{1,2}^{\frac{1}{p},\frac{1}{q}}\coloneqq \{(x,t,r): r^{\frac{1}{p}} t^{\frac{1}{q}}\geq |x|, r\geq 0,t\geq 0\}.
	\end{align}
\end{example}

\begin{example}[the second-order cone]\label{example-soc}
	The second-order cone $\mathcal{K}^{3}_2\coloneqq \{(x,t,r) : r\geq \sqrt{t^2+x^2}\}$ is linearly isomorphic to the power cone of $\mathcal{P}_{1,2}^{\frac{1}{2},\frac{1}{2}}=\{(x,t,r): r^{\frac{1}{2}}t^{\frac{1}{2}}\geq |x|, r\geq 0, t\geq 0\}$ via the bijective linear mapping $\mathcal{T}$ given by $\mathcal{T}(x,t,r)\coloneqq (2x,r-t,r+t)$; that is $\mathcal{T}\mathcal{P}_{1,2}^{\frac{1}{2},\frac{1}{2}}=\mathcal{K}^3_2$.
\end{example}

\begin{remark}[How restrictive is Assumption~\ref{assu-blanket}?]
Examples \ref{example-exp}-\ref{example-soc} cover the exponential cone, the power cone, the second-order cones and beyond.
These cones are useful in modern applications and they are implemented in several codes for conic programming such as  MOSEK \cite{mosek-2025}, DDS \cite{Karimi-Tuncel-2024}, Alfonso \cite{Papp-Yildiz-2022}, Hypatia \cite{Coey-Kapelevich-Vielma-2022} and Clarabel \cite{Goulart-Chen-2024}.
This suggests that
 Assumption \ref{assu-blanket}, although special, captures many important perspective cones.

 We also mention in passing that Hildebrand showed in \cite[Theorem 2]{Hildebrand-2014} that every three-dimensional pointed, solid, and closed convex cone with automorphism group of dimension at least $2$ is linearly isomorphic to exactly one of the five cones below :
\begin{enumerate}[1.]
\item the exponential cone in Example \ref{example-exp};
\item the nonnegative orthant $\R^3_+$;
\item the perspective cone  $\{(x,t,r): r^{\frac{1}{p}} t^{\frac{1}{q}}\geq x, r\geq 0,t\geq 0\}$  with $1/p+ 1/q=1$, which arises from the mapping that $x\mapsto x^p$ if $x\geq 0$ and $x\mapsto 0$ otherwise, where $p\in[2,\infty)$;
\item the perspective cone in Example \ref{example-power};
\item the perspective cone $\{(x,t,r): r^{\frac{1}{p}} t^{\frac{1}{q}}\geq x\geq 0, r\geq 0,t\geq 0\}$  with $1/p+ 1/q=1$, which arises from the mapping that $x\mapsto x^p$ if $x\geq 0$ and $x\mapsto \infty$ otherwise, where $p\in[2,\infty)$.
\end{enumerate}
The conclusion is that half of the isomorphism classes of nonpolyhedral cones with automorphism group of dimension at least $2$ are covered by cones  as in \eqref{per-cone} with $f$ satisfying Assumption~\ref{assu-blanket}. 
Here we recall that for $\R^3_+$, the sole polyhedral cone in the list above, its error bounds follow from Hoffman's lemma \cite{Hoffman-1952}.
We also recall that linearly isomorphic cones have the same facial residual functions up to positive rescaling in view of \cite[Proposition 17]{Lourencco-2021}.

\end{remark}

\subsection{Error bounds}

The definitions of Lipschitzian and H\"{o}lderian error bounds are as follows.
\begin{definition}[Lipschitzian and H\"{o}lderian error bounds]\label{def-error-bound}
Let $C_1$, $C_2 \subseteq \R^n$ be closed convex sets with $C_1\cap C_2\neq \emptyset$.
We say that $\{C_1, C_2\}$ satisfies a \emph{uniform H\"{o}lderian error bound} with exponent $\gamma\in (0,1]$ if for every bounded set $B$ there exists  constant $\kappa_B>0$  such that
\begin{align*}
{\rm d}({\bm x},C_1 \cap C_2)\leq \kappa_B \max\{{\rm d}({\bm x},C_1), {\rm d}({\bm x},C_2) \}^{\gamma}\quad \forall\, {\bm x}\in B.
\end{align*}
If $\gamma= 1$, then the error bound is said to be \emph{Lipschitzian}.
\end{definition}

In order to obtain error bounds for \eqref{prob-fea}, we will follow the approach described in \cite{Lindstrom-Lourenco-Pong-2023,Lindstrom-Lourenco-Pong-2025,Lourencco-2021}, see Appendix~\ref{section-appedix-a1} for more details.
The most technical part of applying the theory described in those papers is computing the so-called \emph{one-step facial residual function ($\mathds{1}$-FRF)}.
\begin{definition}[One-step facial residual function ($\mathds{1}$-FRF)]\label{def-1frf}
Let $\mathcal{K}$ be a closed convex cone and ${\bm z}\in \mathcal{K}^*$.
A function $\psi_{\mathcal{K},{\bm z}}: \R_+ \times \R_+ \to \R_+$ is called a one-step facial residual function ($\mathds{1}$-FRF) for $\mathcal{K}$ and ${\bm z}$
 if it satisfies the following two conditions:
\begin{enumerate}[\rm (i)]
\item $\psi_{\mathcal{K},{\bm z}}$  is nonnegative, nondecreasing in each argument and $\psi_{\mathcal{K},{\bm z}}(0,t)=0$ for all $t\in\R_+$;
\item for each ${\bm x}\in {\rm span}(\mathcal{K})$ and any $\epsilon\ge 0$,
\[
{\rm d}({\bm x},\mathcal{K})\leq \epsilon,\,   \langle {\bm x}, {\bm z} \rangle \leq \epsilon \Longrightarrow {\rm d}({\bm x},\mathcal{K}\cap \{{\bm z}\}^\bot)\leq \psi_{\mathcal{K},{\bm z}}(\epsilon, \|{\bm x}\|). \footnote{Note that item (ii) can be equivalently written as
\[
 {\rm d}({\bm x},\mathcal{K}\cap \{{\bm z}\}^\bot)\leq \psi_{\mathcal{K},{\bm z}}(\max\{{\rm d}({\bm x},\mathcal{K}), \langle {\bm x}, {\bm z}\rangle  \}, \|{\bm x}\|)
 \quad \forall\, x\in {\rm span}(\mathcal{K}).
\]}
\]
\end{enumerate}
\end{definition}
The $\mathds{1}$-FRF is a special case of the FRF described in  \cite{Lourencco-2021}
and, intuitively, expresses an error bound between a cone $\cK$ and one of its supporting hyperplanes.
It can be shown that FRFs always exist, e.g., \cite[Section~3.2]{Lourencco-2021}.
Once all the $\mathds{1}$-FRFs for a cone and its subfaces are determined, by making use of a facial reduction algorithm \cite{Borwein-Wolkowicz-1981,Lourenco-Muramatsu-Tsuchiya-2018,Pataki-2013,Waki-Muramatsu-2013} it is possible to obtain error bounds for \eqref{prob-fea}, see Theorem~\ref{the-error-bound}.

In this paper, we will focus on $\mathds{1}$-FRFs having the following format. 
\begin{definition}[H\"olderian $\mathds{1}$-FRFs]\label{def-holder-1frf}
Let $\mathcal{K}$ be a closed convex cone and ${\bm z}\in \mathcal{K}^*$. We say that the $\mathds{1}$-FRF for $\mathcal{K}$ and ${\bm z}$ is \emph{H\"olderian of exponent $\alpha\in(0,1]$} if there exist nonnegative nondecreasing functions $\rho_1, \rho_2:\R_+\rightarrow \R_+$ such that the $\mathds{1}$-FRF, denoted by $\psi_{\mathcal{K},{\bm z}}:\R_+\times \R_+ \rightarrow \R_+$, is given by
$\psi_{\mathcal{K},{\bm z}}(\epsilon,t) = \rho_1(t)\epsilon+\rho_2(t)\epsilon^{\alpha}$.
\end{definition}
\begin{remark}
	If $\mathcal{K}$ and ${\bm z}$ admit a H\"{o}lderian $\mathds{1}$-FRF of exponent $\alpha$, then
they also admit H\"{o}lderian $\mathds{1}$-FRF of exponent $\beta$ for any $\beta\in (0,\alpha]$.

If   $\mathcal{K}$ in Definition \ref{def-holder-1frf} is pointed, then for any  ${\bm z}\in {\rm ri}(\mathcal{K}^*)$, the corresponding $\mathds{1}$-FRF can be taken to be H\"olderian of exponent $1$; see \cite[Lemma 2.2]{Lindstrom-Lourenco-Pong-2025}.  In view of this fact and that ${\rm epi}\, f^\pi$ is pointed thanks to  Proposition \ref{prop-pointed}, we only need to study $\mathds{1}$-FRFs associated with those ${\bm z}\in {\rm bdry}(({\rm epi}\, f^\pi)^*)$.
\end{remark}

Formulae for $\mathds{1}$-FRFs  are known for symmetric cones \cite[Theorem~35]{Lourencco-2021},
exponential cones \cite{Lindstrom-Lourenco-Pong-2023},
$p$-cones \cite{Lindstrom-Lourenco-Pong-2025}, generalized power cones \cite{Lin-et-2024} and log-determinant cones \cite{Lin-Lindstrom-et-2024}.
See also Table~\ref{table-eb} for
some examples in three dimensions.

Table~\ref{table-eb} illustrates the motivating phenomenon for this paper.
In all examples therein, the cases in which the known $\mathds{1}$-FRFs are \emph{not} dominated by a power factor of exponent $1/2$ seem to be confined to sets of smaller dimensions.
For example, in the power cone case, every
${\bm z}\in {\rm bdry}(\mathcal{K}^*)\setminus \{{\bm 0}\}$ with $z_x \neq 0$ leads to a H\"olderian $\mathds{1}$-FRF of exponent $1/2$.
This fails for the case that $z_x =  z_t = 0$ with $p>2$, but the set of such ${\bm z}$ has smaller Hausdorff dimension than the set of ${\bm z}$ affording exponent $1/2$.



\begin{table}[h]
	\centering
	\begin{tabular}{c|cccc}
		\hline
		\multirow{2}{*}{$\mathcal{K}$}
		& ${\bm z}=(z_x,z_t,z_r)$ & \multirow{2}{*}{$\mathds{1}$-FRFs for $\mathcal{K}$ and ${\bm z}$} & dominated by\\
		& ${\bm z}\in {\rm bdry}(\mathcal{K}^*)\setminus \{\bm 0\}$ & & the exponent 1/2 \\
		\hline
		\multirow{6}{*}{\makecell[c]{power cone \\ $p\in [2,\infty)$ \\ $1/p+1/q=1$ \\see \eqref{power-cone} }}
		&  \multirow{2}{*}{$z_x\neq 0$} &  $(\epsilon,t)\mapsto \rho_1(t) \epsilon+\rho_2(t) \epsilon^{\frac{1}{2}}$   &   \multirow{2}{*}{\ding{51}}\\
		&   &  \cite[Corollary 3.9(i)]{Lin-et-2024} & \\
		\cline{2-4}
		& $z_x=z_r= 0$; or &  $(\epsilon,t)\mapsto \rho_1(t) \epsilon+\rho_2(t) \epsilon^{1/q}$ &  \multirow{2}{*}{\ding{51}} \\
		& $z_x=z_t=0$, $p=2$  &   \cite[Corollary 3.9(ii)]{Lin-et-2024}  & \\
		\cline{2-4}
		&  \multirow{2}{*}{$z_x=z_t= 0$, $p>2$} &  $(\epsilon,t)\mapsto \rho_1(t) \epsilon+\rho_2(t) \epsilon^{1/p}$ &  \multirow{2}{*}{\ding{55}} \\
		&   &   \cite[Corollary 3.9(ii)]{Lin-et-2024}  & \\
		\hline
		\multirow{8}{*}{\makecell[c]{exponential \\ cone \\ see \eqref{exp-cone}}}
		& \multirow{2}{*}{$z_x<0$} & $(\epsilon,t)\mapsto \rho_1(t) \epsilon+\rho_2(t) \epsilon^{\frac{1}{2}}$ & \multirow{2}{*}{\ding{51}}\\
		& &  \cite[Corollary 4.7]{Lindstrom-Lourenco-Pong-2023} & \\
		\cline{2-4}
		& \multirow{2}{*}{$z_x=0$, $z_t>0$, $z_r>0$} & $(\epsilon,t)\mapsto \rho_1(t) \epsilon$ & \multirow{2}{*}{\ding{51}}\\
		& & \cite[Corollary 4.11(i)]{Lindstrom-Lourenco-Pong-2023}& \\
		\cline{2-4}
		& \multirow{2}{*}{$z_x=0$, $z_t=0$, $z_r>0$} &  log-type function & \multirow{2}{*}{\ding{55}}\\
		& &\cite[Corollary 4.11(ii)]{Lindstrom-Lourenco-Pong-2023} & \\
		\cline{2-4}
		& \multirow{2}{*}{$z_x=0$, $z_t>0$, $z_r=0$}  & entropy-type function & \multirow{2}{*}{\ding{51}}\\
		& &\cite[Corollaries 4.3 and 4.4]{Lindstrom-Lourenco-Pong-2023} & \\
		\hline
		\multirow{6}{*}{\makecell[c]{$p$-cone \\ $p\in (1,\infty)$}}
		& \multirow{2}{*}{$z_t\neq 0$, $z_x\neq 0$}  &  $(\epsilon,t)\mapsto \rho_1(t) \epsilon+\rho_2(t) \epsilon^{\frac{1}{2}}$ &  \multirow{2}{*}{\ding{51}}\\
        & & \cite[Corollary 4.1]{Lindstrom-Lourenco-Pong-2025} & \\
        \cline{2-4}
		& $z_t\neq 0$, $z_x= 0$, $p\in (1,2]$; or &  $(\epsilon,t)\mapsto \rho_1(t) \epsilon+\rho_2(t) \epsilon^{\frac{1}{p}}$  &  \multirow{2}{*}{\ding{51}} \\
		& $z_t= 0$, $z_x\neq 0$, $p\in (1, 2]$ & \cite[Corollary 4.1]{Lindstrom-Lourenco-Pong-2025}  &\\
		\cline{2-4}
		& $z_t\neq 0$, $z_x= 0$, $p> 2$; or  & $(\epsilon,t)\mapsto \rho_1(t) \epsilon+\rho_2(t) \epsilon^{\frac{1}{p}}$ & \multirow{2}{*}{\ding{55}}\\
		&  $z_t= 0$, $z_x\neq 0$, $p> 2$  & \cite[Corollary 4.1]{Lindstrom-Lourenco-Pong-2025}   &\\
		\hline
	\end{tabular}
	\caption{$\mathds{1}$-FRFs for exposed faces of some important cones in $\R^3$. Here, $\rho_1, \rho_2: \R_+\rightarrow \R_+$ are nonnegative nondecreasing functions. 
Note that  the $p$-cone is written as $\{ (v_r,v_x,v_t): v_r\geq \sqrt[p]{|v_x|^p+|v_t|^p}\}$ in \cite[(1.1)]{Lindstrom-Lourenco-Pong-2025}, while it is given by $\{ (v_x,v_t,v_r): v_r\geq \sqrt[p]{|v_x|^p+|v_t|^p}\}$ in the table.}\label{table-eb}
\end{table}

\section{Facial structure of epigraphical cones}\label{section-face}
In this section, we discuss the facial structure of epigraphical cones and their exposed faces.
We start with a description of the dual cone.

Let $(f^*)^{\pi}$ be the perspective function of the conjugate function of $f$. Similar to \eqref{per-cone}, its epigraph is given as follows:
\begin{align}\label{per-cone-conjugate}
{\rm epi}\,(f^*)^\pi= \{(x,0,s):\; (x,s)\in {\rm epi}\,(f^*)_\infty\}\cup \{(x,t,r):\; (x,r)\in t\cdot {\rm epi}\,f^*, t>0\}.
\end{align}
The following proposition characterizes the dual cone of ${\rm epi}\,f^\pi$ in terms of  $f^*$.

\begin{proposition}[the dual cone of ${\rm epi}\, f^\pi$]\label{prop-dual-cone}
Consider \eqref{per-cone}.
The dual cone 
satisfies
\begin{align}
({\rm epi}\,f^\pi)^*=\underbrace{\{(-x,s,0):  (x,s)\in {\rm epi}\,(f^*)_\infty\}}_{\eqqcolon  \mathcal{K}^{*\theta}}   \cup\underbrace{\{(-x,r,t):\  (x,r)\in t\cdot {\rm epi}\,f^*, t>0\}}_{\eqqcolon  \mathcal{K}^{*t}}. \label{dual-cone}
\end{align}
\end{proposition}
\begin{proof}
Recall from \eqref{per-cone} that $\mathcal{K}^t=\{(x,t,r):\; (x,r)\in t\cdot {\rm epi}\,f, t>0\}$. Then ${\cal K}^t\cup\{\bm 0\}$ is the convex cone  generated by $(u,1,v)$ with $(u,v)\in {\rm epi}\, f$.
According to  \cite[Theorem 8.2]{Rockafellar-1970}, we see that
\begin{align}\label{prop-dual-cone-1}
{\rm cl}({\mathcal{K}^t}\cup\{\bm 0\})={\mathcal{K}^t}\cup{\mathcal{K}^\theta}={\rm epi}\, f^\pi.
\end{align}
Similarly, we deduce from \eqref{per-cone-conjugate} that ${\rm epi}\, (f^*)^\pi = {\rm cl}(\{(x,t,r):\; (x,r)\in t\cdot {\rm epi}\,f^*, t>0\}\cup \{{\bm 0}\})$. In view of these observations, we see further that
\begin{align*}
(x^*,t^*,r^*)\in {\rm epi}\, (f^*)^\pi
\overset{\rm (a)}{\Longleftrightarrow} & (-x^*,r^*,t^*) \in (\mathcal{K}^t\cup\{{\bm 0}\})^* \\
\overset{\rm (b)}{\Longleftrightarrow} & (-x^*,r^*,t^*) \in ({\rm cl}(\mathcal{K}^t\cup\{{\bm 0}\}))^* \overset{\rm (c)}{=}({\rm epi}\, f^\pi)^*,
\end{align*}
where (a) follows from \cite[Theorem 14.4]{Rockafellar-1970},
(b) follows from the fact that $({\rm cl}(\mathcal{Q}))^*=\mathcal{Q}^*$ for any nonempty convex cone $\mathcal{Q}$,
and (c) follows from  \eqref{prop-dual-cone-1}.
The above display together with \eqref{per-cone-conjugate} implies  \eqref{dual-cone}.
\end{proof}

\subsection{Exposed  faces}\label{sec:exposed_faces}

In this subsection, we study the facial structure of ${\rm epi}\, f^\pi$ in \eqref{per-cone} and we recall that we are under the blanket Assumption~\ref{assu-blanket}.
By definition, each exposed face is of
the form ${\rm epi}\, f^{\pi} \cap \{{\bm z}^*\}^\bot$  for some ${\bm z}^*\in ({\rm epi}\, f^{\pi})^*$.
However, if ${\bm z}^*$ is zero or is in ${\rm int}(({\rm epi}\, f^{\pi})^*)$ (noting that $({\rm epi}\, f^{\pi})^*$ is solid and pointed thanks to Proposition \ref{prop-pointed}), then ${\bm z}^*$
exposes a trivial face, see \eqref{pre-1}.
Accordingly, in order to understand the nontrivial faces of ${\rm epi}\, f^\pi$ it is enough to focus on the nonzero ${\bm z}^*$ contained in
the boundary of $({\rm epi}\, f^{\pi})^*$.


 We first classify the nontrivial exposed faces that are exposed by ${\bm z}^*=(-x^*,r^*,t^*)\in {\rm bdry}(\mathcal{K}^{*t})\cap \mathcal{K}^{*t}$, where $t^*>0$ and $\mathcal{K}^{*t}$ is given in \eqref{dual-cone}.
\begin{proposition}\label{prop-e-face}
Let  ${\bm z}^*=(-x^*,r^*,t^*)\in  {\rm bdry}(\mathcal{K}^{*t})\cap \mathcal{K}^{*t}$ be such that ${\rm epi}\, f^{\pi} \cap \{{\bm z}^*\}^\bot$ is a nontrivial exposed face, where $t^*>0$ and $\mathcal{K}^{*t}$ is given in \eqref{dual-cone}.
\begin{enumerate}[\rm (i)]
\item If $\beta\coloneqq x^*/t^*\in {\rm int}({\rm dom}\, f^*)$ and $r^*=t^*f^*(\beta)$, then we have
\begin{align}\label{prop-e-face-1}
\mathcal{F}^e_\beta\coloneqq {\rm epi}\,f^\pi\cap \{{\bm z}^*\}^\bot=\{t(\bar{\beta},1, f(\bar{\beta})): t\geq 0\},
\end{align}
where $\bar{\beta} \coloneqq (f^*)'(\beta)$.
\item If $\beta\coloneqq x^*/t^*\in {\rm bdry}({\rm dom}\, f^*)\cap {\rm dom}\, f^*$ and $r^*\geq t^* f^*(\beta)$, then we have
\begin{align}\label{prop-e-face-2}
\mathcal{F}^e_{\rm bd}\coloneqq {\rm epi}\,f^\pi\cap \{{\bm z}^*\}^\bot=\mathcal{K}^\theta \cap \{{\bm z}^*\}^\bot=\{(x,0, f_\infty(x)): f_\infty(x)= \beta x \},
\end{align}
where $\mathcal{K}^\theta$ is defined in \eqref{per-cone}.
\end{enumerate}
\end{proposition}
\begin{proof}
We first discuss the conditions that a given
\begin{align}\label{prop-e-face-2-1}
(x,t,r)\in{\rm epi}\,f^\pi\cap \{{\bm z}^*\}^\bot,
\end{align}
must satisfy, where ${\bm z}^*\in  {\rm bdry}(\mathcal{K}^{*t})\cap \mathcal{K}^{*t}$.
Since $(x,t,r)\in  {\rm epi}\,f^\pi$ and ${\bm z}^*\in  {\rm bdry}(\mathcal{K}^{*t})\cap \mathcal{K}^{*t}$,
we see that
\begin{align}\label{prop-e-face-2-2}
  \begin{cases}
r\geq tf(x/t)& \mbox{if $t>0$},\\
r\geq f_\infty(x)& \mbox{if $t=0$},\\
\end{cases}\,\, \mbox{and}\,\,
  \begin{cases}
  r^*= t^*f^*(x^*/t^*) & \mbox{if $x^*/t^*\in {\rm int}({\rm dom}\,f^*)$},\\
  r^*\geq t^*f^*(x^*/t^*) & \mbox{if $x^*/t^*\in {\rm bdry}({\rm dom}\,f^*)\cap {\rm dom}\,f^*$}.
  \end{cases}
\end{align}
Consider two cases for $t$ as follows:
\begin{enumerate}[\rm (a)]
\item $t>0$;
\item $t=0$.
  \end{enumerate}

{\bf Case (a):}
From \eqref{prop-e-face-2-1},  the condition $(x,t,r)\in \{{\bm z}^*\}^\bot$ means that
\begin{align*}
\langle(-x^*,r^*,t^*),(x,t,r)\rangle = -x^*x+tr^*+rt^*=0,
\end{align*}
which implies that (noting that $t>0$ and $t^*>0$)
\begin{align*}
0&=r+t\frac{r^*}{t^*}- x\frac{x^*}{t^*} \geq tf(x/t)+tf^*(x^*/t^*)- x \frac{x^*}{t^*}\notag\\
&=t\left(f(x/t)+f^*(x^*/t^*)-\frac{x}{t}\frac{x^*}{t^*} \right)\geq 0, 
\end{align*}
where the first inequality follows from \eqref{prop-e-face-2-2},
and the second inequality follows from \eqref{young}.
Since $t>0$,  we deduce from the last display that
\begin{align}
&f(x/t)+f^*(x^*/t^*)- \frac{x}{t}\frac{x^*}{t^*}=0, \label{prop-e-face-4}\\
&r=tf(x/t),\,\,\,\, r^*=t^*f^*(x^*/t^*).\label{prop-e-face-4-1}
\end{align}
Invoking  \eqref{young-1} and the fact that $f$ is Legendre,   we deduce from  \eqref{prop-e-face-4} that
\begin{align}\label{prop-e-face-5}
\frac{x}{t}=  (f^*)'(x^*/t^*).
\end{align}

{\bf Case (b):} From  \eqref{prop-e-face-2-1}, $(x,t,r)\in \{{\bm z}^*\}^\bot$  implies  that
\begin{align*}
\langle(-x^*,r^*,t^*),(x,0,r)\rangle= -x x^*+rt^*=0.
\end{align*}
Recall that $t^*>0$. The last display means that
\begin{align}\label{prop-e-face-6}
  x\frac{x^*}{t^*}=r \overset{\rm (a)}{\geq} f_\infty(x)\overset{\rm (b)}{=} \sigma_{{\rm dom}\, f^*}(x)=\sup \{  xy : y\in {\rm dom}\, f^*\} \overset{\rm (c)}{\geq}  x\frac{x^*}{t^*},
\end{align}
where (a) follows from \eqref{prop-e-face-2-2},
 (b) follows from \eqref{infty-sigma}, and (c) holds because $z^*\in \mathcal{K}^{*t}$ (which implies that $x^*/t^*\in {\rm dom}\, f^*$). We  deduce from \eqref{prop-e-face-6} that
\begin{align}\label{prop-e-face-7}
r=f_\infty(x)=\sigma_{{\rm dom}\, f^*}(x)= x\frac{x^*}{t^*},
\end{align}
and the third equality is equivalent to
\begin{align}\label{prop-e-face-8}
  x\in \mathcal{N}_{{\rm dom}\, f^*}(x^*/t^*).
\end{align}

After the above discussion regarding \eqref{prop-e-face-2-1}, we now proceed in our proof by letting $\beta \coloneqq x^*/t^*$ and, in addition, when $\beta\in {\rm int}({\rm dom}\,f^*)$, we let $\bar\beta \coloneqq (f^*)'(\beta)$.
We further consider two cases:
\begin{enumerate}[\rm (i)]
\item $\beta\in {\rm int}({\rm dom}\, f^*)$. Then $r^* = t^*f^*(\beta)$ in view of \eqref{prop-e-face-2-2} and hence ${\bm z}^*=(-x^*,r^*,t^*)=t^*(-\beta,f^*(\beta),1)$.
Now, if $(x,t,r)\in{\rm epi}\,f^\pi\cap \{{\bm z}^*\}^\bot$ and $t = 0$, then we see from \eqref{prop-e-face-8} that $x=0$ and thus we  deduce from \eqref{prop-e-face-7} that $r=f_\infty(x)=0$, which further implies that $\mathcal{K}^\theta\cap \{{\bm z}^*\}^\bot=\{\bm 0\}$.
 Combining this observation with  \eqref{prop-e-face-4-1} and \eqref{prop-e-face-5}, we deduce that
\begin{align*}
\mathcal{F}^e_\beta&\coloneqq {\rm epi}\,f^\pi\cap \{{\bm z}^*\}^\bot=(\mathcal{K}^t \cap \{{\bm z}^*\}^\bot)\cup (\mathcal{K}^\theta \cap \{{\bm z}^*\}^\bot)\\
&=\{(x,t,r): x/t =  \bar\beta, r=tf(x/t), t>0 \}\cup \{\bm 0\}\\
&=\{t(\bar\beta,1,f(\bar\beta)): t\geq 0 \}.
\end{align*}

\item  $\beta\in{\rm bdry}({\rm dom}\, f^*)\cap {\rm dom}\, f^*$.
In this case,
we know that  $(f^*)'(\beta)$ does not exist thanks to the property of Legendre functions. Now, suppose $(x,t,r)\in{\rm epi}\,f^\pi\cap \{{\bm z}^*\}^\bot$.
In view of \eqref{prop-e-face-5}, we must have $t = 0$, which implies that $\mathcal{K}^t\cap \{{\bm z}^*\}^\bot=\emptyset$.
Therefore, in view of \eqref{prop-e-face-2-2} and  \eqref{prop-e-face-7},
 if ${\bm z}^*=(-x^*,r^*,t^*)=t^*(-\beta,r^*/t^*,1)$  with $\beta\in{\rm bdry}({\rm dom}\, f^*)\cap {\rm dom}\, f^*$ and $r^*\geq t^*f^*(\beta)$,  we deduce  that
\begin{align*}
\mathcal{F}^e_{\rm bd}&\coloneqq {\rm epi}\,f^\pi\cap \{{\bm z}^*\}^\bot=\mathcal{K}^\theta\cap \{{\bm z}^*\}^\bot =\{(x,0,f_\infty(x)): f_\infty(x)= \beta x\}.
\end{align*}
\end{enumerate}
This completes the proof.
\end{proof}

Next, we consider the nontrivial exposed faces that arise from ${\rm epi}\, f^{\pi} \cap \{{\bm z}^*\}^\bot$ for some ${\bm z}^*=(-x^*, s^*,0)\in\mathcal{K}^{*\theta} \setminus \{\bm 0\}$, where $\mathcal{K}^{*\theta}$ is given in \eqref{dual-cone}.

\begin{proposition}\label{prop-md-face}
Let ${\bm z}^*=(-x^*,s^*,0)\in\mathcal{K}^{*\theta}\setminus \{\bm 0\}$ be such that ${\rm epi}\, f^{\pi} \cap \{{\bm z}^*\}^\bot$ is a nontrivial exposed face, where $\mathcal{K}^{*\theta}$ is given in \eqref{dual-cone}.
\begin{enumerate}[\rm (i)]
\item If   $s^*>(f^*)_{\infty}(x^*)$ and $x^* = 0$, then we have
\begin{align}\label{prop-md-face-1}
\mathcal{F}^e_{\theta}\coloneqq {\rm epi}\, f^{\pi} \cap \{{\bm z}^*\}^\bot=\mathcal{K}^\theta\cap \{{\bm z}^*\}^\bot =\{(x,0,s): (x,s)\in {\rm epi}\, f_\infty \} = \mathcal{K}^\theta,
\end{align}
where $\mathcal{K}^\theta$ is defined in \eqref{per-cone}.
\item If   $s^*>(f^*)_{\infty}(x^*)$ and $x^* \neq 0$, then we have
\begin{align}\label{prop-md-face-1_5}
\mathcal{F}^e_{\neq\theta}\coloneqq {\rm epi}\, f^{\pi} \cap \{{\bm z}^*\}^\bot=\mathcal{K}^\theta\cap \{{\bm z}^*\}^\bot =\{(0,0,s): s\ge 0 \}.
\end{align}
\item If    $s^*=(f^*)_\infty(x^*)$ with $x^*\neq 0$, then we have
\begin{align}\label{prop-md-face-2}
\mathcal{F}^e_{\mathds{1}} \coloneqq {\rm epi}\, f^{\pi} \cap \{{\bm z}^*\}^\bot=\{ (x,t,r) :   r\geq t f(\alpha),  x=\alpha t,  t\geq 0 \},
\end{align}
where $\alpha$ is the unique element in ${\rm bdry}({\rm dom}\, f)\cap {\rm dom}\, f$ such that $x^*\in\mathcal{N}_{{\rm dom}\, f}(\alpha)$.\footnote{Since ${\rm dom}\,f$ is an interval, the $\alpha$ is either the left or right end point, depending on the sign of $x^*$.}
\end{enumerate}
\end{proposition}
\begin{proof}
Let
\begin{align}\label{prop-md-face-2-1}
 (x,t,r)\in{\rm epi}\,f^\pi\cap \{{\bm z}^*\}^\bot,
\end{align}
where ${\bm z}^*\in\mathcal{K}^{*\theta}\setminus \{\bm 0\}$.
Since $(x,t,r)\in  {\rm epi}\,f^\pi$ and ${\bm z}^*\in  \mathcal{K}^{*\theta}$,
we have that
\begin{align}\label{prop-md-face-2-2}
 \begin{cases}
r\geq tf(x/t)& \mbox{if $t>0$},\\
r\geq f_\infty(x)& \mbox{if $t=0$},\\
\end{cases}\,\, \mbox{and} \,\,
  s^*\geq  (f^*)_\infty(x^*).
\end{align}
 We consider two cases for $t$ as follows:
\begin{enumerate}[\rm (a)]
\item $t>0$;
\item $t=0$.
\end{enumerate}

{\bf Case (a):} Since $(x,t,r)\in \{{\bm z}^*\}^\bot$, we see that
\[
\langle(-x^*,s^*,0),(x,t,r)\rangle  =  -xx^*+ts^*=0.
\]
This implies that
\begin{align}\label{prop-md-face-3}
\frac{x}{t} x^* =s^* \overset{\rm (a)}{\geq} (f^*)_\infty(x^*)\stackrel{\rm (b)}{=}\sigma_{{\rm dom}\, f}(x^*)=\sup\{ x^* y : y\in {\rm dom}\, f\}\stackrel{\rm (c)}{\geq} \frac{x}{t} x^*,
\end{align}
where (a) follows from the second inequality of \eqref{prop-md-face-2-2},
 (b) follows from  \eqref{infty-sigma},
 and (c) follows from  the fact that $x/t\in {\rm dom}\, f$ (see \eqref{prop-md-face-2-2}).
It follows  from \eqref{prop-md-face-3} that
\begin{align}\label{prop-md-face-4}
s^*=(f^*)_\infty(x^*) \,\, \mbox{and}\,\,
  x^*\in \mathcal{N}_{{\rm  dom}\, f}(x/t).
 \end{align}
From the first equality in \eqref{prop-md-face-4}, we see that $x^*\neq 0$ (otherwise we will have ${\bm z}^*={\bm 0}$, a contradiction).
 Then, since $x^*\neq 0$, the inclusion in \eqref{prop-md-face-4}  implies that
 $x/t\in {\rm bdry}({\rm dom}\, f)\cap {\rm dom}\, f$.  This implies that there exists $\alpha\in  {\rm bdry}({\rm dom}\, f)\cap {\rm dom}\, f$ such that
\begin{align}\label{prop-md-face-4-1}
\mbox{$x^*\in \mathcal{N}_{{\rm dom}\, f}(\alpha)$ and $x/t =\alpha$},
 \end{align}
and the $\alpha$, being an end point of the interval ${\rm dom}\, f$, uniquely depends on the sign of $x^*$.

{\bf Case (b):}  According to \eqref{prop-md-face-2-1}, we have  that
\begin{align}\label{prop-md-face-5}
\langle(-x^*,s^*,0),(x,0,r)\rangle = x x^*=0 \quad \mbox{and\ \ \ $(x,r)\in {\rm epi}\, f_\infty$}.
\end{align}

After discussing the above two cases for $t$, we next consider three cases for ${\bm z}^*=(-x^*,s^*,0)\in\mathcal{K}^{*\theta}\setminus \{\bm 0\}$ as follows:
\begin{enumerate}[\rm (i)]
\item $s^*>(f^*)_\infty(x^*)$ and $x^* = 0$;
\item $s^*>(f^*)_\infty(x^*)$ and $x^* \neq 0$;
\item $s^*=(f^*)_\infty(x^*)$ and $x^* \neq 0$.
\end{enumerate}

For cases (i) and (ii), the  first equality in  \eqref{prop-md-face-4} implies  that $\mathcal{K}^t \cap \{{\bm z}^*\}^\bot=\emptyset$.
Therefore, we can  deduce from  \eqref{prop-md-face-5} that
\[
{\rm epi}\, f^{\pi} \cap \{{\bm z}^*\}^\bot=\mathcal{K}^\theta\cap \{{\bm z}^*\}^\bot = \{(x,0,s): xx^* = 0, (x,s)\in {\rm epi}\, f_\infty \},
\]
and the desired expressions for ${\cal F}^e_{\theta}$ and ${\cal F}^e_{\neq \theta}$ follow immediately.

Finally, for case (iii), it follows that for the unique $\alpha$ satisfying $x^*\in\mathcal{N}_{{\rm dom}\, f}(\alpha)$,
\begin{align}
\mathcal{F}^e_{\mathds{1}} \coloneqq &\, {\rm epi}\, f^{\pi} \cap \{{\bm z}^*\}^\bot\notag\\
=&\{(x,0,s) :   x x^*=0, (x,s)\in {\rm epi}\, f_\infty \}\cup\{ (x,t,r) :   r\geq t f(x/t),  x/t=\alpha,  t>0 \}\notag\\
=&\{(0,0,s): s\geq 0\}\cup\{ (x,t,r) :   r\geq tf(\alpha),  x/t=\alpha,  t>0 \}\notag\\
=&\{ (x,t,r) :   r\geq t f(\alpha),  x=\alpha t, t\geq 0 \}, \notag
\end{align}
where the second equality follows from \eqref{prop-md-face-2-2}, \eqref{prop-md-face-4}, \eqref{prop-md-face-4-1} and \eqref{prop-md-face-5},
and the third equality holds because $x^*\neq 0$.
This gives the face $\mathcal{F}^e_{\mathds{1}}$ in \eqref{prop-md-face-2}.
\end{proof}

All exposed faces of ${\rm epi}\, f^\pi$ given in Propositions \ref{prop-e-face} and \ref{prop-md-face} are summarized in  Table \ref{table-e-face}.
\begin{table}[h]
 \centering
\begin{tabular}{c|c|c}
\hline
\multirow{2}{*}{{\rm Case}}
& \multirow{2}{*}{${\bm z}^*\in {\rm bdry}(({\rm epi}\, f^\pi)^*)\setminus \{\bm 0\}$} & \multirow{2}{*}{the exposed face  ${\rm epi}\, f^\pi\cap\{{\bm z}^*\}^\bot$} \\
&   &   \\
\hline
\multirow{3}{*}{1}
&  ${\bm z}^*=(-x^*,r^*,t^*)$, $t^*>0$,  &  \multirow{3}{*}{$\mathcal{F}^e_\beta=\{t(\bar{\beta},1, f(\bar{\beta})): t\geq 0\}$ with $\bar{\beta}\coloneqq (f^*)'(\beta)$, see \eqref{prop-e-face-1}}\\
&$\beta\coloneqq x^*/t^* \in {\rm int}({\rm dom}\, f^*)$,  & \\
& $r^*=t^*f^*(\beta)$ & \\
\hline
\multirow{3}{*}{2}
& ${\bm z}^*=(-x^*,r^*,t^*)$, $t^*>0$,  & \\
& $\beta\coloneqq x^*/t^*$, $r^*\geq t^* f^*(\beta)$ & $\mathcal{F}^e_{\rm bd}= \{(x,0, f_\infty(x)): f_\infty(x)= \beta x \}$, see \eqref{prop-e-face-2}\\
& $\beta\in {\rm bdry}({\rm dom}\, f^*)\cap {\rm dom}\, f^*$ &   \\
\hline
\multirow{2}{*}{3}
& ${\bm z}^*=(-x^*,s^*,0)$, & \multirow{2}{*}{$\mathcal{F}^e_{\theta}=\{(x,0,s): (x,s)\in {\rm epi}\, f_\infty \}$, see \eqref{prop-md-face-1}} \\
& $s^*>(f^*)_{\infty}(x^*)$, $x^* = 0$ & \\
\hline
\multirow{2}{*}{4}
& ${\bm z}^*=(-x^*,s^*,0)$,  &  \multirow{2}{*}{$\mathcal{F}^e_{\neq\theta}=\{(0,0,s): s\ge 0 \}$, see \eqref{prop-md-face-1_5}} \\
& $s^*>(f^*)_{\infty}(x^*)$, $x^* \neq 0$ & \\
\hline
\multirow{2}{*}{5}
& ${\bm z}^*=(-x^*,s^*,0)$,  & $\mathcal{F}^e_{\mathds{1}} =\{ (x,t,r) :   r\geq t f(\alpha),  x=\alpha t,  t\geq 0 \}$  \\
& $s^*=(f^*)_\infty(x^*)$, $x^*\neq 0$ & with $\alpha$ uniquely defined via $x^*\in\mathcal{N}_{{\rm dom}\, f}(\alpha)$, see \eqref{prop-md-face-2} \\
\hline
\end{tabular}
  \caption{All exposed faces of ${\rm epi}\, f^\pi$}\label{table-e-face}
\end{table}

\subsection{Non-Exposed  faces}

In this section, we study the non-exposed faces of ${\rm epi}\, f^{\pi}$ in \eqref{per-cone} and, yet again, we recall that we are under the blanket Assumption~\ref{assu-blanket}. 
Generally speaking, every proper face of a closed convex cone is contained in some proper exposed face, e.g., see \cite[Proposition 3.6]{Borwein-Wolkowicz-1981}.

Accordingly, if there are any non-exposed faces, they must be subfaces of the faces described in Section~\ref{sec:exposed_faces}.
The one-dimensional faces $\mathcal{F}^e_{\beta}$, $\mathcal{F}^e_{\rm bd}$ and $\mathcal{F}^e_{\neq \theta}$ only contain the trivial exposed face $\{\bm 0\}$ as proper faces, so they cannot contain non-exposed faces of ${\rm epi}\, f^{\pi}$.
Thus, the non-exposed proper faces of
${\rm epi}\, f^{\pi}$ (if any) must be contained in some proper 2-dimensional exposed face of ${\rm epi}\, f^{\pi}$.

According to Lemma \ref{lemma-dom-fstar}(i),  ${\rm cl}({\rm dom}\, f^*)$ must take one of the following forms:
\[
[-a,\infty),\, (-\infty,b], \,[-a,b],  \, \R,
\]
for some $a,b\geq 0$ with $b+a>0$.
The non-exposed faces of ${\rm epi}\, f^\pi$ in \eqref{per-cone} are described below, according to the above four cases.
\begin{enumerate}[\rm (i)]
\item  ${\rm cl}({\rm dom}\, f^*)=[-a,\infty)$ for some $a\geq 0$. In this case, from \eqref{infty-sigma} we have
\[
f_\infty(x) = \sigma_{{\rm dom}\, f^*}(x) = -ax + \delta_{(-\infty,0]}(x).
\]
We consider two further subcases.
\begin{itemize}
\item  ${\rm int}({\rm dom}\,(f^*)_\infty)= \emptyset$. This means ${\rm dom}\,(f^*)_\infty = \{0\}$.
Hence, from Proposition \ref{prop-md-face}, we see that the cases giving rise to
$\mathcal{F}^e_{\neq \theta}$ and $\mathcal{F}^e_{\mathds{1}}$ are void and the only 2-dimensional exposed face of ${\rm epi}\, f^{\pi}$ is $\mathcal{F}^e_{\theta}=\{(x,0,s): x\leq 0, s\geq -ax\}$.
Then, the one-dimensional exposed faces of $\mathcal{F}^e_{\theta}$ are the following:
\begin{align}\label{face-ne-f12}
\hat{\mathcal{F}}_{1}\coloneqq \{x(1,0,-a): x\leq 0\},\ \ \
\hat{\mathcal{F}}_{2}\coloneqq \{(0,0,s): s\geq 0\}.
\end{align}

If ${\rm dom}\, f^*=[-a,\infty)$, then $\hat{\mathcal{F}}_2$ is the unique non-exposed face since $\hat{\mathcal{F}}_1=\mathcal{F}_{\rm bd}^e$ with $\beta=-a$ according to \eqref{prop-e-face-2}.
If ${\rm dom}\, f^*=(-a,\infty)$, then $\hat{\mathcal{F}}_1$ and $\hat{\mathcal{F}}_2$ are non-exposed faces because the case giving rise to $\mathcal{F}_{\rm bd}^e$ is void.

\item ${\rm int}({\rm dom}\,(f^*)_\infty)\neq \emptyset$. In this case, we must have ${\rm dom}\,(f^*)_\infty \neq \{0\}$.  According to Proposition~\ref{prop-md-face}, the 2-dimensional exposed faces of ${\rm epi}\, f^{\pi}$ are $\mathcal{F}^e_{\theta}=\{(x,0,s): x\leq 0, s\geq -ax\}$ and $\mathcal{F}^e_{\mathds{1}} =\{ (x,t,r) :   r\geq t f(\alpha),  x=\alpha t,  t\geq 0 \}$ for the
 $\alpha$ uniquely defined via $ x^*\in\mathcal{N}_{{\rm dom}\, f}(\alpha)$,  while $\mathcal{F}^e_{\neq \theta}=\hat{\mathcal{F}}_{2}=\{(0,0,s): s\geq 0\}$ is a one-dimensional exposed face of  ${\rm epi}\, f^{\pi}$.
    The one-dimensional exposed faces of $\mathcal{F}^e_{\theta}$ are $\hat{\cal F}_1$ and $\hat{\cal F}_2$,
 and the one-dimensional exposed faces of $\mathcal{F}^e_{\mathds{1}}$ are $\hat{\mathcal{F}}_2$ and faces of the form\footnote{We abuse the notation and use $\hat{\mathcal{F}_{3}}$ to denote faces of this type, where $\alpha$ depends on the choice of $x^*$ and is a boundary point of ${\rm dom}\,f$.
 	Hence, there can be up to two faces of this form. Moreover, since $\alpha$ is a boundary point of ${\rm dom}\,f$, this type of faces does not overlap with exposed faces of the type ${\cal F}^e_{\beta}$.}
\begin{align}\label{face-ne-f3}
\hat{\mathcal{F}_{3}}\coloneqq \{t(\alpha,1,f(\alpha)): t\geq 0\}.
\end{align}

 If  ${\rm dom}\, f^*=[-a,\infty)$, then $\hat{\mathcal{F}}_3$ is the unique type of non-exposed face since $\hat{\mathcal{F}}_1=\mathcal{F}^e_{\rm bd}$ with $\beta=-a$ from \eqref{prop-e-face-2} and $\hat{\cal F}_2=\mathcal{F}^e_{\neq \theta}$ is an exposed face of ${\rm epi}\, f^\pi$. If  ${\rm dom}\, f^*=(-a,\infty)$, then $\hat{\mathcal{F}}_1$ and (the type) $\hat{\mathcal{F}}_3$ are non-exposed faces since the case giving rise to $\mathcal{F}_{\rm bd}^e$ is void (see \eqref{prop-e-face-2}) and  $\hat{\cal F}_2=\mathcal{F}^e_{\neq \theta}$ is an exposed face of ${\rm epi}\, f^\pi$.
\end{itemize}

\item ${\rm cl}({\rm dom}\, f^*)=(-\infty,b]$ for some $b\geq 0$.
In this case, from \eqref{infty-sigma} we have
\[
f_\infty(x) = \sigma_{{\rm dom}\, f^*}(x) = bx + \delta_{[0,\infty)}(x).
\]
We consider two further subcases.
\begin{itemize}
\item  ${\rm int}({\rm dom}\,(f^*)_\infty)= \emptyset$.  This means ${\rm dom}\,(f^*)_\infty = \{0\}$.
Thus, we have from Proposition \ref{prop-md-face}  that the cases leading to  $\mathcal{F}^e_{\neq \theta}$ and $\mathcal{F}^e_{\mathds{1}}$ are void  and the only 2-dimensional   exposed face of ${\rm epi}\, f^{\pi}$ is $\mathcal{F}^e_{\theta}=\{(x,0,s): x\geq  0, s\geq bx\}$.
Then, the nontrivial exposed faces of   $\mathcal{F}^e_{\theta}$ are $\hat{\mathcal{F}}_2$ and
\begin{align}\label{face-ne-f4}
\hat{\mathcal{F}}_{4}\coloneqq \{x(1,0,b): x\geq 0\}.
\end{align}

If ${\rm dom}\, f^*=(-\infty,b]$, then $\hat{\mathcal{F}}_2$ is the unique non-exposed face since $\hat{\mathcal{F}}_4=\mathcal{F}_{\rm bd}^e$ with $\beta=b$ according to \eqref{prop-e-face-2}.
If ${\rm dom}\, f^*=(-\infty,b)$, then $\hat{\mathcal{F}}_2$ and $\hat{\mathcal{F}}_4$ are non-exposed faces since the case giving rise to $\mathcal{F}_{\rm bd}^e$ is void (see \eqref{prop-e-face-2}).

\item ${\rm int}({\rm dom}\,(f^*)_\infty)\neq \emptyset$.
In this case, it holds that ${\rm dom}\,(f^*)_\infty \neq \{0\}$ and hence in view of Proposition~\ref{prop-md-face} we see that $\mathcal{F}^e_{\theta}=\{(x,0,s): x\geq  0, s\geq bx\}$  and  $\mathcal{F}^e_{\mathds{1}} =\{ (x,t,r) :   r\geq t f(\alpha),  x=\alpha t,  t\geq 0 \}$ for the $\alpha$ uniquely defined via $x^*\in\mathcal{N}_{{\rm dom}\, f}(\alpha)$ are the 2-dimensional exposed faces of ${\rm epi}\, f^{\pi}$
 while $\mathcal{F}^e_{\neq \theta}=\hat{\mathcal{F}}_{2}=\{(0,0,s): s\geq 0\}$ is a one-dimensional exposed face of  ${\rm epi}\, f^{\pi}$.
 The one-dimensional exposed faces of $\mathcal{F}^e_{\mathds{1}}$ are $\hat{\mathcal{F}}_2$ and $\hat{\mathcal{F}}_3$,  and $\mathcal{F}^e_{\theta}$ contains one-dimensional exposed faces  $\hat{\mathcal{F}}_2$ and $\hat{\mathcal{F}}_4$.

If  ${\rm dom}\, f^*=(-\infty,b]$, then $\hat{\mathcal{F}}_3$ is the unique type of non-exposed faces since $\hat{\mathcal{F}}_4=\mathcal{F}_{\rm bd}^e$ with $\beta=b$ according to \eqref{prop-e-face-2} and $\hat{\mathcal{F}}_{2}=\mathcal{F}^e_{\neq \theta}$ is an exposed face of ${\rm epi}\, f^\pi$. If  ${\rm dom}\, f^*=(-\infty,b)$, then (the type) $\hat{\mathcal{F}}_3$ and $\hat{\mathcal{F}}_4$ are non-exposed faces since the case giving rise to $\mathcal{F}_{\rm bd}^e$ is void (see \eqref{prop-e-face-2}) and $\hat{\mathcal{F}}_{2}=\mathcal{F}^e_{\neq \theta}$ is an exposed face of ${\rm epi}\, f^\pi$.
\end{itemize}

\item  ${\rm cl}({\rm dom}\, f^*)=[-a,b]$ for some $a,b\geq 0$ with $b+a>0$. In this case, invoking \eqref{infty-sigma} we have
\begin{align}\label{case_f_infty}
f_\infty(x) = \sigma_{{\rm dom}\, f^*}(x) =
\begin{cases}
 bx & \mbox{if $x\geq 0$}, \\
 -ax & \mbox{if $x<  0$}.
\end{cases}
\end{align}

In this case, we see that ${\rm int}({\rm dom}\,(f^*)_\infty)= \emptyset$ because ${\rm dom}\,f^*$ is bounded. Hence, the cases leading to $\mathcal{F}^e_{\neq \theta}$ and $\mathcal{F}^e_{\mathds{1}}$ are void according to Proposition \ref{prop-md-face}.
Thus, $\mathcal{F}^e_\theta$ is the unique 2-dimensional  exposed face of ${\rm epi}\, f^{\pi}$.
 From \eqref{case_f_infty} and \eqref{prop-md-face-1}, we have that
$\mathcal{F}^e_\theta=\{(x,0,s): s\geq bx, s\geq -ax, x\in \R\}$.
The one-dimensional exposed faces of  $\mathcal{F}^e_{\theta}$
are $\hat{\mathcal{F}}_1$ and $\hat{\mathcal{F}}_4$.

If ${\rm dom}\, f^*=[-a,b]$, then there is no non-exposed face since $\hat{\mathcal{F}}_1=\mathcal{F}_{\rm bd}^e$ with $\beta=-a$ and  $\hat{\mathcal{F}}_4=\mathcal{F}_{\rm bd}^e$ with $\beta=b$ according to \eqref{prop-e-face-2}.
If ${\rm dom}\, f^*=(-a,b]$, then $\hat{\mathcal{F}}_1$ is the unique non-exposed faces since $\hat{\mathcal{F}}_4=\mathcal{F}_{\rm bd}^e$ with $\beta=b$ according to \eqref{prop-e-face-2}.
If ${\rm dom}\, f^*=[-a,b)$, then $\hat{\mathcal{F}}_4$ is the unique non-exposed faces since $\hat{\mathcal{F}}_1=\mathcal{F}_{\rm bd}^e$ with $\beta=-a$ according to \eqref{prop-e-face-2}.
If ${\rm dom}\, f^*=(-a,b)$, then $\hat{\mathcal{F}}_1$ and $\hat{\mathcal{F}}_4$ are the non-exposed faces since the case corresponding to $\mathcal{F}_{\rm bd}^e$ is void (see \eqref{prop-e-face-2}).

\item ${\rm dom}\, f^*=\R$.
In this case, we have ${\rm dom}\, f_\infty=\{0\}$  thanks to \eqref{infty-sigma}. From this fact and  \eqref{prop-md-face-1}, we see that $\mathcal{F}^e_\theta=\hat{\mathcal{F}}_2=\{(0,0,s): s\geq 0\}$  reduces to a one-dimensional exposed face of ${\rm epi}\, f^\pi$. Next, we consider two further subcases.
\begin{itemize}
\item If ${\rm int}({\rm dom}\,(f^*)_\infty)= \emptyset$, then there is no non-exposed face since $\mathcal{F}^e_{\mathds{1}}$ is void (see \eqref{prop-md-face-2}) and $\mathcal{F}^e_\theta$ is a one-dimensional face of ${\rm epi}\, f^\pi$.
\item If ${\rm int}({\rm dom}\,(f^*)_\infty)\neq \emptyset$,  then the only (type of) 2-dimensional exposed face of ${\rm epi}\, f^\pi$ is $\mathcal{F}^e_{\mathds{1}} =\{ (x,t,r) :   r\geq t f(\alpha),  x=\alpha t,  t\geq 0 \}$ for the $\alpha$ uniquely defined via $x^*\in\mathcal{N}_{{\rm dom}\, f}(\alpha)$.  The one-dimensional exposed faces of $\mathcal{F}^e_{\mathds{1}}$ are $\hat{\mathcal{F}}_2$ and (of the type) $\hat{\mathcal{F}}_3$.
One can see that $\hat{\mathcal{F}}_3$ is the unique type of non-exposed faces of ${\rm epi}\, f^\pi$ since $\hat{\mathcal{F}}_2=\mathcal{F}^e_\theta$ is an exposed face of ${\rm epi}\, f^\pi$.
 \end{itemize}
\end{enumerate}
Table \ref{table-ne-face} contains the summary of all possible non-exposed faces of ${\rm epi}\, f^\pi$.

\begin{table}[h]
 \centering
\begin{tabular}{ccccccc}
\hline
Case & ${\rm dom}\, f^*$ & ${\rm int}({\rm dom}\,(f^*)_\infty)$ & $\hat{\mathcal{F}}_1$ & $\hat{\mathcal{F}}_2$ &  $\hat{\mathcal{F}}_3$ & $\hat{\mathcal{F}}_4$ \\
\hline
1& $[-a,\infty)$ & $\emptyset$ & \ding{55} & \ding{51} & \ding{55} & \ding{55} \\
2& $(-a,\infty)$ & $\emptyset$ & \ding{51} & \ding{51} & \ding{55} & \ding{55} \\
3& $[-a,\infty)$ & $\neq\emptyset$ & \ding{55} & \ding{55} & \ding{51} &  \ding{55}\\
4& $(-a,\infty)$ & $\neq\emptyset$ & \ding{51} & \ding{55} & \ding{51} & \ding{55} \\
\hline
5&  $(-\infty,b]$ & $\emptyset$ & \ding{55} & \ding{51} & \ding{55} & \ding{55} \\
6& $(-\infty,b)$ & $\emptyset$ & \ding{55} & \ding{51} & \ding{55} & \ding{51} \\
7& $(-\infty,b]$ & $\neq\emptyset$ & \ding{55} &\ding{55} & \ding{51} &\ding{55}  \\
8& $(-\infty,b)$ & $\neq\emptyset$ &\ding{55}&\ding{55} & \ding{51} & \ding{51} \\
\hline
9&  $[-a,b]$ & $\emptyset$ &\ding{55} & \ding{55} & \ding{55} & \ding{55} \\
10& $(-a,b]$ & $\emptyset$ & \ding{51} & \ding{55} & \ding{55} & \ding{55} \\
11& $[-a,b)$ & $\emptyset$ & \ding{55} & \ding{55} & \ding{55} & \ding{51} \\
12& $(-a,b)$ & $\emptyset$ & \ding{51} & \ding{55} & \ding{55} & \ding{51} \\
\hline
13& $\R$ & $\emptyset$ &\ding{55} & \ding{55} &  \ding{55} & \ding{55} \\
14& $\R$ & $\neq \emptyset$ & \ding{55} & \ding{55} & \ding{51} & \ding{55} \\
\hline
\end{tabular}
  \caption{Non-exposed faces of ${\rm epi}\, f^\pi$; see \eqref{face-ne-f12}, \eqref{face-ne-f3} and  \eqref{face-ne-f4} for the definitions of $\hat{\mathcal{F}}_1$ to $\hat{\mathcal{F}}_4$.}\label{table-ne-face}
\end{table}

\section{Facial residual functions}\label{section-frf}

In this section, we show
several cases where $\mathds{1}$-FRFs for
${\rm epi}\,f^\pi$ and ${\bm z}^* \in ({\rm epi}\,f^\pi)^*$ can be taken to be
H\"olderian of exponent $1/2$ as in Definition~\ref{def-holder-1frf}.
This includes the cases where ${\bm z}^*$ exposes the following faces.
\begin{enumerate}
  \item all faces of the form ${\cal F}^e_\beta$  as in \eqref{prop-e-face-1} with $\beta\in {\rm int}({\rm dom}\, f^*)\setminus f'(W)$ (see Theorem~\ref{thm-frf-1d-f}), where $f'(W)\coloneqq \{ v : v= f'(w), w\in W\}$ is a finite set with $W\subset  {\rm int}({\rm dom}\,f)$  given in Assumption \ref{assu-blanket};
\item those ${\cal F}^e_{\rm bd}$ as in \eqref{prop-e-face-2} with $r^* > t^* f^*(\beta)$ (see Theorem~\ref{thm-frf-2d-f});
  \item those ${\cal F}^e_{\neq \theta}$ as in \eqref{prop-md-face-1_5} (see Theorem~\ref{thm-frf-f0}).
\end{enumerate}
We will then argue that ${\bm z}^*\in {\rm bdry}(({\rm epi}\, f^{\pi})^*)$ that leads to other exposed faces form a set of measure zero (with respect to the two-dimensional Hausdorff measure) in ${\rm bdry}(({\rm epi}\, f^{\pi})^*)$ in Section~\ref{section-measure}.

\subsection{$\mathds{1}$-FRF concerning $\mathcal{F}^e_\beta$}

In this section, we compute $\mathds{1}$-FRFs concerning the class of one-dimensional faces $\mathcal{F}^e_\beta$.
Recall that we are under the blanket Assumption~\ref{assu-blanket} and recall from Proposition \ref{prop-e-face}(i) that
\begin{align*}
\mathcal{F}^e_\beta={\rm epi}\, f^\pi \cap \{{\bm z}^*\}^\bot= \{t \hat{{\bm f}}: t\geq 0\},
\end{align*}
where
\begin{align}
&\hat{{\bm f}}= (\bar{\beta},1, f(\bar{\beta}))  \,\, \mbox{with $\bar\beta\coloneqq  (f^*)'(\beta) \in {\rm int}({\rm dom}\, f)$},  \label{def-f}\\
& {\bm z}^*=t^*(-\beta, f^*(\beta), 1)\in {\rm bdry}(({\rm epi}\, f^\pi)^*) \,\, \mbox{with $\beta\in {\rm int}({\rm dom}\, f^*)$ and $t^*>0$}. \label{def-z}
\end{align}
We will first show that, for ${\bm z}^*$ as in \eqref{def-z} with $\beta\notin f'(W)$ (where $W$ is given in Assumption~\ref{assu-blanket}), an error bound in the form of \eqref{thm-frf-2} in Theorem \ref{thm-frf} holds for $\mathfrak{g}=|\cdot|^{\frac{1}{2}}$ .

\begin{theorem}\label{thm-frf-1d-f}
Let $\mathcal{F}^e_\beta={\rm epi}\, f^\pi \cap \{{\bm z}^*\}^\bot$ be a one-dimensional face with ${\bm z}^*=t^*(-\beta, f^*(\beta), 1)$, $\beta\in {\rm int}({\rm dom}\, f^*)\setminus  f'(W)$ and $t^*>0$, where $f'(W)\coloneqq \{ v : v= f'(w), w\in W\}$ with $W\subset  {\rm int}({\rm dom}\,f)$  given in Assumption \ref{assu-blanket}.
Let $\eta>0$ and $\gamma_{{\bm z}^*, \eta}$ be defined in \eqref{thm-frf-1} with $\mathfrak{g}=|\cdot|^{\frac{1}{2}}$ and $({\rm epi}\, f^\pi,\mathcal{F}^e_{\beta},{\bm z}^*)$ in place of $(\mathcal{K},\mathcal{F},{\bm z})$. Then we have that $\gamma_{{\bm z}^*, \eta}\in(0,\infty]$ and
\[
{\rm d}({\bm q}, \mathcal{F}^e_\beta)\leq
\kappa_{{\bm z}^*,\eta}  ({\rm d}({\bm q}, {\rm epi}\, f^\pi))^{\frac{1}{2}} \quad \forall\, {\bm q}\in\{{\bm z}^*\}^\bot \cap B(\eta),
\]
where $\kappa_{{\bm z}^*,\eta} \coloneqq \max\{2\eta^{\frac{1}{2}}, 2\gamma_{{\bm z}^*,\eta}^{-1}\}$.
\end{theorem}
\begin{proof}
In view of Lemma \ref{lemma-gamma}, we need to show that
\begin{align}\label{thm-frf-1d-f-1}
\liminf_{k\rightarrow\infty} \frac{\|{\bm w}^k-{\bm v}^k\|^{\frac{1}{2}}}{\|{\bm w}^k-{\bm u}^k\|}>0
\end{align}
 for any $\bar{{\bm v}}\in \mathcal{F}^e_{\beta}$ and   sequence $\{{\bm v}^k\}\subset{\rm bdry}({\rm epi}\, f^\pi)\cap B(\eta)\backslash\mathcal{F}^e_\beta$ such that
\begin{align*}
\lim_{k\rightarrow \infty}{\bm v}^k=\lim_{k\rightarrow \infty}{\bm w}^k=\bar{{\bm v}} \quad \mbox{with ${\bm w}^k={\rm P}_{\{{\bm z}^*\}^\bot}{\bm v}^k$, ${\bm u}^k={\rm P}_{{\mathcal{F}^e_\beta}}{\bm w}^k$ and ${\bm w}^k\neq {\bm u}^k$}.
\end{align*}
By passing to subsequences if necessary, we consider two possible cases.
\begin{enumerate}[\rm (a)]
\item ${\bm v}^k\in \mathcal{K}^\theta$ for all $k$;
\item ${\bm v}^k\in \mathcal{K}^t$ for all $k$.
\end{enumerate}

{\bf Case (a):} We first assume ${\bm v}^k\in \mathcal{K}^\theta$ for all $k$, i.e.
\begin{align}\label{thm-frf-1d-f-1-1}
{\bm v}^k=(v^k_x,0,v^k_r) \,\, \mbox{with $(v^k_x, v^k_r)\in {\rm epi}\, f_\infty$}.
\end{align}

Since $\beta\in {\rm int}({\rm dom}\, f^*)\setminus f'(W)$, there exists sufficiently small $\varepsilon\in(0,1)$ such that
\begin{align}\label{thm-frf-1d-f-2-0}
(1\pm\varepsilon)\beta\in{\rm dom}\, f^*\ \ \ {\rm and}\ \ \ \beta+\varepsilon \xi\in {\rm dom}\, f^* \quad \mbox{$\forall\, \xi \in [-1,1]$}.
\end{align}
It follows that
\begin{align}
\|{\bm w}^k-{\bm v}^k\|&=\frac{|\langle {\bm z}^*, {\bm v}^k\rangle|}{\|{\bm z}^*\|} \overset{\rm (a)}{=}\frac{\langle {\bm z}^*, {\bm v}^k\rangle}{\|{\bm z}^*\|}=\frac{\langle t^*(-\beta, f^*(\beta), 1), (v^k_x,0, v^k_r)\rangle}{\|{\bm z}^*\|}\notag\\
&= \frac{t^*}{\|{\bm z}^*\|}(v^k_r- v^k_x \beta)\geq \frac{t^*}{\|{\bm z}^*\|}(f_\infty(v^k_x)- v^k_x \beta) \notag \\
&\overset{\rm (b)}{=}  \frac{t^*}{\|{\bm z}^*\|}(\sup\{ v^k_x y : y\in {\rm dom}\, f^* \}- v^k_x \beta) \notag\\
&\overset{\rm (c)}{\geq}
\begin{cases}
\frac{t^*}{\|{\bm z}^*\|}( v^k_x(\beta+\varepsilon (v^k_x/|v^k_x|)) - v^k_x \beta) &  \mbox{if $v^k_x\neq 0$}, \\
 0  &   \mbox{if $v^k_x= 0$,}
\end{cases} \notag \\
&= \frac{\varepsilon t^*}{\|{\bm z}^*\|}|v^k_x|, \label{thm-frf-1d-f-2-1}
\end{align}
where (a) follows from the facts that ${\bm z}^*\in ({\rm epi}\, f^\pi)^*$ and ${\bm v}^k\in {\rm epi}\, f^\pi$ and thus $\langle {\bm z}^*, {\bm v}^k\rangle \geq 0$,
(b) follows from \eqref{infty-sigma}, and (c) follows from \eqref{thm-frf-1d-f-2-0}. Next, notice that
\begin{align*}
\frac{1}{t^*}|\langle {\bm z}^*, {\bm v}^k\rangle|&=v^k_r- \beta v^k_x=  v^k_r-\frac{1}{1+{\rm sgn}(v^k_r)\varepsilon}(1+{\rm sgn}(v^k_r)\varepsilon) \beta v^k_x \\
&{\geq}  v^k_r-\frac{1}{1+{\rm sgn}(v^k_r)\varepsilon}f_\infty(v^k_x)\geq  \frac{\varepsilon {\rm sgn}(v^k_r) v^k_r}{1+{\rm sgn}(v^k_r)\varepsilon}\geq  \frac{\varepsilon}{1+\varepsilon}|v^k_r|,
\end{align*}
where the first inequality  follows from  the fact that  $(1\pm \varepsilon)\beta\in {\rm dom}\,f^*$ (see \eqref{thm-frf-1d-f-2-0}) and  \eqref{infty-sigma}, and the second inequality follows from  \eqref{thm-frf-1d-f-1-1}.
The last display implies that
\begin{align}\label{thm-frf-1d-f-2-2}
\|{\bm w}^k-{\bm v}^k\|=\frac{|\langle {\bm z}^*, {\bm v}^k\rangle|}{\|{\bm z}^*\|}\geq \frac{\varepsilon t^*}{(1+\varepsilon)\|{\bm z}^*\|} |v^k_r|.
\end{align}


 Let $c_1=\min\{\frac{\varepsilon t^*}{2\|{\bm z}^*\|},\frac{\varepsilon t^*}{2(1+\varepsilon)\|{\bm z}^*\|}\}$. In view of  \eqref{thm-frf-1d-f-2-1} and \eqref{thm-frf-1d-f-2-2}, we  deduce that
\begin{align}
\|{\bm w}^k-{\bm v}^k\|&=0.5\|{\bm w}^k-{\bm v}^k\|+0.5\|{\bm w}^k-{\bm v}^k\|\geq \frac{\varepsilon t^*}{2\|{\bm z}^*\|} |v^k_x|+\frac{\varepsilon t^*}{2(1+\varepsilon)\|{\bm z}^*\|}|v^k_r|\notag\\
&\geq c_1\|(v^k_x,0,v^k_r)\|=c_1\|{\bm v}^k\|. \label{thm-frf-1d-f-2-3}
\end{align}
Then,  it follows  that
\begin{align*}
\|{\bm w}^k-{\bm u}^k\|&={\rm d}({\bm w}^k,\mathcal{F}^e_\beta)\leq \|{\bm w}^k\|\leq \|{\bm v}^k\|\leq c_1^{-1}\| {\bm w}^k-{\bm v}^k\| \quad \forall\, k,
\end{align*}
where the first inequality holds because ${\bm 0}\in \mathcal{F}^e_\beta$, the  second inequality follows from
the fact that ${\bm 0}\in \{{\bm z}^*\}^\bot$ and the nonexpansiveness of the projection mapping, and the last inequality follows from \eqref{thm-frf-1d-f-2-3}.
Therefore, the above display shows  that \eqref{thm-frf-1d-f-1} is satisfied.

{\bf Case (b):} Next, we assume ${\bm v}^k\in \mathcal{K}^t$ for all $k$, i.e.
\begin{align}\label{Caseiivk}
0\neq {\bm v}^k=(v^k_x,v^k_t,v^k_t (f(v^k_x/v^k_t)+\varsigma^k)), \quad v^k_t>0 \ \ \ \ \ \forall \, k,
\end{align}
where $\varsigma^k\geq 0$ if $v^k_x/v^k_t\in {\rm bdry}({\rm dom}\, f)\cap {\rm dom}\,f$  with  $v^k_t\varsigma^k\leq \eta$ since ${\bm v}^k\in B(\eta)$, and $\varsigma^k= 0$ otherwise.\footnote{Recall that $\{{\bm v}^k\}\subset{\rm bdry}({\rm epi}\, f^\pi)$. }

Consider the following two functions $h_1$, $h_2$ from $\R$ to $\R\cup\{\infty\}$
\begin{align}
 h_1(\zeta)&\coloneqq f(\zeta)+f^*(\beta)-\beta \zeta,\label{thm-frf-h1} \\
 h_2(\zeta)&\coloneqq |\zeta-\bar{\beta}|+|f(\zeta)-f(\bar{\beta})|,\label{thm-frf-h2}
\end{align}
where $\bar{\beta} \coloneqq (f^*)'(\beta)$ satisfies  $\bar{\beta}\in {\rm int}({\rm dom}\,  f)\setminus W$ since $\beta\in {\rm int}({\rm dom}\,  f^*)\setminus f'(W)$ and $(f')^{-1}=(f^*)'$ thanks to \cite[Theorem 5.10]{Bauschke-Borwein-Combettes-2001}.
It is clear that both  $h_1$ and $h_2$  are nonnegative.
In view of the definition of $h_1$, \eqref{young-1} and $\beta= f'(\bar{\beta})\in {\rm int}({\rm dom}\, f^*)$, we know that
\begin{align}\label{thm-frf-h1-1}
\mbox{$h_1(\zeta)=0$ if and only if $\zeta=\bar\beta$}.
\end{align}
In addition, from the fact that $(f')^{-1}=(f^*)'$ thanks to \cite[Theorem 5.10]{Bauschke-Borwein-Combettes-2001},  we  have
\begin{align}\label{thm-frf-h1-2}
 h'_1(\bar{\beta})= f'(\bar\beta)-\beta= f'((f^*)'(\beta))-\beta=0.
\end{align}

Recall that $\hat{{\bm f}} = (\bar{\beta},1, f(\bar{\beta}))$  in \eqref{def-f} and ${\bm z}^*=t^*(-\beta, f^*(\beta), 1)$ in \eqref{def-z}.
We compute
\begin{align}
\|{\bm w}^k-{\bm v}^k\|&=\frac{|\langle {\bm z}^*, {\bm v}^k\rangle|}{\| {\bm z}^*\|}=\frac{\langle t^*(-\beta, f^*(\beta), 1), (v^k_x,v^k_t,v^k_t (f(v^k_x/v^k_t)+\varsigma^k))\rangle}{\| {\bm z}^*\|}\notag\\
&=\frac{t^*}{\|{\bm z}^*\|}v^k_t(-\beta (v^k_x/v^k_t)+ f^*(\beta)+f(v^k_x/v^k_t)+\varsigma^k)\notag\\
&=\frac{t^*v^k_t(h_1(v^k_x/v^k_t)+\varsigma^k)}{\|{\bm z}^*\|}, \label{thm-frf-1d-f-3}\\
\|{\bm w}^k-{\bm u}^k\|&\overset{\rm (a)}{\leq} \|{\bm v}^k-{\rm P}_{\mathcal{F}^e_\beta}{\bm v}^k\| \overset{\rm (b)}{\leq}  \|{\bm v}^k-v^k_t \hat{{\bm f}}\| =\|(v^k_x,v^k_t,v^k_t (f(v^k_x/v^k_t)+\varsigma^k))-v^k_t(\bar{\beta},1, f(\bar{\beta}))\|\notag \\
& =v^k_t\|(v^k_x/v^k_t-\bar{\beta},0, f(v^k_x/v^k_t)-f(\bar{\beta})+\varsigma^k)\|\notag \\
&\leq v^k_t h_2(v^k_x/v^k_t)+  v^k_t \varsigma^k, \label{thm-frf-1d-f-4}
\end{align}
where (a) follows from Lemma \ref{lemma-bound-wu}, and (b) holds because $v^k_t \hat{{\bm f}}\in \mathcal{F}^e_\beta$.

Since $\bar\beta\in {\rm int}({\rm dom}\, f)\setminus W$, we can choose $\Delta > 0$ such that $\Xi:= \{\xi: |\xi-\bar{\beta}|\le \Delta\}\subset {\rm int}({\rm dom}\,f)\setminus W$.
Since $h_1$ is strongly convex on $\Xi$ (thanks to Assumption \ref{assu-blanket} and the definition of $h_1$ in \eqref{thm-frf-h1}), we can now invoke \cite[Proposition~4.2]{Goebel-Rockafellar-2008} to deduce the existence of $\bar{c} > 0$ such that
\begin{equation*}
  h_1(\zeta) \ge h_1(\bar{\beta}) + h'_1(\bar{\beta})(\zeta - \bar{\beta}) + \bar{c} |\zeta - \bar{\beta}|^2\quad \mbox{whenever $|\zeta-\bar{\beta}|\le \Delta$}.
\end{equation*}
Recalling that $h_1(\bar\beta)=h'_1(\bar\beta)=0$ (thanks to \eqref{thm-frf-h1-1} and \eqref{thm-frf-h1-2}), we further deduce that
\begin{align}\label{thm-frf-1d-f-5-2}
h_1(\zeta)\geq \bar{c} |\zeta-\bar\beta|^2 \quad \mbox{whenever $|\zeta-\bar{\beta}|< \Delta$}.
\end{align}

On the other hand, we see that $f$ is Lipschitz continuous
on $\Xi$  and we denote by  $l_{\bar\beta}$ the
 Lipschitz continuity modulus.
Using this fact together with \eqref{thm-frf-h2}, we obtain the following
\begin{align}
 h_2(\zeta)=|\zeta-\bar{\beta}|+|f(\zeta)-f(\bar{\beta})|\leq (1+l_{\bar\beta})|\zeta-\bar{\beta}| \quad \mbox{whenever $|\zeta-\bar{\beta}|< \Delta$}. \label{thm-frf-1d-f-6}
\end{align}
We  deduce from \eqref{thm-frf-1d-f-5-2} and \eqref{thm-frf-1d-f-6} that there exists $c_2>0$ such that
\begin{align}\label{thm-frf-1d-f-8}
h_1(\zeta)^{\frac{1}{2}}\geq c_2 h_2(\zeta) \quad {\rm whenever}\  |\zeta-\bar\beta |< \Delta.
\end{align}

Define the function  $G:\R\times\R \rightarrow \R_+ \cup \{\infty\}$ as follows:
\begin{equation*}
G(\zeta,\xi)= \\
\left\{
  \begin{split}
&    \frac{-\beta \zeta + \xi +f^*(\beta)}{\sqrt{1+\|(\zeta,\xi)\|^2}}  &&\mbox{if $(\zeta,\xi)\in {\rm epi}\,f$ and $|\zeta-\bar{\beta}|\geq \Delta$}, \\
  &  \infty  &&\mbox{else}.
  \end{split}
  \right.
\end{equation*}
The function $G$ is lower semicontinuous. Moreover, $G$ is positive since for any $(\zeta,\xi)\in {\rm dom}\,G$ we have that
\[
-\beta \zeta + \xi +f^*(\beta)\geq -\beta \zeta + f(\zeta) +f^*(\beta) > 0,
\]
where the first inequality holds because $(\zeta,\xi)\in{\rm epi}\, f$ and the last inequality follows from  \eqref{young} and \eqref{young-1} because $\zeta\neq \bar{\beta}$ whenever $(\zeta,\xi)\in {\rm dom}\, G$.
Next, we will  show that
\begin{align}\label{thm-frf-1d-f-10}
\inf G>0.
\end{align}
Since $G$ is lower semicontinuous and positive, it suffices to show that
\begin{align}\label{thm-frf-1d-f-10-0}
\ell\coloneqq \liminf_{\|(\zeta,\xi)\|\rightarrow \infty} G(\zeta,\xi)>0.
\end{align}
To this end, let $\{(\zeta^l,\xi^l)\}\subset {\rm epi}\, f$ be a sequence with $\|(\zeta^l,\xi^l)\|\rightarrow \infty$ such that
$\lim_{l\rightarrow \infty} G(\zeta^l,\xi^l)=\ell$.
 Passing to a subsequence if necessary, we have
\begin{align}\label{thm-frf-1d-f-10-1}
(\zeta^*,\xi^*)\coloneqq \lim_{l\rightarrow \infty}\frac{(\zeta^l,\xi^l)}{\|(\zeta^l,\xi^l)\|}\overset{\rm (a)}{\in} {\rm epi}\, f_\infty\overset{\rm (b)}{=}{\rm epi}\,\sigma_{{\rm dom}\, f^*},
\end{align}
where (a) follows from \cite[Proposition 2.1.4]{Auslender-Alfred-Teboulle-2006}, and (b) follows from \eqref{infty-sigma}.
Then, it holds that
\begin{align}
\ell = \lim_{l\rightarrow \infty} G(\zeta^l,\xi^l)&= \lim_{l\rightarrow \infty}\frac{-\beta \zeta^l + \xi^l +f^*(\beta)}{\sqrt{1+\|(\zeta^l,\xi^l)\|^2}}\notag \\
&=\lim_{l\rightarrow \infty}\frac{-\langle(\beta,-1), (\zeta^l,\xi^l)/\|(\zeta^l,\xi^l)\|\rangle  +f^*(\beta)/\|(\zeta^l,\xi^l)\|}{\sqrt{1/\|(\zeta^l,\xi^l)\|^2+1}}\notag \\
&= - \beta \zeta^*+\xi^*,  \label{thm-frf-1d-f-10-2}
\end{align}
where the last equality follows from \eqref{thm-frf-1d-f-10-1}.
 Note that $\|(\zeta^*,\xi^*)\|=1$.
If $\zeta^*=0$, then $\xi^*=1$ (since $\xi^*\ge 0$ in view of \eqref{thm-frf-1d-f-10-1} and the fact that $f_\infty \geq 0$ from Lemma \ref{lemma-nonegative}); it then follows from \eqref{thm-frf-1d-f-10-2} that
\begin{align*}
\ell=\lim_{l\rightarrow \infty} G(\zeta^l,\xi^l)=- \beta \zeta^*+\xi^*=\xi^*=1>0.
\end{align*}
On the other hand, if $\zeta^*\neq 0$, we can choose small $\varepsilon>0$ such that $\beta+\varepsilon \zeta^*/|\zeta^*|\in {\rm dom}\, f^*$ since $\beta\in {\rm int}({\rm dom}\, f^*)$. Thus, we deduce that
\begin{align*}
\ell=\lim_{l\rightarrow \infty} G(\zeta^l,\xi^l)&= - \beta \zeta^*+\xi^*= -( \beta+\varepsilon \zeta^*/|\zeta^*|) \zeta^*+\xi^*+\varepsilon |\zeta^*|\notag \\
& \geq  \xi^*-\sigma_{{\rm dom}\, f^*}(\zeta^*) +\varepsilon |\zeta^*| \geq \varepsilon |\zeta^*|>0,
\end{align*}
where the first inequality holds because $\beta+\varepsilon \zeta^*/|\zeta^*|\in {\rm dom}\, f^*$,
and the second inequality follows from \eqref{thm-frf-1d-f-10-1}.
The last two displays together imply that \eqref{thm-frf-1d-f-10-0} holds.

Let $\zeta^k\coloneqq v^k_x/v^k_t\in {\rm dom}\, f$. If $|\zeta^k-\bar{\beta}|\geq \Delta$, then we have
\begin{align}
\frac{\|{\bm w}^k-{\bm v}^k\|}{\|{\bm w}^k-{\bm u}^k\|}\overset{\rm (a)}{\geq}& \frac{\|{\bm w}^k-{\bm v}^k\|}{\|{\bm v}^k\|}=\frac{\langle {\bm z}^*, {\bm v}^k\rangle}{\|{\bm z}^*\| \|{\bm v}^k\|}
\overset{\rm (b)}{=}\frac{t^*}{\|{\bm z}^*\|}\frac{- \beta v^k_x+f^*(\beta)v^k_t+v^k_t(f(v^k_x/v^k_t)+\varsigma^k)}{\sqrt{(v^k_x)^2+(v^k_t)^2+(v^k_t)^2( f(v^k_x/v^k_t)+\varsigma^k)^2}}\notag \\
\overset{\rm (c)}{=}&\frac{t^*}{\|{\bm z}^*\|}\frac{- \beta \zeta^k+f^*(\beta)+(f(\zeta^k)+\varsigma^k)}{\sqrt{(\zeta^k)^2+1+ (f(\zeta^k)+\varsigma^k)^2}}
\overset{\rm (d)}{\geq} \frac{t^*}{\|{\bm z}^*\|}\inf G>0, \label{thm-frf-1d-f-11}
\end{align}
where (a) follows from the facts that $\|{\bm v}^k\|\neq0$ and $\|{\bm w}^k-{\bm u}^k\|\leq \|{\bm w}^k\|\leq \|{\bm v}^k\|$ since $\{{\bm z}^*\}^{\bot}$ and $\mathcal{F}^e_{\beta}$ contain the origin, (b) follows from the definition of ${\bm v}^k$ and ${\bm z}^*$,
(c) follows upon recalling that $\zeta^k=v^k_x/v^k_t$ and $v^k_t>0$, and (d) follows from \eqref{thm-frf-1d-f-10} thanks to the facts that $(\zeta^k,f(\zeta^k)+\varsigma^k)\in {\rm epi}\, f$ and $|\zeta^k-\bar{\beta}|\geq \Delta$.

Next, we  discuss the case $|\zeta^k-\bar\beta|< \Delta$.  In this case, we have
\begin{align}\label{thm-frf-1d-f-12}
&\|{\bm w}^k-{\bm v}^k\|^{\frac{1}{2}}
\stackrel{\rm (a)}{=}\sqrt{\frac{t^* v^k_t}{\|{\bm z}^*\|}}(h_1(\zeta^k)+\varsigma^k)^{\frac{1}{2}}
\stackrel{\rm (b)}{\geq} \sqrt{\frac{t^* v^k_t}{2\|{\bm z}^*\|}}(h_1(\zeta^k)^{\frac{1}{2}}+(\varsigma^k)^{\frac{1}{2}})\notag\\
\stackrel{\rm (c)}{\geq} & \sqrt{\frac{t^* v^k_t}{2\|{\bm z}^*\|}} \left(c_2 h_2(\zeta^k)+\sqrt{v^k_t/\eta} \varsigma^k \right)
 =\sqrt{\frac{t^*}{2\|{\bm z}^*\|v^k_t}} \left(c_2 v^k_th_2(\zeta^k)+\sqrt{v^k_t/\eta} v^k_t \varsigma^k \right) \notag\\
\stackrel{\rm (d)}{\geq} &\sqrt{\frac{t^*}{2\|{\bm z}^*\|v^k_t}}\min\left\{c_2, \sqrt{v^k_t/\eta}\right\}  \|{\bm w}^k-{\bm u}^k\|
\stackrel{\rm (e)}{\geq} \min \left\{\frac{c_2 \sqrt{t^*}}{\sqrt{2\|{\bm z}^*\|\eta}}, \frac{\sqrt{t^*}}{\sqrt{2\|{\bm z}^*\|\eta}}\right\} \|{\bm w}^k-{\bm u}^k\| ,
\end{align}
where (a) follows from \eqref{thm-frf-1d-f-3},
(b) follows by the fact $\sqrt{a+b}\geq (\sqrt{a}+\sqrt{b})/\sqrt{2}$ for any $a,b\geq 0$,
(c) follows from \eqref{thm-frf-1d-f-8} and $v^k_t \varsigma^k \leq \eta$ (see \eqref{Caseiivk}),
(d) follows from \eqref{thm-frf-1d-f-4},
and (e) holds because $0<v^k_t\leq \eta$ as ${\bm v}^k\in B(\eta)$.
Then we can see from \eqref{thm-frf-1d-f-11} and  \eqref{thm-frf-1d-f-12} that \eqref{thm-frf-1d-f-1} is satisfied.

Therefore, according to Lemma \ref{lemma-gamma}, we know that $\gamma_{{\bm z}^*, \eta}\in (0,\infty]$. Invoking Theorem \ref{thm-frf}, we obtain the desired error bound.
\end{proof}

Combining Theorem \ref{thm-frf-1d-f}, Lemma \ref{lemma-frf} and Theorem \ref{thm-frf}, we obtain the  following $\mathds{1}$-FRF concerning $\mathcal{F}^e_{\beta}$.
\begin{corollary}[$\mathds{1}$-FRF  concerning $\mathcal{F}^e_{\beta}$]\label{coro-1frf-1d-f}
Let $\mathcal{F}^e_\beta={\rm epi}\, f^\pi \cap \{{\bm z}^*\}^\bot$ be the one-dimensional face with ${\bm z}^*=t^*(-\beta,f^*(\beta), 1)$, $\beta\in {\rm int}({\rm dom}\, f^*)\setminus f'(W)$ and $t^*>0$, where $f'(W)\coloneqq \{ v : v= f'(w), w\in W\}$ with $W\subset  {\rm int}({\rm dom}\,f)$  given in Assumption \ref{assu-blanket}.
 Let $\kappa_{{\bm z}^*,  t} \coloneqq \max\{2t^{\frac{1}{2}}, 2\gamma_{{\bm z}^*,t}^{-1}\} $ be defined in Theorem \ref{thm-frf-1d-f}, which is a nondecreasing function with respect to $t$ by the definition.
Then, the $\mathds{1}$-FRF $\psi_{{\rm epi}\, f^\pi,{\bm z}^*}: \R_+\times \R_+\rightarrow \R_+$ for ${\rm epi}\, f^\pi$ and ${\bm z}^*$ is given by
\begin{align*}
\psi_{{\rm epi}\, f^\pi,{\bm z}^*}(\epsilon,t)\coloneqq \max\{\epsilon,\epsilon/\|{\bm z}^*\|\}+\kappa_{{\bm z}^*,t}(\epsilon+\max\{\epsilon,\epsilon/\|{\bm z}^*\|\})^{\frac{1}{2}}.
\end{align*}
In particular, in view of Definition \ref{def-holder-1frf}, the above $\mathds{1}$-FRF is H\"{o}lderian of exponent $1/2$.
\end{corollary}

\subsection{$\mathds{1}$-FRF concerning $\mathcal{F}^e_{\rm bd}$}

In this section, we study the $\mathds{1}$-FRF concerning the  face $\mathcal{F}^e_{\rm bd}$.
Recall that we are under the blanket Assumption~\ref{assu-blanket} and recall from  Proposition  \ref{prop-e-face}(ii) that
\begin{align}\label{face-2d}
\mathcal{F}^e_{\rm bd}={\rm epi}\, f^\pi \cap \{{\bm z}^*\}^\bot={\rm epi}\, f^\pi \cap \{t^*(-\beta,r^*/t^*,1)\}^\bot= \{(x,0,f_\infty(x)): f_{\infty}(x)= x\beta\},
\end{align}
where ${\bm z}^*=t^*(-\beta,r^*/t^*,1) \in {\rm bdry}(({\rm epi}\, f^\pi)^*)$ with
\begin{align}
  \beta\in {\rm bdry}({\rm dom}\, f^*)\cap {\rm dom}\, f^*,\ \  r^*\ge t^*f^*(\beta),\ \  t^*>0. \label{def-z-beta-r}
\end{align}
We will show that, for those ${\bm z}^*=t^*(-\beta,r^*/t^*,1)$ given as in \eqref{def-z-beta-r} with $r^*> t^*f^*(\beta)$, an error bound in the form of \eqref{thm-frf-2} in Theorem \ref{thm-frf} holds for $\mathfrak{g}=|\cdot|$.
First, we present an example to illustrate that for those ${\bm z}^*$ given as in \eqref{def-z-beta-r} with $r^*= t^*f^*(\beta)$, \eqref{thm-frf-2} in Theorem \ref{thm-frf} can fail to hold for $\mathfrak{g}=|\cdot|$.

\begin{example}
Let $f(x)=\exp(x)$.
One can see from Example~\ref{example-exp} that ${\rm bdry}({\rm dom}\, f^*)\cap {\rm dom}\, f^*=\{0\}$ and $f^*(0)=0$. Thus, setting in \eqref{def-z-beta-r} $\beta = 0$, $t^* > 0$ and $r^* = t^* f^*(0)= 0$, we consider ${\bm z}^* = (0,0,t^*)$, and we see from \eqref{face-2d} that $\mathcal{F}^e_{\rm bd}= {\rm epi}\, f^\pi \cap \{{\bm z}^*\}^\bot=\{(x,0,f_\infty(x)): f_\infty(x) = 0\}=\{(x,0,0) : x\leq 0\}$.
It was established in \cite[Remark~4.14]{Lindstrom-Lourenco-Pong-2023} that an error bound of the form \eqref{thm-frf-2} cannot hold for $\{{\rm epi}\, f^\pi, \{{\bm z}^*\}^\bot\}$.
\end{example}

\begin{proposition}\label{prop-R-d}
Let $\beta$ be defined in \eqref{def-z-beta-r} and define $R_\beta:\R^2\rightarrow \R$ as
\begin{align}\label{prop-R-d-0}
R_\beta(x,\xi)\coloneqq  \xi + f^*(\beta)- x\beta.
\end{align}
Define
\begin{align}\label{prop-def-varpi}
\varpi\coloneqq \max\{\max\{f^*(y): y\in{\rm bdry}({\rm dom}\, f^*)\cap {\rm dom}\, f^*\},0\}.
\end{align}
Then it holds that
\begin{align}\label{prop-R-d-2-1}
\infty>\varpi\geq f^*(y) \quad \forall\, y \in {\rm bdry}({\rm dom}\, f^*)\cap {\rm dom}\, f^*.
\end{align}
Moreover, there exists $c>0$ such that
\begin{align}\label{prop-R-d-1}
R_\beta(x,\xi)+\varpi\geq c\, {\rm d}((x,\xi), {\rm epi}\, f_\infty) \quad \forall\, (x,\xi)\in {\rm epi}\, f.
\end{align}
\end{proposition}
\begin{proof}
From the definition of $\varpi$ in \eqref{prop-def-varpi} and the fact that ${\rm bdry}({\rm dom}\, f^*)$ contains at most two points, we see that \eqref{prop-R-d-2-1} holds.

Next, observe that  for any  $(x,\xi)\in {\rm epi}\, f$,
\begin{align*}
{\rm d}((x,\xi), {\rm epi}\, f_\infty)\leq \|(x,\xi)-(x,f(x))\| + d((x,f(x)),{\rm epi}\, f_\infty) = \xi-f(x)+  d((x,f(x)),{\rm epi}\, f_\infty),
\end{align*}
and
\begin{align*}
R_\beta(x,\xi)= \xi-f(x)+f(x)+f^*(\beta)- x\beta= \xi-f(x)+R_\beta(x,f(x)).
\end{align*}
Thus, to establish \eqref{prop-R-d-1}, it suffices to show that there exists $c > 0$ such that
\begin{align}\label{prop-R-d-2}
R_\beta(x,f(x))+\varpi \geq c  {\rm d}((x, f(x)), {\rm epi}\, f_\infty)\quad \forall\, x\in{\rm dom}\,f.
\end{align}
To this end, we first note from Lemma \ref{lemma-dom-fstar}(ii) that ${\rm dom}\, f^*$ must take one of the following forms:
\[
[-a,\infty),\, (-\infty, b],\, [-a, b),\, (-a,b],\, [-a,b],
\]
for some $a,b \geq 0$ with $a>-b$. We prove \eqref{prop-R-d-2} for ${\rm dom}\, f^*$ as above separately as follows.

\begin{itemize}
\item \underline{${\rm dom}\, f^*=[-a,\infty)$}.
In this case, we see that ${\rm bdry}({\rm dom}\, f^*)\cap {\rm dom}\, f^*=\{-a\}$, i.e., $\beta = -a$.
Notice that $R_{-a}(x, f(x))>0$ for any $x$ (thanks to \eqref{young-2}),
$x\mapsto  R_{-a}(x, f(x))$ is lower semicontinuous (thanks to the lower semicontinuity of $f$)
and ${\rm d}((x, f(x)), {\rm epi}\, f_\infty)\leq |x|$ for any $x\in {\rm dom}\, f$ (thanks to facts that $\{(0,r): r\geq 0\}\subseteq  {\rm epi}\, f_\infty$ and that $f$ is nonnegative). In view of this,
it suffices to show that \eqref{prop-R-d-2} holds for all $x\in {\rm dom}\, f$ with sufficiently large $|x|$.
To this end, we compute
\begin{align}\label{prop-R-d-3}
&f_\infty(x)\overset{\rm (a)}{=}\sigma_{{\rm dom }\,f^*}(x)=\sup\{ x\cdot y: y\in [-a,\infty) \}=
\begin{cases}
|a x| & \mbox{if $x\leq 0$}, \\
\infty & \hbox{if $x>0$},
\end{cases}
\end{align}
where (a) follows from \eqref{infty-sigma}.
Then
  \[
 f(x)+\varpi-f_\infty(x)= f(x)+\varpi-(-ax)\geq f(x)+f^*(-a)-(-a) x> 0 \quad \forall\, x\leq 0,
  \]
  where the first equality follows from \eqref{prop-R-d-3},
   the first inequality holds because $\varpi\geq f^*(-a)$,
   and the last inequality follows from  \eqref{young-2} thanks to the fact that $-a\in{\rm bdry}({\rm dom}\, f^*)\cap {\rm dom}\, f^*$.
The above display implies that $(x,f(x)+\varpi)\in {\rm epi}\,f_\infty$ whenever $x\leq 0$ and $x\in{\rm dom}\, f$, which means that
  \begin{align*}
  R_{-a}(x,f(x)) >0 = {\rm d}((x,f(x)+\varpi), {\rm epi}\, f_\infty)\geq {\rm d}((x,f(x)), {\rm epi}\, f_\infty)-\varpi
  \end{align*}
  whenever $x\leq 0$ and $x\in{\rm dom}\, f$,
  where the last inequality follows from the triangle inequality.

Next, from \eqref{prop-R-d-3} and Lemma \ref{lemma-infty},  we know that $\infty=f_\infty(1)=\lim_{x\rightarrow \infty} f(x)/x=\lim_{x\rightarrow \infty}(f(x)+f^*(-a)-(-ax))/x$.
 Then, there exists positive constant $M$ such that $f(x)+f^*(-a)-(-ax))\ge x$ whenever $x\ge M$.
 Hence, for any $x\in {\rm dom}\,f$ with $x\ge M$, we have
 \begin{align*}
 R_{-a}(x,f(x))+\varpi
 \geq f(x)+f^*(-a)-(-a x) \geq  x= {\rm d}((x,f(x)), {\rm epi}\, f_\infty),
 \end{align*}
 where the first inequality holds because $\varpi\geq 0$ and the last equality follows from the observation that ${\rm d}((x,f(x)), {\rm epi}\, f_\infty)=x$ whenever $x>0$ and $x\in {\rm dom}\,f$, thanks to \eqref{prop-R-d-3}.

 Combining the last two displays, we have
\begin{align*}
 R_{-a}(x,f(x))+\varpi\geq {\rm d}((x,f(x)), {\rm epi}\, f_\infty) \quad \mbox{$\forall\, |x|\geq M$ and $x\in {\rm dom}\, f$.}
 \end{align*}
Therefore, we deduce that \eqref{prop-R-d-2} is satisfied with $\beta=-a$.

\item \underline{${\rm dom}\, f^*=(-\infty,b]$}.  It follows from a similar argument as in the case for  ${\rm dom}\, f^*=[-a,\infty)$. We omit the details for brevity.

\item \underline{${\rm dom}\, f^*=[-a,b)$}.
In this case, we see that $\beta = -a$. Moreover, from \eqref{infty-sigma}, we see that ${\rm dom}\, f_\infty =\R$,
 which also implies that  ${\rm dom}\, f=\R$ thanks to \cite[Proposition 2.1.6]{Auslender-Alfred-Teboulle-2006}.\footnote{From \cite[Definition 2.5.1]{Auslender-Alfred-Teboulle-2006}, we have that ${\rm epi}\, f_\infty=({\rm epi}\, f)_\infty$, where $C_\infty$ denotes the asymptotic cone of the set $C$ (see \cite[Definition 2.1.2]{Auslender-Alfred-Teboulle-2006}).  From this fact and \cite[Proposition 2.1.6]{Auslender-Alfred-Teboulle-2006}, we deduce that ${\rm epi}\, f={\rm epi}\, f + ({\rm epi}\, f)_\infty={\rm epi}\, f + {\rm epi}\, f_\infty$, which implies that ${\rm dom}\, f=\R$ since ${\rm dom}\, f_\infty=\R$.}
 Similar to the discussion in ${\rm dom}\, f^*=[-a,\infty)$, it suffices to show that  \eqref{prop-R-d-2} holds with $\beta=-a$ for all sufficiently large $|x|$.

Direct computation shows that
\begin{align}\label{prop-R-d-5}
&f_\infty(x)\overset{\rm (a)}{=}\sigma_{{\rm dom }\,f^*}(x)=\sup\{ x\cdot y: y\in [-a,b) \}=
\begin{cases}
|a x| & \mbox{if $x\leq 0$}, \\
bx & \hbox{if $x>0$},
\end{cases}
\end{align}
and
\begin{align}
&{\rm d}((x,f(x)), {\rm epi}\, f_\infty)\leq
\begin{cases}
\max\{|a x|-f(x) , 0\}& \mbox{if $x\leq 0$}, \\
\max\{bx-f(x) , 0\} & \hbox{if $x>0$},
\end{cases}\label{prop-R-d-6}
\end{align}
where (a) follows from \eqref{infty-sigma}.
 This leads to
 \begin{align*}
 f(x)+\varpi- f_\infty(x)=f(x)+\varpi- (-ax)\geq f(x)+f^*(-a)-(-ax)>0 \quad \forall\, x\leq 0,
 \end{align*}
 where the first equality follows from \eqref{prop-R-d-5},
   the first inequality holds because $\varpi\geq f^*(-a)$,
   and the last inequality follows from  \eqref{young-2} since $-a\in{\rm bdry}({\rm dom}\, f^*)\cap {\rm dom}\, f^*$.
The last display shows that ${\rm d}((x,f(x)+\varpi), {\rm epi}\, f_\infty)=0$ whenever $x\leq 0$ (recall that ${\rm dom}\,f=\R$). Then, we have
\[
R_{-a}(x,f(x))>0= {\rm d}((x,f(x)+\varpi), {\rm epi}\, f_\infty)\geq  {\rm d}((x,f(x)), {\rm epi}\, f_\infty)-\varpi \quad \forall\, x\leq 0,
\]
  where the last inequality follows from the triangle inequality.

Now, we consider $x\geq 0$. Since $b=f_\infty(1)=\lim_{x\rightarrow \infty} f(x)/x$ thanks to \eqref{prop-R-d-5} and Lemma \ref{lemma-infty}, there exist $\epsilon \in (0,(a+b)/2)$ (note that $a+b>0$) and $M>0$ such that
\begin{align}\label{prop-R-d-7}
f(x)\geq (b-\epsilon) x \quad   \forall\, x\geq M.
\end{align}
From \eqref{prop-R-d-7} and the fact that $b+a>2\epsilon$, by further increasing $M$ if necessary, we see that
\begin{align}\label{prop-R-d-8}
(b+a-2\epsilon)x+f^*(-a) \geq 0 \ \ {\rm and}\ \ \epsilon x \geq \max\{bx-f(x) , 0\} \quad \mbox{whenever $x\geq M$.}
\end{align}
Then, it follows that
\begin{align*}
R_{-a}(x,f(x))+\varpi &\overset{\rm (a)}{\geq}
f(x)+f^*(-a)-(-ax) \overset{\rm (b)}{\geq} (b+a-\epsilon)x+f^*(-a) \\
 &\overset{\rm (c)}{\geq} \epsilon x \overset{\rm (d)}{\geq}{\rm d}((x,f(x)), {\rm epi}\, f_\infty) \quad \forall\, x\geq M,
\end{align*}
where (a) follows from the fact that $\varpi\geq 0$ and the definition of $R_{-a}$ in  \eqref{prop-R-d-0},
(b) follows from \eqref{prop-R-d-7},
(c) follows from first inequality in \eqref{prop-R-d-8},
and (d) follows from \eqref{prop-R-d-6} the second inequality in \eqref{prop-R-d-8}.

Therefore, we deduce that \eqref{prop-R-d-2}  is satisfied with $\beta=-a$.

\item \underline{${\rm dom}\, f^*=(-a,b]$}. For this case, \eqref{prop-R-d-2} follows from a similar argument
as in ${\rm dom}\, f^*=[-a,b)$. We omit the details for brevity.

\item \underline{${\rm dom}\, f^*=[-a,b]$}.
In this case, the $\beta$ under consideration can be $-a$ or $b$.
In view of \eqref{infty-sigma}, the boundedness of ${\rm dom}\, f^*$ implies that ${\rm dom}\, f_\infty={\rm dom}\, f=\R$.
In addition,  observe that for any $x\in \R$, we have
\begin{align*}
\infty &\overset{\rm (a)}{>}f(x)+\varpi-f_\infty(x)=\inf \{f(x)+\varpi- yx: y\in [-a,b]\}\notag\\
&  =   \inf\{f(x)+\varpi- y x: y\in \{-a,b\}\}    \overset{\rm (b)}{\geq}     \inf\{f(x)+f^*(y)- y x: y\in \{-a,b\}\}  \geq 0,
\end{align*}
where (a) follows from the fact that ${\rm dom}\, f_\infty={\rm dom}\, f=\R$, (b) follows from the facts that $\varpi\geq f^*(y)$ for any $y\in {\rm bdry}({\rm dom}\, f^*)$ thanks to \eqref{prop-R-d-2-1}, and the last inequality follows from \eqref{young}.
The last display together with the fact that ${\rm dom}\, f_\infty={\rm dom}\, f=\R$ implies that ${\rm epi}\, f+(0,\varpi) \subseteq {\rm epi}\, f_\infty$ and hence we see that
\[
0={\rm d}((x,f(x)+\varpi), {\rm epi}\, f_\infty)\geq  {\rm d}((x,f(x)), {\rm epi}\, f_\infty)-\varpi \quad \forall\, x\in \R,
\]
where the last inequality follows from the triangle inequality.

On the other hand,   it always holds that $R_\beta(x,f(x))>0$ for any $x$ thanks to the fact that $\beta\in {\rm bdry}({\rm dom}\, f^*)$ and \eqref{young-2}.
Therefore, we see that \eqref{prop-R-d-2} holds with $c = 1$.
\end{itemize}
\end{proof}

\begin{lemma}\label{lemma-k0}
Let $\mathcal{F}^e_{\rm bd}={\rm epi}\, f^\pi \cap \{{\bm z}^*\}^\bot$ be a proper exposed  face defined in \eqref{face-2d}, where ${\bm z}^*=t^*(-\beta, r^*/t^*, 1)$ with
 $\beta\in {\rm bdry}({\rm dom}\, f^*)\cap {\rm dom}\, f^*$, $r^*>t^*f^*(\beta)$ and $t^*>0$.
 Let $\hat{{\bm z}}\coloneqq t^*(-\beta,1)$.
  Then the following statements hold.
 \begin{enumerate}[\rm (i)]
 \item $\hat{{\bm z}}\in ({\rm epi}\, f_\infty)^*$.
\item  Let $\mathcal{K}^\theta$ be given in \eqref{per-cone} and $\mathcal{S}_{\hat{{\bm z}}}\coloneqq \{(x,0,r) : \langle(x,r), \hat{{\bm z}}\rangle=0\}$. Then we have for any ${\bm v}=(v_x,0,v_r)\in \mathcal{K}^\theta$ that
 \begin{equation}\label{lemma-k0-1}
\begin{split}
&{\rm d}({\bm v},\mathcal{F}^e_{\rm bd})={\rm d}((v_x,v_r),{\rm epi}\, f_\infty \cap \{\hat{{\bm z}}\}^\bot), \\
&{\rm d}({\bm v},\mathcal{S}_{\hat{{\bm z}}})={\rm d}((v_x,v_r),\{\hat{{\bm z}}\}^\bot),\,\,\,\,\,\,\,\,\,\,\,
{\rm d}({\bm v},\mathcal{K}^\theta)={\rm d}((v_x,v_r),{\rm epi}\, f_\infty )=0.
\end{split}
\end{equation}
\end{enumerate}
\end{lemma}
\begin{proof}
According to \eqref{per-cone} and \eqref{dual-cone}, we see that
$\langle (-t^*\beta,t^*), (x,s)\rangle =\langle (-t^*\beta,r^*,t^*), (x,0,s)\rangle \geq 0$  for any $(x,s)\in {\rm epi}\, f_\infty$, which establishes item (i).
Next, from Proposition  \ref{prop-e-face}(ii), we have
\begin{align*}
\mathcal{F}^e_{\rm bd}&={\rm epi}\, f^\pi \cap \{{\bm z}^*\}^\bot=\mathcal{K}^\theta \cap \{{\bm z}^*\}^\bot=\mathcal{K}^\theta \cap \mathcal{S}_{\hat{{\bm z}}},
\end{align*}
where the last equality follows from the fact that the second component of elements in $\mathcal{K}^\theta$ are zero. From the above display, the definition of $\mathcal{S}_{\hat{{\bm z}}}$ and the fact that ${\bm v}\in \mathcal{K}^\theta$, we obtain the equalities in \eqref{lemma-k0-1}.
\end{proof}

Now, we are ready to prove the following Lipschitzian error bound.
\begin{theorem}\label{thm-frf-2d-f}
Let $\mathcal{F}^e_{\rm bd}={\rm epi}\, f^\pi \cap \{{\bm z}^*\}^\bot$ be a proper exposed  face defined in \eqref{face-2d}, where ${\bm z}^*=t^*(-\beta, r^*/t^*, 1)$ with
 $\beta\in {\rm bdry}({\rm dom}\, f^*)\cap {\rm dom}\, f^*$, $r^*>t^*f^*(\beta)$ and $t^*>0$.
  Let $\eta>0$ and $\gamma_{{\bm z}^*, \eta}$ be defined in \eqref{thm-frf-1} with $\mathfrak{g}=|\cdot|$ and $({\rm epi}\, f^\pi,\mathcal{F}^e_{\rm bd},{\bm z}^*)$ in place of $(\mathcal{K},\mathcal{F},{\bm z})$. Then we have $\gamma_{{\bm z}^*, \eta}\in(0,\infty]$ and
\[
{\rm d}({\bm q}, \mathcal{F}^e_{\rm bd})\leq \kappa_{{\bm z}^*,\eta} {\rm d}({\bm q}, {\rm epi}\, f^\pi) \quad \forall\, {\bm q}\in\{{\bm z}^*\}^\bot \cap B(\eta),
\]
where $\kappa_{{\bm z}^*,\eta} \coloneqq \max\{2, 2\gamma_{{\bm z}^*,\eta}^{-1}\}$.
\end{theorem}
\begin{proof}
Letting $x^* \coloneqq t^*\beta$ and $\hat{{\bm z}}\coloneqq (-x^*,t^*)$, then we see that $\hat{{\bm z}}\in ({\rm epi}\, f_\infty)^*$ thanks to Lemma \ref{lemma-k0}(i).
Applying Lemma \ref{lemma-1frf}, we see that  ${\rm epi}\, f_\infty$ and $\hat{\bm z}$ admit a $\mathds{1}$-FRF $\psi_{{\rm epi}\, f_\infty, \hat{{\bm z}}}:\R_+\times \R_+\rightarrow \R_+$ given by
\begin{align}\label{thm-frf-2d-f-3-1}
\psi_{{\rm epi}\, f_\infty, \hat{{\bm z}}}(\epsilon,t)=\tau \epsilon,
\end{align}
for some $\tau>0$.

According to Lemma \ref{lemma-gamma}, we need to show that
\begin{align}\label{thm-frf-2d-f-1}
\liminf_{k\rightarrow\infty} \frac{\|{\bm w}^k-{\bm v}^k\|}{\|{\bm w}^k-{\bm u}^k\|}>0
\end{align}
 for any $\bar{{\bm v}}\in \mathcal{F}^e_{\rm bd}$ and   sequence $\{{\bm v}^k\}\subset{\rm bdry}({\rm epi}\, f^\pi)\cap B(\eta)\backslash\mathcal{F}^e_{\rm bd}$ such that
\begin{align}\label{thm-frf-2d-f-2}
\lim_{k\rightarrow \infty}{\bm v}^k=\lim_{k\rightarrow \infty}{\bm w}^k=\bar{{\bm v}} \quad \mbox{with ${\bm w}^k={\rm P}_{\{{\bm z}^*\}^\bot}{\bm v}^k$, ${\bm u}^k={\rm P}_{{\mathcal{F}^e_{\rm bd}}}{\bm w}^k$ and ${\bm w}^k\neq {\bm u}^k$}.
\end{align}
By passing to subsequences if necessary, we will consider the following two  cases.
\begin{enumerate}[\rm (a)]
\item ${\bm v}^k\in \mathcal{K}^\theta$ for all $k$;
\item ${\bm v}^k\in \mathcal{K}^t$ for all $k$.
\end{enumerate}

{\bf Case (a):} We first assume ${\bm v}^k\in \mathcal{K}^\theta$ for all $k$, i.e.
\begin{align*}
{\bm v}^k=(v^k_x,0,v^k_r), \quad \mbox{where $(v^k_x, v^k_r)\in {\rm epi}\, f_\infty$}.
\end{align*}
Since ${\bm w}^k={\rm P}_{\{{\bm z}^*\}^\bot}{\bm v}^k$, we have (recall that $x^* \coloneqq t^*\beta$)
\begin{align}
\|{\bm w}^k-{\bm v}^k\|=\frac{|\langle {\bm z}^*, {\bm v}^k\rangle|}{\|{\bm z}^*\|}=\frac{|\langle (-x^*,r^*,t^*), (v^k_x,0,v^k_r)\rangle|}{\|{\bm z}^*\|}=\frac{| -x^* v^k_x+t^* v^k_r|}{\|{\bm z}^*\|}. \label{thm-frf-2d-f-3}
\end{align}
Let $\mathcal{S}_{\hat{{\bm z}}}\coloneqq {\{(x,0,r) : \langle(x,r), \hat{{\bm z}}\rangle=0\}}$. Then, for all $k$
\begin{align}
\|{\bm w}^k-{\bm u}^k\|\overset{\rm (a)}{\leq} &\|{\bm v}^k-{\rm P}_{\mathcal{F}^e_{\rm bd}}{\bm v}^k\|
\overset{\rm (b)}{=}{\rm d}((v^k_x,v^k_r),{\rm epi}\, f_\infty \cap \{\hat{{\bm z}}\}^\bot)  \notag \\
\overset{\rm (c)}{\leq}& \psi_{{\rm epi}\, f_\infty, \hat{{\bm z}}} ({\rm d}({\bm v}^k,\mathcal{K}^\theta)+ \|\hat{\bm z}\|{\rm d}({\bm v}^k,\mathcal{S}_{\hat{{\bm z}}}), \|{\bm v}^k\| )
\overset{\rm (d)}{=} \psi_{{\rm epi}\, f_\infty, \hat{{\bm z}}} ( \|\hat{\bm z}\|{\rm d}({\bm v}^k,\mathcal{S}_{\hat{{\bm z}}}), \|{\bm v}^k\| ) \notag \\
\overset{\rm (e)}{=}& \tau \|\hat{\bm z}\| {\rm d}({\bm v}^k,\mathcal{S}_{\hat{{\bm z}}})
 {=}
 \tau   \|\hat{\bm z}\| \frac{| -x^* v^k_x+t^*v^k_r|}{\|(-x^*,t^*)\|}
 \overset{\rm (f)}{=}  c_1 \|{\bm w}^k-{\bm v}^k\|,    \label{thm-frf-2d-f-5}
\end{align}
where (a) follows from Lemma \ref{lemma-bound-wu},
(b) follows from \eqref{lemma-k0-1} in Lemma \ref{lemma-k0}(ii) since ${\bm v}^k \in \mathcal{K}^\theta$,
(c) follows from \eqref{lemma-k0-1} and the fact that $\psi_{{\rm epi}\, f_\infty, \hat{{\bm z}}}$ is a $\mathds{1}$-FRF for ${\rm epi}\, f_\infty$ and  $\hat{{\bm z}}$ (see Definition \ref{def-1frf}),
(d) follows from the fact that ${\bm v}^k\in \mathcal{K}^\theta$,
(e) follows from \eqref{thm-frf-2d-f-3-1}, and
(f) follows from \eqref{thm-frf-2d-f-3} upon letting $c_1\coloneqq \tau \|{\bm z}^*\|$ (recall that $\hat{{\bm z}} \coloneqq  (-x^*,t^*)$).
The above display shows that \eqref{thm-frf-2d-f-1} holds.

{\bf Case (b):}  We next assume ${\bm v}^k\in \mathcal{K}^t$ for all $k$; that is,
\begin{align*}
0\neq {\bm v}^k=(v^k_x,v^k_t,v^k_t (f(v^k_x/v^k_t)+\varsigma^k)), \quad v^k_t>0 \,\,\,\,\,\, \forall \, k,
\end{align*}
where $\varsigma^k\geq 0$ if $v^k_x/v^k_t\in {\rm bdry}({\rm dom}\, f)\cap {\rm dom}\, f$  with  $v^k_t\varsigma^k\leq \eta$ since ${\bm v}^k\in B(\eta)$, and $\varsigma^k= 0$ otherwise.

Recall from \eqref{prop-R-d-2-1} that $\varpi=\max\{\max\{f^*(y): y\in  {\rm bdry}({\rm dom}\, f^*)\cap{\rm dom}\, f^*\},0\}<\infty$ and define
\begin{align}\label{thm-def-varpi}
\bar\varpi=\frac{r^*/t^*-f^*(\beta)}{\varpi+r^*/t^*-f^*(\beta)}.
\end{align}
Then $\bar{\varpi}>0$ thanks to $\varpi\geq 0$ and $r^*/t^*-f^*(\beta)>0$.
We can now compute that
\begin{align}
\|{\bm w}^k-{\bm v}^k\|&=\frac{|\langle {\bm v}^k,{\bm z}^*\rangle|}{\|{\bm z}^*\|}=\frac{\langle {\bm v}^k,{\bm z}^*\rangle}{\|{\bm z}^*\|}=\frac{\langle (v^k_x,v^k_t,v^k_t(f(v^k_x/v^k_t)+\varsigma^k)), (-x^*,r^*,t^*)\rangle}{\|{\bm z}^*\|} \notag \\
&=\frac{- v^k_x x^*+v^k_t r^*+ t^* v^k_t (f(v^k_x/v^k_t)+\varsigma^k)}{\|{\bm z}^*\|} \notag\\
&=v^k_t t^* \frac{- (v^k_x/v^k_t) \beta + f^*(\beta)+  f(v^k_x/v^k_t)+\varsigma^k+(r^*/t^*-f^*(\beta))}{\|{\bm z}^*\|}\notag \\
&\overset{\rm (a)}{=}v^k_t \frac{t^*}{\|{\bm z}^*\|}(R_\beta(v^k_x/v^k_t,f(v^k_x/v^k_t)+\varsigma^k)+r^*/t^*-f^*(\beta)) \notag \\
 & = v^k_t \frac{t^*\bar{\varpi}}{\|{\bm z}^*\|}(\bar{\varpi}^{-1}R_\beta(v^k_x/v^k_t,f(v^k_x/v^k_t)+\varsigma^k)+\bar{\varpi}^{-1}(r^*/t^*-f^*(\beta))) \notag \\
 &\overset{\rm (b)}{\geq}v^k_t \frac{t^*\bar{\varpi}}{\|{\bm z}^*\|}(R_\beta(v^k_x/v^k_t,f(v^k_x/v^k_t)+\varsigma^k)+\varpi+r^*/t^*-f^*(\beta)),\label{thm-frf-2d-f-6}
\end{align}
where (a) follows from the definition of $R_\beta$ in \eqref{prop-R-d-0},
 and (b) follows from \eqref{thm-def-varpi} since $R_\beta(x,\xi)\geq 0$  for any $(x,\xi)\in {\rm epi}\, f$ thanks to \eqref{young}.
In addition, notice that
\begin{align}
&{\rm d}({\bm v}^k,\mathcal{K}^\theta)&\notag\\
=&\inf_{(x,r)\in {\rm epi}\, f_\infty} (|v^k_x-x|^2+(v^k_t)^2+(v^k_t (f(v^k_x/v^k_t)+\varsigma^k)-r)^{2})^{\frac{1}{2}} \notag \\
\overset{\rm (a)}{=}&v^k_t\inf_{(x,r)\in {\rm epi}\, f_\infty} (|v^k_x/v^k_t-x|^2+1+(f(v^k_x/v^k_t)+\varsigma^k-r)^2)^{\frac{1}{2}}  \notag \\
\leq& v^k_t ({\rm d}((v^k_x/v^k_t,f(v^k_x/v^k_t)+\varsigma^k),{\rm epi}\, f_\infty)+1) \notag \\
 \overset{\rm (b)}{\leq}  & v^k_t c^{-1}((R_\beta(v^k_x/v^k_t,f(v^k_x/v^k_t)+\varsigma^k)+\varpi)+c) \notag \\
\overset{\rm (c)}{=}   & v^k_t c^{-1}\frac{\|{\bm z}^*\|}{t^* \bar\varpi} \frac{t^*\bar\varpi}{\|{\bm z}^*\|}\left((R_\beta(v^k_x/v^k_t,f(v^k_x/v^k_t)+\varsigma^k)+\varpi)+c\frac{r^*/t^*-f^*(\beta)}{r^*/t^*-f^*(\beta)}\right) \notag \\
\leq &  v^k_t  \max\left\{c^{-1},\frac{1}{r^*/t^*-f^*(\beta)}\right\}\frac{\|{\bm z}^*\|}{t^* \bar\varpi} \frac{t^* \bar\varpi}{\|{\bm z}^*\|}(R_\beta(v^k_x/v^k_t,f(v^k_x/v^k_t)+\varsigma^k)+\varpi+r^*/t^*-f^*(\beta)) \notag \\
 \overset{\rm (d)}{\leq}& c_2 \|{\bm w}^k-{\bm v}^k\|, \label{thm-frf-2d-f-7}
\end{align}
where (a) follows from the fact that ${\rm epi}\, f_\infty$ is a cone and $v^k_t >0$,
(b) follows from  \eqref{prop-R-d-1},
(c) follows from the facts that $t^*>0$, $z^*\neq 0$, $r^*/t^*-f^*(\beta)>0$ and $\bar\varpi>0$,
and (d) follows from \eqref{thm-frf-2d-f-6} upon letting $c_2\coloneqq \max\left\{c^{-1},\frac{1}{r^*/t^*-f^*(\beta)}\right\}\frac{\|{\bm z}^*\|}{t^* \bar\varpi}$.
Recall that $\hat{{\bm z}}\coloneqq(-x^*,t^*)\in  ({\rm epi}\, f_\infty)^*$ and $\mathcal{S}_{\hat{{\bm z}}}\coloneqq {\{(x,0,r) : \langle(x,r), \hat{{\bm z}}\rangle=0\}}$.
Then
\begin{align}
&{\rm d}({\bm v}^k,\mathcal{S}_{\hat{{\bm z}}})\notag\\
=& {\rm d}((v^k_x, v^k_t, v^k_t (f(v^k_x/v^k_t)+\varsigma^k)), \{(x,0,r) : \langle(x,r), (-x^*,t^*)\rangle=0\}) \notag \\
=& \sqrt{{\rm d}((v^k_x,  v^k_t (f(v^k_x/v^k_t)+\varsigma^k)),\{ (-x^*,t^*)\}^\bot)^2+(v^k_t)^2} \notag\\
\leq& \frac{| -x^*v^k_x+ t^*(v^k_t(f(v^k_x/v^k_t)+\varsigma^k))|}{\|\hat{{\bm z}}\|} +v^k_t \notag \\
\overset{\rm (a)}{=}& v^k_t \frac{t^*}{\|\hat{{\bm z}}\|} \left( |-(v^k_x/v^k_t) \beta+ f(v^k_x/v^k_t)+\varsigma^k)| + \frac{\|\hat{{\bm z}}\|}{t^*} \right) \notag \\
\overset{\rm (b)}{=}& v^k_t \frac{t^*}{\|\hat{{\bm z}}\|} \left( |R_\beta(v^k_x/v^k_t,f(v^k_x/v^k_t)+\varsigma^k)-f^*(\beta)| + \frac{\|\hat{{\bm z}}\|}{t^*} \right)\notag \\
\overset{\rm (c)}{\leq}& v^k_t \frac{t^*}{\|\hat{{\bm z}}\|} \left( R_\beta(v^k_x/v^k_t,f(v^k_x/v^k_t)+\varsigma^k)+\varpi+ |f^*(\beta)| +\frac{\|\hat{{\bm z}}\|}{t^*}\right) \notag \\
\overset{\rm (d)}{=}& v^k_t \frac{\|{\bm z}^*\|}{\|\hat{{\bm z}}\| \bar\varpi}\frac{t^* \bar\varpi}{\|{\bm z}^*\|} \left( (R_\beta(v^k_x/v^k_t,f(v^k_x/v^k_t)+\varsigma^k)+\varpi)+ (|f^*(\beta)| +\|\hat{{\bm z}}\|/t^*)\frac{r^*/t^*-f^*(\beta)}{r^*/t^*-f^*(\beta)}\right) \notag \\
 \leq& \max\left\{\frac{|f^*(\beta)| + \|\hat{{\bm z}}\|/t^*}{r^*/t^*-f^*(\beta)}, 1\right\} \frac{\|{\bm z}^*\|}{\|\hat{{\bm z}}\|\bar\varpi}\frac{v^k_t t^* \bar\varpi}{\|{\bm z}^*\|} \left( R_\beta(v^k_x/v^k_t,f(v^k_x/v^k_t)+\varsigma^k)+\varpi+ r^*/t^*-f^*(\beta) \right)  \notag \\
 \overset{\rm (e)}{=}& c_3 \|{\bm w}^k-{\bm v}^k\|,   \label{thm-frf-2d-f-8}
\end{align}
where (a) follows from the fact that $\beta=x^*/t^*$ and $t^*>0$,
(b) follows from the definition of $R_\beta$ in \eqref{prop-R-d-0},
(c) follows from the triangle inequality and the facts that $R_\beta$ is nonnegative and $\varpi\geq 0$,
(d) follows from the facts that ${\bm z}^*\neq 0$, $r^*/t^*-f^*(\beta)>0$ and $\bar\varpi>0$,
and (e) follows from \eqref{thm-frf-2d-f-6} and upon letting $c_3\coloneqq \max\left\{\frac{|f^*(\beta)| + \|\hat{{\bm z}}\|/t^*}{r^*/t^*-f^*(\beta))}, 1\right\}\frac{\|{\bm z}^*\|}{\|\hat{{\bm z}}\| \bar\varpi}$.

Letting $\bar{{\bm v}}^k\coloneqq {\rm P}_{\mathcal{K}^\theta}{\bm v}^k$, we deduce that
\begin{align}
\|{\bm w}^k-{\bm u}^k\| & \leq \|{\bm v}^k-{\rm P}_{\mathcal{F}^e_{\rm bd}}{\bm v}^k\|= \|{\bm v}^k- \bar{{\bm v}}^k +\bar{{\bm v}}^k-{\rm P}_{\mathcal{F}^e_{\rm bd}}\bar{{\bm v}}^k+{\rm P}_{\mathcal{F}^e_{\rm bd}}\bar{{\bm v}}^k -{\rm P}_{\mathcal{F}^e_{\rm bd}}{\bm v}^k\| \notag\\
 & \leq 2\|{\bm v}^k-\bar{{\bm v}}^k\|+\|\bar{{\bm v}}^k-{\rm P}_{\mathcal{F}^e_{\rm bd}}\bar{{\bm v}}^k\|, \label{thm-frf-2d-f-9}
\end{align}
where the first inequality follows from  Lemma \ref{lemma-bound-wu},
and the last inequality follows from the triangle inequality and the nonexpansiveness of the projection operator.
It follows that for all $k$
\begin{align}
&\|{\bm w}^k-{\bm u}^k\|-2\|{\bm v}^k-\bar{{\bm v}}^k\| \overset{\rm (a)}{\leq} \|\bar{{\bm v}}^k - {\rm P}_{\mathcal{F}^e_{\rm bd}}\bar{{\bm v}}^k\|\overset{\rm (b)}{=}  {\rm d}((\bar{v}^k_x,\bar{v}^k_r),{\rm epi}\, f_\infty \cap \{\hat{{\bm z}}\}^\bot) \notag\\
\overset{\rm (c)}{\leq}& \psi_{{\rm epi}\, f_\infty, \hat{{\bm z}}} ({\rm d}(\bar{{\bm v}}^k,\mathcal{K}^\theta)+ \|\hat{\bm z}\| {\rm d}(\bar{{\bm v}}^k,\mathcal{S}_{\hat{{\bm z}}}), \|\bar{{\bm v}}^k\| )\overset{\rm (d)}{=}   \psi_{{\rm epi}\, f_\infty, \hat{{\bm z}}} (\|\hat{\bm z}\| {\rm d}(\bar{{\bm v}}^k,\mathcal{S}_{\hat{{\bm z}}}), \|\bar{{\bm v}}^k\| )  \notag \\
\overset{\rm (e)}{\leq}  & \psi_{{\rm epi}\, f_\infty, \hat{{\bm z}}} (\|\hat{\bm z}\| {\rm d}({\bm v}^k,\mathcal{K}^\theta)+ \|\hat{\bm z}\| {\rm d}({\bm v}^k,\mathcal{S}_{\hat{{\bm z}}}), \|\bar{{\bm v}}^k\| ) \overset{\rm (f)}{\leq}  \psi_{{\rm epi}\, f_\infty, \hat{{\bm z}}} ((c_2+c_3 )\|\hat{\bm z}\|\|{\bm w}^k-{\bm v}^k\|), \|\bar{{\bm v}}^k\| )  \notag \\
\overset{\rm (g)}{=} & \tau(c_2+c_3)\|\hat{\bm z}\|\|{\bm w}^k-{\bm v}^k\| \notag
\end{align}
where (a) follows from  \eqref{thm-frf-2d-f-9},
(b) follows from \eqref{lemma-k0-1} in Lemma \ref{lemma-k0} since $\bar{{\bm v}}^k\in \mathcal{K}^\theta$,
(c) follows from \eqref{lemma-k0-1} and the fact that $\psi_{{\rm epi}\, f_\infty, \hat{{\bm z}}}$ is a $\mathds{1}$-FRF for ${\rm epi}\, f_\infty$ and  $\hat{{\bm z}}$ (see Definition \ref{def-1frf}),
(d) holds because $\bar{{\bm v}}^k\in \mathcal{K}^\theta$,
(e) follows from the fact that $\psi_{{\rm epi}\, f_\infty, \hat{{\bm z}}}$ is nondecreasing,
(f) follows from \eqref{thm-frf-2d-f-7} and \eqref{thm-frf-2d-f-8},
and (g) follows from \eqref{thm-frf-2d-f-3-1}.
The last display together with \eqref{thm-frf-2d-f-7} shows that there exists a constant $c_4>0$ such that
\begin{align}\label{thm-frf-2d-f-11}
\|{\bm w}^k-{\bm u}^k\|\leq  c_4\|{\bm w}^k-{\bm v}^k\| \quad \forall\, k.
\end{align}
Therefore, we see that \eqref{thm-frf-2d-f-1} holds.
\end{proof}

In view of Theorem \ref{thm-frf-2d-f}, Lemma \ref{lemma-frf} and Theorem \ref{thm-frf}, we obtain a $\mathds{1}$-FRF concerning $\mathcal{F}^e_{\rm bd}$ as  follows.
\begin{corollary}[$\mathds{1}$-FRF concerning $\mathcal{F}^e_{\rm bd}$]\label{coro-1frf-2d-f}
Let $\mathcal{F}^e_{\rm bd}={\rm epi}\, f^\pi \cap \{{\bm z}^*\}^\bot$ be the exposed face \eqref{face-2d}, where ${\bm z}^*=t^*(-\beta, r^*/t^*, 1)$ with
 $\beta\in {\rm bdry}({\rm dom}\, f^*)\cap {\rm dom}\, f^*$, $r^*>t^*f^*(\beta)$ and $t^*>0$. Let $\kappa_{{\bm z}^*,  t} \coloneqq  \max\{2, 2\gamma_{{\bm z}^*,t}^{-1}\}$ be defined in Theorem \ref{thm-frf-2d-f}, which is a nondecreasing function with respect to $t$ by the definition.
Then, the $\mathds{1}$-FRF $\psi_{{\rm epi}\, f^\pi,{\bm z}^*}: \R_+\times \R_+\rightarrow \R_+$ for ${\rm epi}\, f^\pi $ and ${\bm z}^*$ is given by
\begin{align*}
\psi_{{\rm epi}\, f^\pi,{\bm z}^*}(\epsilon,t)\coloneqq\max\{\epsilon,\epsilon/\|{\bm z}^*\|\}+\kappa_{{\bm z}^*,t}(\epsilon+\max\{\epsilon,\epsilon/\|{\bm z}^*\|\}).
\end{align*}
In particular, in view of Definition \ref{def-holder-1frf}, the above $\mathds{1}$-FRF is H\"{o}lderian of exponent $1$ (and hence, dominated by $1/2$).
\end{corollary}

\subsection{$\mathds{1}$-FRF concerning $\mathcal{F}^e_{\neq \theta}$}

In this section, we study the $\mathds{1}$-FRF concerning the face $\mathcal{F}^e_{\neq \theta}$.
Recall that we are under the blanket Assumption~\ref{assu-blanket}.
\begin{lemma}\label{lemma-k}
	Let $\mathcal{F}^e_{\neq \theta}={\rm epi}\, f^\pi \cap \{{\bm z}^*\}^\bot$ be a proper exposed  face defined in \eqref{prop-md-face-1_5}, where ${\bm z}^*=(-x^*, s^*, 0 )\in \mathcal{K}^{*\theta}$ with  $s^*> (f^*)_\infty(x^*)$ and $x^*\neq 0$.
	Let $\check{{\bm z}}\coloneqq (-x^*,0)$.
	Then the following statements hold.
	\begin{enumerate}[\rm (i)]
		\item $\check{{\bm z}}\in ({\rm epi}\, f_\infty)^*$.
		\item  Let $\mathcal{K}^\theta$ be given in \eqref{per-cone} and $\mathcal{S}_{\check{{\bm z}}}\coloneqq \{(x,0,r) : \langle(x,r), \check{{\bm z}}\rangle=0\}$. Then for any ${\bm v}=(v_x,0,v_r)\in \mathcal{K}^\theta$   we have that
		\begin{equation}\label{lemma-k-1}
			\begin{split}
				&{\rm d}({\bm v},\mathcal{F}^e_{\neq \theta})={\rm d}((v_x,v_r),{\rm epi}\, f_\infty \cap \{\check{{\bm z}}\}^\bot), \\
				&{\rm d}({\bm v},\mathcal{S}_{\check{{\bm z}}})={\rm d}((v_x,v_r),\{\check{{\bm z}}\}^\bot),\,\,\,\,\,\,\,\,\,\,\,
				{\rm d}({\bm v},\mathcal{K}^\theta)={\rm d}((v_x,v_r),{\rm epi}\, f_\infty )=0.
			\end{split}
		\end{equation}
	\end{enumerate}
\end{lemma}
\begin{proof}
	It follows from a similar argument to the proof of Lemma~\ref{lemma-k0}. We omit the details for brevity.
	\end{proof}

\begin{theorem}\label{thm-frf-f0}
Consider the face $\mathcal{F}^e_{\neq \theta}={\rm epi}\, f^\pi \cap \{{\bm z}^*\}^\bot$ given in \eqref{prop-md-face-1_5}, where ${\bm z}^*=(-x^*,s^*,0)\in\mathcal{K}^{*\theta}$ with $s^*>(f^*)_\infty(x^*)$ and $x^*\neq 0$.
 Let $\eta>0$ and $\gamma_{{\bm z}^*, \eta}$ be defined in \eqref{thm-frf-1} with $\mathfrak{g}=|\cdot|$ and $({\rm epi}\, f^\pi,\mathcal{F}^e_{\neq \theta},{\bm z}^*)$ in place of $(\mathcal{K},\mathcal{F},{\bm z})$. Then we have that $\gamma_{{\bm z}^*, \eta}\in(0,\infty]$ and
\[
{\rm d}({\bm q}, \mathcal{F}^e_{\neq\theta})\leq \kappa_{{\bm z}^*, \eta} {\rm d}({\bm q}, {\rm epi}\, f^\pi) \quad \forall\, {\bm q}\in\{{\bm z}^*\}^\bot \cap B(\eta),
\]
where $\kappa_{{\bm z}^*, \eta} \coloneqq \max\{2, 2\gamma_{{\bm z}^*,\eta}^{-1}\}$.
\end{theorem}
\begin{proof}
According to Lemma \ref{lemma-gamma}, taking any $\bar{{\bm v}}\in \mathcal{F}^e_{\neq \theta}$ and a sequence $\{{\bm v}^k\}\subset {\rm bdry}({\rm epi}\,f^\pi)\cap B(\eta)\setminus \mathcal{F}^e_{\neq \theta}$ such that
\begin{align}\label{thm-frf-f0-1}
\lim_{k\rightarrow \infty}{\bm v}^k=\lim_{k\rightarrow \infty}{\bm w}^k=\bar{{\bm v}} \quad \mbox{with ${\bm w}^k={\rm P}_{\{{\bm z}^*\}^\bot}{\bm v}^k$, ${\bm u}^k={\rm P}_{\mathcal{F}^e_{\neq \theta}}{\bm w}^k$ and ${\bm w}^k\neq {\bm u}^k$},
\end{align}
we need to show that
\begin{align}\label{thm-frf-f0-2}
\liminf_{k\rightarrow\infty} \frac{\|{\bm w}^k-{\bm v}^k\|}{\|{\bm w}^k-{\bm u}^k\|}>0.
\end{align}
By passing to subsequences if necessary, we consider the following two cases.
\begin{enumerate}[{\rm (a)}]
\item
 ${\bm v}^k \in \mathcal{K}^\theta$ for all $k$.
\item
 ${\bm v}^k \in \mathcal{K}^t$ for all $k$.
\end{enumerate}

{\bf Case (a):} ${\bm v}^k=(v^k_x, 0, v^k_r)\in \mathcal{K}^\theta$ with $v^k_r\geq f_\infty(v^k_x)$.
We compute
\begin{align}\label{thm-frf-f0-2-1}
\| {\bm w}^k-{\bm v}^k\|=\frac{|\langle {\bm z}^*,{\bm v}^k\rangle|}{\|{\bm z}^*\|}=\frac{| x^*v^k_x|}{\|{\bm z}^*\|}.
\end{align}

Recall from Lemma~\ref{lemma-k} that $\check{{\bm z}}= (-x^*,0)\in ({\rm epi}\, f_\infty)^*$.
  Applying Lemma \ref{lemma-1frf}, we see that ${\rm epi}\, f_\infty$ and $\check{\bm z}$ admit a $\mathds{1}$-FRF $\psi_{{\rm epi}\, f_\infty, \check{{\bm z}}}:\R_+\times \R_+\rightarrow \R_+$ given by
\begin{align}\label{thm-frf-f0-0}
\psi_{{\rm epi}\, f_\infty, \check{{\bm z}}}(\epsilon,t)=\tau \epsilon,
\end{align}
for some $\tau>0$.
Define $\mathcal{S}_{\check{{\bm z}}}\coloneqq \{(x,0,r) : \langle(x,r), \check{{\bm z}}\rangle=0\}$.  It then follows that for all $k$,
\begin{align}
\| {\bm w}^k-{\bm u}^k \|\overset{\rm (a)}{\leq} &\|{\bm v}^k-{\rm P}_{\mathcal{F}^e_{\neq \theta}} {\bm v}^k\|\overset{\rm (b)}{=}{\rm d}((v^k_x,v^k_r),{\rm epi}\, f_\infty \cap \{\check{{\bm z}}\}^\bot) \notag \\
\overset{\rm (c)}{\leq}& \psi_{{\rm epi}\, f_\infty, \check{{\bm z}}} ({\rm d}({\bm v}^k,\mathcal{K}^\theta)+ \|\check{\bm z}\|{\rm d}({\bm v}^k,\mathcal{S}_{\check{{\bm z}}}), \|{\bm v}^k\| )
\overset{\rm (d)}{=} \psi_{{\rm epi}\, f_\infty, \check{{\bm z}}} (\|\check{\bm z}\| {\rm d}({\bm v}^k,\mathcal{S}_{\check{{\bm z}}}), \|{\bm v}^k\| )\notag\\
\overset{\rm (e)}{=} &\tau \|\check{\bm z}\|  {\rm d}({\bm v}^k,\mathcal{S}_{\check{{\bm z}}})=  \tau  \|\check{\bm z}\| \frac{| v^k_x x^*|}{|x^*|}\overset{\rm (f)}{=}  c_1 \|{\bm w}^k-{\bm v}^k\|,    \label{thm-frf-f0-4}
\end{align}
where (a) follows from Lemma \ref{lemma-bound-wu},
(b) follows from the first equality in \eqref{lemma-k-1} (with ${\bm v}^k$ in place of  ${\bm v}$),
(c) follows from  \eqref{lemma-k-1} (with ${\bm v}^k$ in place of  ${\bm v}$) and the fact that $\psi_{{\rm epi}\, f_\infty, \check{{\bm z}}}$ is a $\mathds{1}$-FRF for ${\rm epi}\, f_\infty$ and  $\check{{\bm z}}$ (see Definition \ref{def-1frf}),
(d) holds because ${\bm v}^k\in\mathcal{K}^\theta$,
(e) follows from \eqref{thm-frf-f0-0},
and (f) follows from  \eqref{thm-frf-f0-2-1} upon letting $c_1\coloneqq \tau\|{\bm z}^*\|$.

{\bf Case (b):}
Next, we consider
${\bm v}^k=(v^k_x, v^k_t, v^k_r)\in \mathcal{K}^t$ with $v^k_r= v^k_t (f(v^k_x/v^k_t)+\varsigma^k)$ and $v^k_t>0$, where $\varsigma^k\geq 0$ if $v^k_x/v^k_t\in {\rm bdry}({\rm dom}\, f)\cap {\rm dom}\, f$  with  $v^k_t\varsigma^k\leq \eta$ since ${\bm v}^k\in B(\eta)$, and $\varsigma^k= 0$ otherwise.
We have
\begin{align}
&{\bm w}^k={\rm P}_{\{{\bm z}^*\}^\bot}{\bm v}^k={\bm v}^k-\frac{\langle {\bm z}^*, {\bm v}^k\rangle}{\|{\bm z}^*\|^2}{\bm z^*}=(v^k_x,v^k_t,v^k_r)-\frac{ -x^*v^k_x+v^k_t s^*}{\|{\bm z}^*\|^2}(-x^*,s^*,0).\label{thm-frf-f0-4-1}
\end{align}
Notice that $(f^*)_\infty$ is a positively homogeneous proper convex function; see \cite[Theorem 8.5]{Rockafellar-1970}. The fact that $0\neq x^*\in {\rm dom}\, (f^*)_\infty \subseteq \R$  implies that $x^*\in {\rm int}({\rm dom}\, (f^*)_\infty)$.
Then $(f^*)_\infty$ is continuous around $x^*$.
So,  there exists $\epsilon>0$ such that $x^*+\epsilon \xi\in {\rm int}({\rm dom}\, (f^*)_\infty)$ for any $\xi\in [-1,1]$ and
\begin{align}\label{thm-frf-f0-5}
0.5(s^*-(f^*)_\infty(x^*))\geq  |(f^*)_\infty(x^*)-(f^*)_\infty(x^*+\epsilon \xi)|.
\end{align}
If $v^k_x\neq 0$, we have
\begin{align}
\|{\bm w}^k-{\bm v}^k\|&=\frac{|\langle {\bm z}^*, {\bm v}^k\rangle|}{\|{\bm z}^*\|}=\frac{|v^k_t s^* -x^* v^k_x|}{\|{\bm z}^*\|}\overset{\rm (a)}{=} \frac{v^k_t (s^*-  x^*(v^k_x/v^k_t))}{\|{\bm z}^*\|}\notag \\
& \overset{\rm (b)}{=} \frac{v^k_t (s^*-  (x^*+\epsilon(\zeta^k/|\zeta^k|)) \zeta^k)+ \epsilon|v^k_x|}{\|{\bm z}^*\|}\notag  \\
& \overset{\rm (c)}{\geq} \frac{v^k_t (s^*-(f^*)_\infty(x^*)+ (f^*)_\infty(x^*)-(f^*)_\infty(x^*+\epsilon(\zeta^k/|\zeta^k|)) ) + \epsilon|v^k_x|}{\|{\bm z}^*\|}\notag  \\
& \overset{\rm (d)}{\geq} \frac{0.5(s^*-(f^*)_\infty(x^*))v^k_t + \epsilon|v^k_x|}{\|{\bm z}^*\|}\overset{\rm (e)}{\geq}  c_2(|v^k_x|+v^k_t), \label{thm-frf-f0-6}
\end{align}
where  (a) follows from  ${\bm z}^*\in({\rm epi}\,f^\pi)^*$ and ${\bm v}^k\in {\rm epi}\,f^\pi$ and hence $\langle {\bm z}^*,{\bm v}^k \rangle \geq 0$,
(b) follows upon letting $\zeta^k\coloneqq v^k_x/v^k_t$,
(c) follows from \eqref{infty-sigma} and the fact that $v^k_x/v^k_t\in {\rm dom}\, f$,
(d) follows  from \eqref{thm-frf-f0-5},
and (e) follows upon letting $c_2\coloneqq \min \{0.5(s^*-(f^*)_\infty(x^*))/\|{\bm z}^*\|,\epsilon/\|{\bm z}^*\|\}>0$.
On the other hand, if $v^k_x= 0$, we also have from \eqref{thm-frf-f0-4-1} that
\begin{align} \label{thm-frf-f0-7}
\|{\bm w}^k-{\bm v}^k\|=(v^k_t s^*)/ \|{\bm z}^*\|= (s^*/\|{\bm z}^*\|)(|v^k_x|+v^k_t)  \quad \forall\,k.
\end{align}
Consequently, we have for all $k$
\begin{align}
\| {\bm w}^k-{\bm u}^k \|\leq \|{\bm v}^k-{\rm P}_{\mathcal{F}^e_{\neq \theta}} {\bm v}^k\|
&\leq \|{\bm v}^k- (0,0, v^k_t(f(v^k_x/v^k_t)+\varsigma^k))\|  \notag\\
&\leq |v^k_x|+v^k_t\leq c_3^{-1} \|{\bm v}^k-{\bm w}^k\|,  \label{thm-frf-f0-8}
\end{align}
where the first inequality follows from Lemma \ref{lemma-bound-wu},
the second inequality follows from the fact that
$(0,0,v^k_t(f(v^k_x/v^k_t)+\varsigma^k))\in \mathcal{F}^e_{\neq \theta}$ (thanks to $\{(0,r): r\geq 0\}\subseteq {\rm epi}\, f_\infty$ and the nonnegativity of $f$ and $\varsigma^k$),
and the last inequality follows from \eqref{thm-frf-f0-6} and \eqref{thm-frf-f0-7} upon setting $c_3 \coloneqq \min\{s^*/\|{\bm z}^*\|,c_2\} > 0$.

 Therefore,  \eqref{thm-frf-f0-4} and \eqref{thm-frf-f0-8} show that \eqref{thm-frf-f0-2} holds for any sequence $\{({\bm v}^k,{\bm w}^k,{\bm u}^k)\}$ satisfying \eqref{thm-frf-f0-1}.
\end{proof}

From Theorem \ref{thm-frf-f0}, Lemma \ref{lemma-frf} and Theorem \ref{thm-frf}, we obtain a $\mathds{1}$-FRF concerning $\mathcal{F}^e_{\neq \theta}$ as follows.
\begin{corollary}[$\mathds{1}$-FRF  concerning $\mathcal{F}^e_{\neq \theta}$]\label{coro-1frf-f0}
Let $\mathcal{F}^e_{\neq \theta}={\rm epi}\, f^\pi \cap \{{\bm z}^*\}^\bot$ be the exposed face given in \eqref{prop-md-face-1_5}, where ${\bm z}^*=(-x^*,s^*,0)\in\mathcal{K}^{*\theta}$ with $s^*>(f^*)_\infty(x^*)$ and $x^*\neq 0$.
 Let $\kappa_{{\bm z}^*, t} \coloneqq \max\{2, 2\gamma_{{\bm z}^*,t}^{-1}\}$ be defined in Theorem \ref{thm-frf-f0}, which is a nondecreasing function with respect to $t$ by the definition.
Then, the $\mathds{1}$-FRF $\psi_{{\rm epi}\, f^\pi,{\bm z}^*}: \R_+\times \R_+\rightarrow \R_+$ for ${\rm epi}\, f^\pi$ and ${\bm z}^*$ is given by
\begin{align*}
\psi_{{\rm epi}\, f^\pi,{\bm z}^*}(\epsilon,t)\coloneqq \max\{\epsilon,\epsilon/\|{\bm z}^*\|\}+\kappa_{{\bm z}^*,t}(\epsilon+\max\{\epsilon,\epsilon/\|{\bm z}^*\|\}).
\end{align*}
In particular, in view of Definition \ref{def-holder-1frf}, the above $\mathds{1}$-FRF is H\"{o}lderian of exponent $1$ (and hence, dominated by $1/2$).
\end{corollary}

\subsection{Error bounds}

In this section, given the $\mathds{1}$-FRFs of several exposed faces in Corollaries \ref{coro-1frf-1d-f}, \ref{coro-1frf-2d-f} and  \ref{coro-1frf-f0}, we are ready to establish error bounds for \eqref{prob-fea}.
This requires the notions of \emph{partial-polyhedral Slater's} (PPS) condition, \emph{distance to the PPS condition} ${\rm d}_{\rm PPS}$,  and \emph{distance to polyhedrality}  $\ell_{\rm poly}$ which are all recalled in Definition \ref{def-mis}.

\begin{theorem}[Error bounds for \eqref{Feas} with $\mathcal{K}= {\rm epi}\, f^\pi$] \label{the-eb}
Consider problem \eqref{prob-fea}. Suppose that $(\mathcal{L}+{\bm a})\cap {\rm epi}\, f^\pi\neq \emptyset$.
Then ${\rm d}_{\rm PPS}({\rm epi}\, f^\pi,\mathcal{L}+{\bm a})\leq 1$, and whenever ${\rm d}_{\rm PPS}({\rm epi}\, f^\pi,\mathcal{L}+{\bm a})= 1$, there exist a chain of faces $\mathcal{F}\subsetneq{\rm epi}\, f^\pi$ with length 2 and a vector ${\bm z}^*$ such that
\begin{equation} \label{the-eb-1}
\begin{split}
&\mbox{$\mathcal{F}={\rm epi}\, f^\pi\cap \{{\bm z}^*\}^\bot$ with ${\bm z}^*\in ({\rm epi}\, f^\pi)^*\cap (\mathcal{L}+{\bm a})^\bot$},\\
&\mbox{$(\mathcal{L}+{\bm a})\cap \mathcal{F}=(\mathcal{L}+{\bm a})\cap {\rm epi}\, f^\pi$ and $\{\mathcal{F}, \mathcal{L}+{\bm a}\}$ satisfies the PPS condition}.
\end{split}
\end{equation}
Moreover, the following items hold.
\begin{enumerate}[\rm (i)]
\item If ${\rm d}_{\rm PPS}({\rm epi}\, f^\pi,\mathcal{L}+{\bm a})=0$, then $\{{\rm epi}\, f^\pi,\mathcal{L}+{\bm a}\}$ satisfies the Lipschitzian error bound.

\item  If ${\rm d}_{\rm PPS}({\rm epi}\, f^\pi,\mathcal{L}+{\bm a})=1$,
 we have the following possibilities.
    \begin{enumerate}[\rm (a)]
    \item If $\mathcal{F}=\mathcal{F}^e_{\beta}$ and the vector ${\bm z}^*$ given in \eqref{the-eb-1} satisfies
\begin{align*}
{\bm z}^*\in \{(-x^*, r^*, t^*) : x^*/t^*\in {\rm int}({\rm dom}\, f^*)\setminus f'(W), r^*= t^*f^*(x^*/t^*),t^*>0\},
\end{align*}
with $f'(W)\coloneqq \{ v : v= f'(w), w\in W\}$ and $W\subset  {\rm int}({\rm dom}\,f)$  given in Assumption \ref{assu-blanket},
then $\{{\rm epi}\, f^\pi, \mathcal{L}+{\bm a}\}$ satisfies a uniform H\"{o}lderian error bound with exponent $1/2$.
\item If $\mathcal{F}=\mathcal{F}^e_{\rm bd}$ and the vector ${\bm z}^*$ given in \eqref{the-eb-1} satisfies
\begin{align*}
{\bm z}^*\in \{(-x^*, r^*, t^*) : x^*/t^*\in {\rm bdry}({\rm dom}\, f^*)\cap {\rm dom}\, f^*, r^*> t^*f^*(x^*/t^*),t^*>0\},
\end{align*}
 then $\{{\rm epi}\, f^\pi, \mathcal{L}+{\bm a}\}$ satisfies the Lipschitzian error bound.
 \item If $\mathcal{F}=\mathcal{F}^e_{\neq \theta}$ and the vector ${\bm z}^*$ given in \eqref{the-eb-1} satisfies
\begin{align*}
{\bm z}^*\in \{(-x^*,s^*,0):  s^*>(f^*)_\infty(x^*), x^* \neq 0\},
\end{align*}
 then $\{{\rm epi}\, f^\pi, \mathcal{L}+{\bm a}\}$ satisfies the Lipschitzian error bound.
\item If $\mathcal{F}=\{\bm 0\}$, then $\{{\rm epi}\, f^\pi, \mathcal{L}+{\bm a}\}$ satisfies the Lipschitzian error bound.
\end{enumerate}
\end{enumerate}
\end{theorem}
\begin{proof}
First, we  show that distance to the PPS condition satisfies ${\rm d}_{\rm PPS}({\rm epi}\, f^\pi,\mathcal{L}+{\bm a})\leq 1$.
 Since ${\rm epi}\, f^\pi\subset \R^3$, all its proper faces are polyhedral, i.e., $\ell_{\rm poly}({\rm epi}\, f^\pi)=1$.
 Applying Proposition \ref{prop-facial-reduction} with ${\rm epi}\, f^\pi$ in place of $\mathcal{K}$, we deduce the existence of a chain of faces $\mathcal{F}\subsetneq{\rm epi}\, f^\pi$ of length $2$ and a vector ${\bm z}$ satisfying \eqref{the-eb-1}.
 Then, from the definition of ${\rm d}_{\rm PPS}$, we see that ${\rm d}_{\rm PPS}({\rm epi}\, f^\pi,\mathcal{L}+{\bm a})\leq 1$.

{\bf (i)} ${\rm d}_{\rm PPS}({\rm epi}\, f^\pi,\mathcal{L}+{\bm a})= 0$.
In this case,  $\{{\rm epi}\, f^\pi, \mathcal{L}+{\bm a}\}$ satisfies the PPS condition.
 Then, in view of  \cite[Corollary 3]{Bauschke-Borwein-Li-1999},
    we  deduce that $\{{\rm epi}\,f^\pi, \mathcal{L}+{\bm a}\}$ satisfies a Lipschitzian error bound.

{\bf (ii)} ${\rm d}_{\rm PPS}({\rm epi}\, f^\pi,\mathcal{L}+{\bm a})= 1$.

In case (a),
from Corollary \ref{coro-1frf-1d-f},  the $\mathds{1}$-FRF is H\"{o}lderian of exponent $1/2$ (see Definition \ref{def-holder-1frf}).
Applying \cite[Lemma 2.1]{Lindstrom-Lourenco-Pong-2025}, we deduce that
$\{{\rm epi}\, f^\pi, \mathcal{L}+{\bm a}\}$ satisfies a uniform H\"{o}lderian error bound with exponent $1/2$.

Similarly, case (b) follows from \cite[Lemma 2.1]{Lindstrom-Lourenco-Pong-2025} and Corollary  \ref{coro-1frf-2d-f},  while case (c) can be deduced from \cite[Lemma 2.1]{Lindstrom-Lourenco-Pong-2025} and Corollary  \ref{coro-1frf-f0}.
Finally, case (d)  follows from \cite[Proposition 27]{Lourencco-2021}.
\end{proof}

\section{Hausdorff measure zero for the ``missing" exposed faces} \label{section-measure}
Every ${\bm z}^* \in ({\rm epi}\, f^\pi)^*$ exposes some
face of ${\rm epi}\, f^\pi$. However, the cases where
${\bm z}^* = {\bm 0}$ or ${\bm z}^* \in {\rm{int}} (({\rm epi}\, f^\pi)^*)$ are of little interest, since the exposed faces are $\{{\bm 0}\}$ and the whole cone ${\rm epi}\, f^\pi$, respectively.
For the nontrivial situation where ${\bm z}^* \in {\rm bdry}(({\rm epi}\, f^\pi)^*) \setminus \{\bm 0\}$, in the previous section we identified several cases where the $\mathds{1}$-FRFs for ${\bm z}^*$ and ${\rm epi}\, f^\pi$ can be taken to be H\"olderian of exponent (at least) $1/2$.
Although we did not compute $\mathds{1}$-FRFs for all cases, we will show in this section that the part of the boundary of ${\rm bdry}(({\rm epi}\, f^\pi)^*)$ omitted in our analysis has Hausdorff dimension strictly smaller than the part we actually analyzed.
Then, we will discuss several implications of this result.

First, we recall the definitions of Hausdorff measure and dimension in Euclidean space;  see  \cite{Evans-et-2015, Mattila-1995}.
Letting $A\subseteq \R^d$, $s\in[0,\infty)$ and  $\delta\in(0,\infty]$, the $s$-dimensional Hausdorff measure is
\[
\mathcal{H}^s(A) \coloneqq\lim_{\delta\rightarrow 0}\mathcal{H}^s_\delta(A)=\sup_{\delta>0}\mathcal{H}^s_\delta(A),
\]
where
\[
\mathcal{H}^s_\delta(A)\coloneqq\inf\left\{ \sum_{j=1}^\infty {\rm diam}(C_j)^s : A\subseteq \bigcup_{j=1}^\infty C_j, {\rm diam}(C_j)\leq \delta \right\}.
\]
The Hausdorff dimension of a set $A\subseteq \R^d$ is
\begin{align}\label{Haus-dim}
H_{\rm dim}(A)\coloneqq \inf\{ 0\leq s<\infty :  \mathcal{H}^s(A)=0\}.
\end{align}
\begin{fact}\label{fact-hausdorff}
Some facts about  Hausdorff measure and dimension.
\begin{enumerate}[\rm (i)]
\item $\mathcal{H}^s(\mathcal{T}A)=\mathcal{H}^s(A)$ for each affine isometry $\mathcal{T}:\R^d\rightarrow \R^d$ and $A\subseteq \R^d$; see \cite[Theorem 2.2(v)]{Evans-et-2015}.
\item  $H_{\rm dim}(A)\leq H_{\rm dim}(B)$ for $A\subseteq B$; see \cite[page 59]{Mattila-1995}.
\item Let $\{A_i\}$ be a sequence of sets. Then $H_{\rm dim}(\cup_{i=1}^\infty A_i)= \sup_i H_{\rm dim}(A_i)$;  see \cite[page 59]{Mattila-1995}.
\item Let $A\subseteq \R^{d}$ and $B\subseteq\R^{s}$ with $H_{\rm dim}(B)=s$. Then $H_{\rm dim}(A\times B)$= $H_{\rm dim}(A)+s$.
\end{enumerate}
\end{fact}
\begin{proof}
For item (iv), from the fact that $H_{\rm dim}(B)=s$ and \cite[Examples 3]{Claude-1982},  the set $B$ is \emph{$R$-regular}  in the sense of \cite[Definition 6]{Claude-1982}.
Then, item (iv) follows from \cite[Theorem 3]{Claude-1982}.
\end{proof}

Recall that we are under the blanket Assumption~\ref{assu-blanket}.
We see from  \eqref{dual-cone} that  the boundary of the dual cone of ${\rm epi}\, f^\pi$ is
\begin{align}
&{\rm bdry}(({\rm epi}\, f^\pi)^*)\notag\\
=&\{t^*(-\beta,r^*/t^*,1):\,    r^*=t^*f^*(\beta),\beta\in{\rm int}({\rm dom}\, f^*), t^*>0\}\notag\\
&\cup \{t^*(-\beta, r^*/t^*, 1) : \beta\in {\rm bdry}({\rm dom}\, f^*)\cap {\rm dom}\, f^*, r^*\geq t^*f^*(\beta),t^*>0\} \notag\\
&\cup \{(-x^*,s^*,0):  s^*>(f^*)_\infty(x^*), x^*\neq 0\} \cup\{(-x^*,s^*,0):  s^*=(f^*)_\infty(x^*), x^*\neq 0\}\notag\\
&\cup \{(-x^*,s^*,0):  s^*\geq (f^*)_\infty(x^*), x^*= 0\}. \label{boundary}
\end{align}

Consider the following partition
\begin{align}\label{Q12}
{\rm bdry}(({\rm epi}\, f^\pi)^*)=Q_1\cup Q_2
\end{align}
where
\begin{align}
Q_1=&Q_{11}\cup Q_{12}\cup Q_{13} \notag \\
\coloneqq &\{t^*(-\beta,r^*/t^*,1):\,    r^*=t^*f^*(\beta),\beta\in{\rm int}({\rm dom}\, f^*)\setminus f'(W), t^*>0\}\notag\\
& \cup\{t^*(-\beta, r^*/t^*, 1) : \beta\in {\rm bdry}({\rm dom}\, f^*)\cap {\rm dom}\, f^*, r^*> t^*f^*(\beta),t^*>0\} \notag\\
&\cup \{(-x^*,s^*,0):  s^*>(f^*)_\infty(x^*), x^* \neq 0\}, \label{Q1}\\
Q_2=&Q_{21}\cup Q_{22}\cup Q_{23}\cup Q_{24}\notag \\
\coloneqq &\{t^*(-\beta,r^*/t^*,1):\,    r^*=t^*f^*(\beta),\beta\in f'(W), t^*>0\}\notag\\
&\cup\{t^*(-\beta, r^*/t^*, 1) : \beta\in {\rm bdry}({\rm dom}\, f^*)\cap {\rm dom}\, f^*, r^*= t^*f^*(\beta),t^*>0\} \notag\\
&\cup\{(-x^*, s^*, 0) : s^*=(f^*)_\infty(x^*), x^* \neq 0\} \notag\\
&\cup \{(-x^*, s^*, 0) : s^*\geq (f^*)_\infty(x^*), x^* = 0\},\label{Q2}
\end{align}
where $f'(W)\coloneqq \{ v : v= f'(w), w\in W\}$ with $W\subset \R$  being the finite set given in Assumption \ref{assu-blanket}.

Recall from Theorem \ref{thm-frf-1d-f} that ${\rm epi}\, f^\pi$ and ${\bm z}^*\in Q_{11}$ admit a H\"{o}lderian $\mathds{1}$-FRF with exponent $1/2$,
while from Theorems \ref{thm-frf-2d-f} and \ref{thm-frf-f0} we know that ${\rm epi}\, f^\pi$ and ${\bm z}^*\in Q_{12}$ or ${\bm z}^*\in Q_{13}$ admit a H\"{o}lderian $\mathds{1}$-FRF with exponent $1$ (and hence with exponent $1/2$ as well).

We aim to show that the Hausdorff dimension of $Q_1$ is strictly larger than that of $Q_2$.
We first recall the following useful theorem.
Here, for a function $h:\R^n\rightarrow \R^m$, its graph restricted to $A\subseteq \R^n$ is denoted by
\[
{\rm gph}(h|_A) = \{ ({\bm x}, h({\bm x})): {\bm x}\in A\}.
\]
\begin{theorem}{\rm \cite[Theorem 2.9]{Evans-et-2015}}\label{the-dimension}
Let $h:\R^n\rightarrow \R^m$ be Lipschitz continuous and $A\subseteq \R^n$. If $\mathcal{H}^n(A)>0$, then $H_{\rm dim}({\rm gph}(h|_A))=n$.
\end{theorem}

\begin{proposition}\label{prop-dim}
Let $Q_1$ and $Q_2$ be given in  \eqref{Q1} and \eqref{Q2}, respectively. Then, it holds that
\[
H_{\rm dim}(Q_1)= 2,\quad  H_{\rm dim}(Q_2)=  1.
\]
\end{proposition}
\begin{proof}
First, we show that $H_{\rm dim}(Q_1)\geq  2$.
Define the  function $F:\R^2\rightarrow \R\cup \{\infty\}$ as
\[
F(x,t)=
\begin{cases}
t f^*(x/t) & \mbox{if $t>0$}, \\
\infty  &  \mbox{otherwise}.
\end{cases}
\]
Then ${\rm dom}\, F=\cup_{t>0} t\cdot ({\rm dom}\, f^*\times \{1\})$, and $F$ is convex according to \cite[Proposition 2.5.7]{Auslender-Alfred-Teboulle-2006}.
Note that  ${\rm int}({\rm dom}\, f^*)\neq \emptyset$ thanks to  the fact that $f^*$ is a Legendre function.
Let
\[
E\coloneqq \bigcup_{t>0} t\cdot ({\rm int}({\rm dom}\, f^*)\setminus f'(W)\times \{1\})\neq \emptyset,
\]
where $f'(W)\coloneqq \{ v : v= f'(w), w\in W\}$ with $W\subset \R$  given in Assumption \ref{assu-blanket}.
As $E$ is an open set contained in ${\rm dom}\, F$, the function $F$ is Lipschitz  continuous over any compact subset of $E$ by \cite[Theorem~10.4]{Rockafellar-1970}.
Notice that
\begin{align}\label{prop-dim-1}
{\rm gph}(F|_{E})=\{(x^*,t^*,r^*) : r^*=t^*f^*(x^*/t^*), x^*/t^*\in{\rm int}({\rm dom}\, f^*)\setminus f'(W),  t^*>0 \}.
\end{align}
Define $\mathcal{T}:\R^3\rightarrow\R^3$ as
$\mathcal{T}(x,t,r) = (-x,r,t)$.
Let $D\varsubsetneq E$  be compact with nonempty interior.
 It follows that
\begin{align}\label{Hdimeq1}
H_{\rm dim}(Q_{11})&=H_{\rm dim}(\{t^*(-\beta,r^*/t^*,1):\,    r^*=t^*f^*(\beta),\beta\in{\rm int}({\rm dom}\, f^*)\setminus f'(W), t^*>0\}) \notag\\
&\overset{\rm (a)}{=}H_{\rm dim}(\mathcal{T}{\rm gph}(F|_{E}))\overset{\rm (b)}{=}   H_{\rm dim}({\rm gph}(F|_{E})) \overset{\rm (c)}{\geq}  H_{\rm dim}({\rm gph}(F|_{D}))\overset{\rm (d)}{=}2,
\end{align}
where  (a) follows from  \eqref{prop-dim-1} and the definition of $\mathcal{T}$,
(b) follows from Fact \ref{fact-hausdorff}(i),
(c) follows from Fact \ref{fact-hausdorff}(ii),
and (d) follows from Theorem \ref{the-dimension} thanks to the facts that $F$ is Lipschitz continuous on $D$ and $\mathcal{H}^2(D)>0$. Therefore, we have
\begin{equation}\label{Hdimeq2}
H_{\rm dim}(Q_1)= H_{\rm dim}(Q_{11}\cup Q_{12}\cup Q_{13})= \max \{H_{\rm dim}(Q_{11}),H_{\rm dim}(Q_{12}),H_{\rm dim}(Q_{13})\}  \geq 2,
\end{equation}
where the second equality follows from Fact \ref{fact-hausdorff}(iii), and the last inequality follows from \eqref{Hdimeq1}.

Next, in view of \cite[Corollary 15]{Silva-Tuncel-2019} and the fact that $H_{\rm dim}(({\rm epi}\, f^\pi)^*)=3$ (since $({\rm epi}\, f^\pi)^*$ is {pointed and} solid thanks to Proposition~\ref{prop-pointed}), we have that the Hausdorff dimension of ${\rm bdry}(({\rm epi}\, f^\pi)^*)=Q_1\cup Q_2$ is $2$, i.e. $H_{\rm dim}(Q_1\cup Q_2)=2$.
This together with \eqref{Hdimeq2} and Fact \ref{fact-hausdorff}(ii) shows that $H_{\rm dim}(Q_1)=2$.

Now, we will show that $H_{\rm dim}(Q_2)= 1$.  Note from \eqref{Q2} that $Q_2=Q_{21}\cup Q_{22}\cup Q_{23}\cup Q_{24}$. We will look at $Q_{21}$, $Q_{22}$, $Q_{23}$ and $Q_{24}$ one by one.

$\underline{Q_{21}=\{t^*(-\beta,r^*/t^*,1):\,    r^*=t^*f^*(\beta),\beta\in f'(W), t^*>0\}}$. Since $f'(W)$ is a finite set (as $W$ is a finite set by Assumption \ref{assu-blanket}), we see that $Q_{21}$ is a finite union of half lines in $\R^3$.

$\underline{Q_{22}=\{t^*(-\beta, r^*/t^*, 1) : \beta\in {\rm bdry}({\rm dom}\, f^*)\cap {\rm dom}\, f^*, r^*= t^*f^*(\beta),t^*>0\}}$.
As ${\rm bdry}({\rm dom}\, f^*)\cap {\rm dom}\, f^*$ contains at most two points, the set $Q_{22}$ is the union of at most two half lines in $\R^3$.

$\underline{Q_{23}=\{(-x^*, s^*, 0) : s^*=(f^*)_\infty(x^*), x^* \neq 0\}}$.
First, if ${\rm int}({\rm dom}\, (f^*)_\infty)= \emptyset$,  then we deduce that ${\rm dom}\, (f^*)_\infty=\{0\}$ and hence $Q_{23}=\emptyset$. Next, we consider the case that ${\rm int}({\rm dom}\, (f^*)_\infty)\neq \emptyset$.
Note that $(f^*)_\infty$ is positively homogeneous, proper and convex thanks to \cite[Theorem 8.5]{Rockafellar-1970}.
 From the positive homogeneity of $(f^*)_\infty$  and the fact that $0\in {\rm dom}\,(f^*)_\infty$ (see \eqref{def-f-infty}), it holds that
 \begin{align*}
 {\rm dom}\,(f^*)_\infty=
 	\begin{cases}
 \R  &	\mbox{if $0 \in {\rm int}({\rm dom}\, (f^*)_\infty)$}, \\
 \mbox{$[0,\infty)$ or  $(-\infty,0]$}  &  \mbox{if $0 \notin {\rm int}({\rm dom}\, (f^*)_\infty)$.}
 \end{cases}
 	\end{align*}
Next, we argue that $(f^*)_\infty$ is piecewise linear.
Indeed, if $0\in{\rm int}({\rm dom}\, (f^*)_\infty)$, the positive homogeneity of $(f^*)_\infty$ implies that
 \begin{equation*}
 (f^*)_\infty(x)=
 \begin{cases}
 	x (f^*)_\infty(1) &\mbox{if $x\geq 0$},\\
 	|x| (f^*)_\infty(-1) & \mbox{if $x< 0$}.\\
 	\end{cases}
 \end{equation*}
 If $0\notin{\rm int}({\rm dom}\, (f^*)_\infty)$, then it  holds either  that ${\rm dom}\, (f^*)_\infty=[0,\infty)$ with $(f^*)_\infty(x)=x(f^*)_\infty(1)$ for $x\geq 0$ or  ${\rm dom}\, (f^*)_\infty =(-\infty,0]$ with $(f^*)_\infty(x)=|x|(f^*)_\infty(-1)$ for $x\leq 0$.
  Therefore, we see that $Q_{23}$ is actually the union of at most two half lines in $\R^3$.

$\underline{Q_{24}=\{(0, s^*, 0) : s^*\geq 0\}}$.  It is clear that $Q_{24}$ is a half line in $\R^3$.

Then the set $Q_2$ is the union of finitely many half lines in $\R^3$.
Moreover, any half line in $\R^3$ is of Hausdorff dimension one.\footnote{Indeed, let $L\coloneqq \{t{\bm b}: t\geq 0\}\subset \R^3$ be a half real line for some ${\bm b}=(b_1,b_2,b_3)\neq 0$.
Without loss of generality, we assume that $b_1=1$ (otherwise we permute the coordinates and reparameterize the $t$ by a positive scaling, and this does not change the Hausdorff dimension). It is clear that
$L$ is the graph of the Lipschitz mapping $t\mapsto t(b_2,b_3)$ restricted to $t\geq 0$. Using Theorem \ref{the-dimension}, we see that $H_{\rm dim}(L)=1$.}
These two facts together with Fact \ref{fact-hausdorff}(iii) imply that $H_{\rm dim}(Q_2)=1$.
\end{proof}

\subsection{Almost everywhere H\"olderian $\mathds{1}$-FRF of exponent $1/2$ and error bounds}

We use the notation  ``$\mathcal{H}^s$-a.e." to indicate almost everywhere with respect to the $s$-dimensional Hausdorff measure $\mathcal{H}^s$ except on a set $A$ with $\mathcal{H}^s(A)=0$.

\begin{theorem}[almost everywhere H\"{o}lderian $\mathds{1}$-FRF of exponent $1/2$]\label{the-almost-1frf}
 Let ${\rm epi}\,f^\pi$ and $({\rm epi}\, f^\pi)^*$ be given   in \eqref{per-cone} and \eqref{dual-cone}, respectively.
 Then, for any $s > 1$, for ${\rm epi}\, f^\pi$ and $\mathcal{H}^s$-a.e. ${\bm z}^*\in{\rm bdry}(({\rm epi}\, f^\pi)^*)$,  the corresponding $\mathds{1}$-FRF can be taken to be H\"{o}lderian of exponent $1/2$.
\end{theorem}
\begin{proof}
The assertion follows from  Corollaries \ref{coro-1frf-1d-f}, \ref{coro-1frf-2d-f}, \ref{coro-1frf-f0}, Proposition \ref{prop-dim}, and the definition of the Hausdorff dimension in \eqref{Haus-dim}.
\end{proof}

\begin{theorem}[almost everywhere H\"{o}lderian error bound  with exponent $1/2$]\label{the-almost-eb}
\begin{enumerate}[\rm (i)]
\item Let ${\rm epi}\,f^\pi$ and $({\rm epi}\, f^\pi)^*$ be given   in \eqref{per-cone} and \eqref{dual-cone}, respectively.
 Then, for any $s > 1$,    $\{{\rm epi}\, f^\pi, \{{\bm z}\}^\bot\}$   satisfies a uniform H\"{o}lderian error bound  with exponent $1/2$ for $\mathcal{H}^s$-a.e. ${\bm z}\in{\rm bdry}(({\rm epi}\, f^\pi)^*)$.

 \item Let $m\ge 2$ be an integer and let $g^{(i)} : \R \to \R_+ \cup\{\infty\}$ satisfy Assumption \ref{assu-blanket} for each $i\in [m] := \{1,\ldots,m\}$.
 Let $\mathcal{K}\coloneqq \mathcal{K}^{(1)}\times\cdots \times \mathcal{K}^{(m)}$, where each $\mathcal{K}^{(i)}$ is a three-dimensional perspective cone generated by  $g^{(i)}$.
  Then    $\{\mathcal{K}, \{{\bm z}\}^\bot\}$   satisfies a uniform H\"{o}lderian error bound  with exponent $1/2$ for  $\mathcal{H}^{3m-1}$-a.e. ${\bm z}\in{\rm bdry}(\mathcal{K}^*)$.
\end{enumerate}
\end{theorem}
\begin{proof}
(i) Recall from \eqref{Q12} that ${\rm bdry}(({\rm epi}\, f^\pi)^*)=Q_1\cup Q_2$.
In view of Corollaries \ref{coro-1frf-1d-f}, \ref{coro-1frf-2d-f}, \ref{coro-1frf-f0}
and the definition of $\mathds{1}$-FRF, we  deduce that $\{{\rm epi}\,f^\pi, \{\bm z\}^\bot\}$ satisfies a uniform H\"{o}lderian error bound  with exponent $1/2$ whenever ${\bm z}\in Q_1$. This together with the definition of Hausdorff dimension (see \eqref{Haus-dim}) and Proposition \ref{prop-dim} implies the almost everywhere uniform H\"{o}lderian error bound  with exponent $1/2$.

(ii) Note that  $\mathcal{K}^*=(\mathcal{K}^{(1)})^*\times\cdots\times(\mathcal{K}^{(m)})^*$.
For each $i\in[m]$, we can write
\begin{align*}
{\rm bdry}((\mathcal{K}^{(i)})^*)=Q^{(i)}_1\cup Q^{(i)}_2,
\end{align*}
where $Q^{(i)}_1$ and  $Q^{(i)}_2$ are respectively defined according to  \eqref{Q1} and \eqref{Q2} with  $g^{(i)}$ in place of $f$. Then, according to Proposition \ref{prop-dim}, we have
 \begin{align}\label{almost-eb-0}
H_{\rm dim}(Q^{(i)}_1)= 2, \quad H_{\rm dim}(Q^{(i)}_2)= 1 \quad \forall\, i\in [m].
\end{align}

 Define the following sets for each $i\in [m]$:
 \begin{align}
 A^{(i)}&=\bigcup_{J\subseteq [m]\atop |J|=i}\underbrace{\{Z_1\times\cdots\times Z_m: \mbox{$Z_k=Q^{(k)}_1$ for $k\in J$ and $Z_l ={\rm int}((\mathcal{K}^{(l)})^*)$ for $l\in [m]\setminus J$}\}}_{\eqqcolon A(i,J)}, \label{almost-eb-2}\\
B^{(i)}&=\bigcup_{J\subseteq [m]\atop |J|=i} \underbrace{\{Z_1\times\cdots\times Z_m: \mbox{$Z_k=Q^{(k)}_2$ for $k\in J$ and $Z_l =(\mathcal{K}^{(l)})^*\setminus Q_2^{(l)}$ for $l\in [m]\setminus J$}\}}_{\eqqcolon B(i,J)}.\label{almost-eb-3}
\end{align}
Then we observe that
\begin{equation}\label{bdryKstar}
 {\rm bdry}(\mathcal{K}^*)=\left(\bigcup_{j=1}^m A^{(j)} \right)\bigcup \left(\bigcup_{j=1}^m B^{(j)}\right).
 \end{equation}
Moreover, in view of item (i) and \cite[Proposition 3.13]{Lindstrom-Lourenco-Pong-2023}, if
\[
\mbox{ ${\bm z}=({\bm z}^{(1)},\ldots, {\bm z}^{(m)})\in {\rm bdry}(\mathcal{K}^*)$
with each ${\bm z}^{(i)}\in Q^{(i)}_1$ or ${\bm z}^{(i)}\in {\rm int}((\mathcal{K}^{(i)})^*)$},
\]
then $\{\mathcal{K},\{{\bm z}\}^\bot\}$ satisfies a uniform H\"{o}lderian error bound  with exponent $1/2$, see this footnote.\footnote{From \cite[Proposition 3.13]{Lindstrom-Lourenco-Pong-2023}, $\mathds{1}$-FRFs for a direct product can be taken to be sums of $\mathds{1}$-FRFs for the individual blocks, up to a positive rescaling. For the blocks such that ${\bm z}^{(i)}\in Q^{(i)}_1$, the $\mathds{1}$-FRFs can be taken to be H\"olderian of exponent $1/2$ by the results in Section~\ref{section-frf} and, for the remaining blocks, the $\mathds{1}$-FRFs can be taken to be H\"olderian of exponent $1$ by \cite[Lemma~2.2]{Lindstrom-Lourenco-Pong-2025}, and, hence, dominated by $1/2$.
Summing those $\mathds{1}$-FRFs and positively rescaling them, we obtain a H\"olderian $\mathds{1}$-FRFs of exponent $1/2$ for $\mathcal{K}$ and ${\bm z}$, from which we can see from the definition of $\mathds{1}$-FRF that this  leads to a uniform H\"olderian error bound with exponent $1/2$ for $\{\mathcal{K},\{{\bm z}\}^\bot\}$, e.g., see the footnote in Definition~\ref{def-1frf}.}
Thus, we know that for any ${\bm z}\in A^{(i)}$, $\{\mathcal{K},\{\bm z\}^\bot\}$ satisfies a uniform H\"{o}lderian error bound  with exponent $1/2$, whereas such error bound may fail if ${\bm z}\in B^{(i)}$. Hence, the proof will be complete if we show that $H_{\rm dim}(\cup_{j=1}^m A^{(j)})\ge 3m-1 >
  H_{\rm dim}(\cup_{j=1}^m B^{(j)})$.

We now compute the Hausdorff dimension of $\cup_{j=1}^m A^{(j)}$ and $\cup_{j=1}^m B^{(j)}$.
To this end, we consider the following sets
  \begin{align*}
 \hat{A}^{(1)}&=\bigcup_{k\in[m]}\underbrace{\{Z_1\times\cdots\times Z_m: \mbox{$Z_k=Q^{(k)}_1$  and $Z_l =(\mathcal{K}^{(l)})^*$ for $l\neq k$}\}}_{\eqqcolon \hat{A}(1,\{k\})}, \\
 \hat{B}^{(1)}&=\bigcup_{k\in[m]}\underbrace{\{Z_1\times\cdots\times Z_m: \mbox{$Z_k=Q^{(k)}_2$  and $Z_l =(\mathcal{K}^{(l)})^*$ for $l\neq k$}\}}_{\eqqcolon \hat{B}(1,\{k\})}.
 \end{align*}
We have $A^{(j)}\subsetneq \hat{A}^{(1)}$ and $B^{(j)}\subsetneq \hat{B}^{(1)}$ for any $j\in [m]$. Using this fact together with Fact \ref{fact-hausdorff}(ii), we have that for any $j\in [m]$,
 \begin{align} \label{almost-eb-4}
 H_{\rm dim}(A^{(j)})\leq H_{\rm dim}(\hat{A}^{(1)}), \quad
  H_{\rm dim}(B^{(j)})\leq H_{\rm dim}(\hat{B}^{(1)}).
 \end{align}
 Since $\mathcal{K}^{(i)}$ is solid and pointed thanks to Proposition \ref{prop-pointed}, we know from \cite[Proposition 2.4.3]{Facchinei-Pang-2003} that $(\mathcal{K}^{(i)})^*$ is also solid and pointed.
From this, we have that $H_{\rm dim}((\mathcal{K}^{(i)})^*)=H_{\rm dim}({\rm int}((\mathcal{K}^{(i)})^*))=H_{\rm dim}((\mathcal{K}^{(i)})^*\setminus Q_2^{(i)})=3$ for each $i\in[m]$, which together with Fact \ref{fact-hausdorff}(iv) implies that
 \begin{align} \label{almost-eb-5}
 H_{\rm dim}(A(1,\{k\}))=H_{\rm dim}(\hat{A}(1,\{k\})), \quad H_{\rm dim}(B(1,\{k\}))=H_{\rm dim}(\hat{B}(1,\{k\})) \quad \forall\, k\in [m].
 \end{align}
Hence, we obtain that
 \begin{align}
H_{\rm dim}(A^{(1)})=&H_{\rm dim}\left(\bigcup_{k\in [m]}A(1,\{k\})\right)\overset{\rm (a)}{=}\max_{k\in [m]}H_{\rm dim}({A}(1,\{k\}) \notag\\
\overset{\rm (b)}{=}&\max_{k\in [m]}H_{\rm dim}(\hat{A}(1,\{k\})\overset{\rm (c)}{=}\max_{k\in [m]}H_{\rm dim}(Q^{(k)}_1)+3(m-1) \overset{\rm (d)}{=} 3m-1,\label{almost-eb-6} \\
H_{\rm dim}(B^{(1)})=&H_{\rm dim}\left(\bigcup_{k\in [m]}B(1,\{k\})\right)\overset{\rm (e)}{=}\max_{k\in [m]}H_{\rm dim}({B}(1,\{k\})\notag\\
\overset{\rm (f)}{=}&\max_{k\in [m]}H_{\rm dim}(\hat{B}(1,\{k\})\overset{\rm (g)}{=}\max_{k\in [m]}H_{\rm dim}(Q^{(k)}_2)+3(m-1) \overset{\rm (h)}{=} 3m-2, \label{almost-eb-7}
 \end{align}
 where (a) and (e) follow from  Fact \ref{fact-hausdorff}(iii),
 (b) and (f) follow from \eqref{almost-eb-5},
 (c) and (g) follow from Fact \ref{fact-hausdorff}(iv) and that $H_{\rm dim}((\mathcal{K}^{(i)})^*)=3$ for each $i$,
and (d) and (h) follow from \eqref{almost-eb-0}.
In view of \eqref{almost-eb-4}, \eqref{almost-eb-6} and \eqref{almost-eb-7}, it holds that
 \begin{align}\label{almost-eb-8}
\max_{j\in [m]} H_{\rm dim}(A^{(j)})=  H_{\rm dim}(A^{(1)})= 3m-1, \quad \max_{j\in [m]} H_{\rm dim}(B^{(j)})=  H_{\rm dim}(B^{(1)})= 3m-2.
\end{align}
 Then, we  deduce that
 \begin{align*}
 H_{\rm dim}(\cup_{j=1}^m A^{(j)})\overset{\rm (a)}{=}\max_{j\in [m]} H_{\rm dim}(A^{(j)})
 \overset{\rm (b)}{=} 3m-1, \\
  H_{\rm dim}(\cup_{j=1}^m B^{(j)})\overset{\rm (c)}{=}\max_{j\in [m]} H_{\rm dim}(B^{(j)})\overset{\rm (d)}{=} 3m-2,
 \end{align*}
 where (a) and (c) follow from Fact \ref{fact-hausdorff}(iii),
 and (b) and (d) follows  from \eqref{almost-eb-8}.
 \end{proof}

The feasibility problem \eqref{prob-fea} is said to be \emph{weakly feasible} if it is feasible and  $(\mathcal{L}+{\bm a}) \cap {\rm int}({\rm epi}\, f^\pi) = \emptyset$, see
\cite[Section~2]{LSZ97}. We say that it is \emph{nontrivially weakly feasible} if it is feasible and,
\begin{align}\label{nontrivial}
(\mathcal{L}+{\bm a}) \cap {\rm int}({\rm epi}\, f^\pi) = \emptyset \ \ {\rm and}\ \
 (\mathcal{L}+{\bm a})\cap {\rm epi}\, f^\pi \neq \{{\bm 0}\}.
\end{align}
Notice that if $(\mathcal{L}+{\bm a}) \cap {\rm int}({\rm epi}\, f^\pi) \neq \emptyset$, then $\{{\rm epi}\, f^\pi, \mathcal{L}+{\bm a}\}$ satisfies a Lipschitzian error bound by \cite[Proposition 2.3]{Lindstrom-Lourenco-Pong-2023} or \cite[Corollary 3]{Bauschke-Borwein-Li-1999}.
On the other hand, if $(\mathcal{L}+{\bm a})\cap {\rm epi}\, f^\pi= \{{\bm 0}\}$,  then $\{{\rm epi}\, f^\pi, \mathcal{L}+{\bm a}\}$ satisfies a Lipschitzian error bound by \cite[Proposition 27]{Lourencco-2021} (see also Theorem \ref{the-eb}).
In summary, for a feasible \eqref{prob-fea} that fails \eqref{nontrivial},  $\{{\rm epi}\, f^\pi, \mathcal{L}+{\bm a}\}$ satisfies a Lipschitzian error bound.

We next discuss error bounds for nontrivially weakly feasible
\eqref{prob-fea}. 
We start with the following lemma.

\begin{lemma}\label{lemma-trivial}
	Consider the feasibility problem \eqref{prob-fea} with $\mathcal{L}\coloneqq \{\bm z\}^\bot$ for some ${\bm z}$.
	Suppose that $(\mathcal{L}+{\bm a})\cap {\rm epi}\, f^\pi\neq \emptyset$.
Then, we have the following equivalence
\begin{equation*}
\begin{aligned}
&\mbox{ $\{ {\rm epi}\, f^\pi, \mathcal{L}+{\bm a}\}$ is nontrivially weakly feasible}\\
\Longleftrightarrow &
	  \mbox{ ${\bm a}\in \mathcal{L}$ and ${\bm z}\in {\rm bdry}(({\rm epi}\, f^\pi)^*)\cup {\rm bdry}(-({\rm epi}\, f^\pi)^*)\setminus \{{\bm 0}\}$}.
\end{aligned}
\end{equation*}
\end{lemma}
\begin{proof}
   We first prove the forward implication.
   From \eqref{nontrivial} and the definition of ${\rm d}_{\rm PPS}$ in Definition \ref{def-mis}, problem \eqref{prob-fea} being  nontrivially weakly feasible implies that ${\rm d}_{\rm PPS}({\rm epi}\, f^\pi,\mathcal{L}+{\bm a})= 1$ since ${\rm d}_{\rm PPS}({\rm epi}\, f^\pi,\mathcal{L}+{\bm a})\leq 1$ from Theorem \ref{the-eb} and ${\rm epi}\,f^\pi$ is non-polyhedral under Assumption~\ref{assu-blanket}. Then,  according to Proposition \ref{prop-facial-reduction}(ii), there exists a ${\bm z}^{(1)}\in ({\rm epi}\, f^\pi)^*\setminus \{{\bm 0}\}$ such that $\mathcal{L}+{\bm a} \subseteq \{{\bm z}^{(1)}\}^\bot$, which implies that  ${\bm a}\in \mathcal{L}$ and hence ${\cal L} +{\bm a} = {\cal L} = \{{\bm z}\}^\perp$. Consequently, we have $\{{\bm z}\}^\perp \subseteq \{{\bm z}^{(1)}\}^\bot$, which together with the fact that ${\bm z}^{(1)}\neq {\bm 0}$ shows that ${\bm z}^{(1)} = \gamma {\bm z}$ for some nonzero $\gamma$. This in turn implies that ${\bm z}\in ({\rm epi}\, f^\pi)^*\cup - ({\rm epi}\, f^\pi)^*\setminus\{{\bm 0}\}$.
   In addition, using ${\cal L}+{\bm a} = \{{\bm z}\}^\perp$, we obtain from \eqref{nontrivial} that
   \[
   {\rm epi}\, f^\pi\cap \{\bm z\}^\bot = {\rm epi}\, f^\pi\cap (\mathcal{L}+{\bm a}) \neq \{\bm 0\}
   \]
   Since ${\bm z}\in ({\rm epi}\, f^\pi)^*\cup - ({\rm epi}\, f^\pi)^*\setminus\{{\bm 0}\}$, the above display together with \eqref{pre-1} and the pointedness of $ {\rm epi}\, f^\pi$ (see Proposition~\ref{prop-pointed}) shows that ${\bm z}\in {\rm bdry}(({\rm epi}\, f^\pi)^*)\cup {\rm bdry}(-({\rm epi}\, f^\pi)^*)\setminus \{{\bm 0}\}$.
	
	Next, we  prove the backward implication.
	In this case, we have $\mathcal{L}+{\bm a}=\mathcal{L}$ thanks to ${\bm a}\in \mathcal{L}$. Thus, since $\mathcal{L}= \{\bm z\}^\bot$ with
${\bm z}\in {\rm bdry}(({\rm epi}\, f^\pi)^*)\cup {\rm bdry}(-({\rm epi}\, f^\pi)^*)\setminus \{{\bm 0}\}$, we have
\[
{\rm epi}\, f^\pi\cap (\mathcal{L} +{\bm a}) = {\rm epi}\, f^\pi\cap \mathcal{L}={\rm epi}\, f^\pi\cap \{\bm z\}^\bot \neq \{{\bm 0}\},
\]
where the nonzero intersection follows from Propositions \ref{prop-e-face} and
	\ref{prop-md-face}; see also Table~\ref{table-e-face}.
Next, since ${\bm z}\in {\rm bdry}(({\rm epi}\, f^\pi)^*)\cup {\rm bdry}(-({\rm epi}\, f^\pi)^*)\setminus \{{\bm 0}\}$,  it always holds that $\langle {\bm z}, {\bm v}\rangle \neq 0$  for any ${\bm v}\in {\rm int}({\rm epi}\, f^\pi)$.\footnote{To see this, suppose to the contrary that there exists some ${\bm v}'\in {\rm int}({\rm epi}\, f^\pi)$  such that $\langle {\bm z}, {\bm v}'\rangle=0$.  Then, there exists some small $\tau>0$ such that ${\bm v}'\pm \tau {\bm z}\in {\rm epi}\, f^\pi$. Hence, $\langle {\bm z}, {\bm v}'-\tau {\bm z}\rangle=-\tau \|{\bm z}\|^2<0$, which yields a contradiction when ${\bm z}\in {\rm bdry}(({\rm epi}\, f^\pi)^*)$. In addition, we also have $\langle {\bm z}, {\bm v}'+\tau {\bm z}\rangle=\tau \|{\bm z}\|^2>0$, which also yields a contradiction when ${\bm z}\in -{\rm bdry}(({\rm epi}\, f^\pi)^*)$}
	Then, we have
    \[
    {\rm int}({\rm epi}\, f^\pi)\cap (\mathcal{L} + {\bm a}) = {\rm int}({\rm epi}\, f^\pi)\cap \mathcal{L}=	{\rm int}({\rm epi}\, f^\pi)\cap \{\bm z\}^\bot= \emptyset.
    \]
	Therefore, problem \eqref{prob-fea} is nontrivially weakly feasible.
\end{proof}

According to Lemma \ref{lemma-trivial}, nontrivially weakly feasible problems as in \eqref{prob-fea}  with a 2-dimensional subspace $\mathcal{L}$ must have ${\bm a}\in {\cal L}$, and hence it suffices to consider feasibility problems of the form $\{{\rm epi}\, f^\pi, \mathcal{L}\}$ in this case. 
Now we are ready to present error bounds for almost every such feasibility problem.
\begin{corollary}[Error bounds for almost every \eqref{Feas} with $\mathcal{K}={\rm epi}\, f^\pi$] \label{coro-almost-eb}
	Let $\Upsilon$ be the collection of ${\bm z}\in \R^3$ such that the feasibility problem \eqref{prob-fea} with ${\cal L}:= \{{\bm z}\}^\perp$ is nontrivially weakly feasible.
	Then, for any $s > 1$, for $\mathcal{H}^s$-a.e. ${\bm z}\in \Upsilon$, $\{{\rm epi}\, f^\pi,  \mathcal{L}\}$ satisfies a uniform H\"{o}lderian error bound  with exponent $1/2$.
\end{corollary}
\begin{proof}
	It follows from Theorem  \ref{the-almost-1frf} and Lemma \ref{lemma-trivial}.
\end{proof}

\section{Final remarks and open questions}

A larger theme of this paper is to try to understand the error bound behavior that an arbitrary closed convex cone  may have.
Our results say that, intuitively, the pattern found in Table~\ref{table-eb} is not a coincidence, at least for cones constructed as in \eqref{per-cone}. Nevertheless,
the following general, albeit imprecise, questions remain.

\begin{quote}
	For arbitrary $\cK$, how mild are the error bounds associated to $\cK$? Is it possible to construct pathological convex cones that have ``bad'' error bounds everywhere (or almost everywhere)?
\end{quote}
Here, we leave some room for interpretation of what is meant by ``bad error bounds''.
For example, one possible interpretation is to consider that a cone is ``good'' if it has H\"olderian $\mathds{1}$-FRFs of exponent $1/2$ (or larger) for $\mathcal{H}^s$-a.e. ${\bm z}^*\in{\rm bdry}(\cK^*)$ for $s = \dim \cK -1$. Then, a ``bad'' cone would be one that is not ``good''.

It is tempting to speculate on the basis of Theorem~\ref{the-almost-1frf} that maybe there are no bad cones (at least in $\R^3$).
However, epigraphical cones as in \eqref{per-cone} are quite special because of Assumption~\ref{assu-blanket}, so the existence of very pathological convex cones cannot be completely ruled out.

\subsection{Error bound regularity}
Yet another way to define ``badness'' would be to forego the measure theoretical considerations and insist that ``good'' cones must be such that \emph{all} their $\mathds{1}$-FRFs are H\"olderian of exponent at least $1/2$.
This leads naturally to the following definition.

Let $\cK \subseteq \R^n$ be a closed convex cone.
We say that $\cK$ is \emph{error bound regular of exponent $\alpha$} if $\cK$ and ${\bm z}^*$ admit a H\"olderian $\mathds{1}$-FRF of exponent $\alpha$ for every ${\bm z}^* \in \cK^*$.
For short, we say that $\cK$ is \emph{$\alpha$-EBR}.
One may verify that if $\cK$ is $\alpha$-EBR, then it is also $\beta$-EBR for every $\beta \in (0,\alpha]$.
Therefore, if $\cK$ is $\alpha$-EBR, we can define its \emph{index of error bound regularity} ($\iEBR$) as the supremum of the $\alpha$ such that
$\cK$ is $\alpha$-EBR.\footnote{Note that the index may not be attainable.}  

All symmetric cones are $1/2$-EBR by \cite[Theorem~35]{Lourencco-2021}.
In view of Table~\ref{table-eb}, power cones are $\min\{1/p,1/2\}$-EBR  with $p\in[2,\infty)$, while $p$-cones are $\min\{1/p,1/2\}$-EBR with $p\in(1,\infty)$. The exponential cone is not $\alpha$-EBR for any $\alpha \in (0,1]$ owing to the fact that there is a case where there is no admissible H\"olderian $\mathds{1}$-FRF.
Polyhedral cones are always $1$-EBR, which is a consequence of \cite[Proposition~18]{Lourencco-2021}.\footnote{We remark that while there is a blanket assumption in \cite{Lourencco-2021} that the cones are pointed, \cite[Proposition~18]{Lourencco-2021} still holds without assuming pointedness of the cone.} In particular, their $\iEBR$ must be $1$ as well.

From this error bound perspective, ``good'' cones would correspond to the ones that are $\alpha$-EBR with $\alpha \in [1/2,1]$.
This leads to the following problem.

\begin{problem}\label{prob:good_eb}
Characterize the cones that are $\alpha$-EBR with $\alpha \in [1/2,1]$.
\end{problem}
It turns out that we can completely characterize the case $\alpha = 1$.

\begin{theorem}[$\mathds{1}$-FRF characterization of polyhedrality]
Let $\cK \subseteq {\cal E}$ be a closed convex cone.
If $\cK$ is $1$-EBR, then $\cK$ is polyhedral.	
\end{theorem}
\begin{proof}
For $\bx \in \cK $ denote by $D(\bx,\cK)$ the \emph{cone of feasible directions at $\bx$}, so that 	
\[
D(\bx,\cK) \coloneqq \{\bd : \exists\, t > 0 \text{ such that } \bx + t \bd \in \cK\}.
\]
Shapiro and Nemirovski proved that a cone is polyhedral if and only if every one of its cone of feasible directions is closed, see \cite[Lemma~1]{NS03}.
For this proof, our goal is to show that  $D(\bz^*,\cK^*)$ is closed for
every $\bz^* \in \cK^*$, which will imply that $\cK^*$ (and, therefore, $\cK$) is polyhedral.

Given $\bz^* \in \cK^*$, let $\hat \cF$ be the unique face of $\cK^*$ satisfying $\bz^* \in \ri(\hat \cF)$. Such a face exists by, say, \cite[Theorem~18.2]{Rockafellar-1970}.
With that, we have
\[
D(\bz^*,\cK^*) = \cK^* + \spanVec(\hat \cF),
\]
e.g., see \cite[Lemma~3.2.1]{Pa00}.
In order to prove that $\cK^* + \spanVec(\hat \cF)$ is closed, it suffices to show that the $\cK^{**} = \cK$ and $\spanVec(\hat \cF)^* = \hat \cF^\perp$ satisfy a Lipschitzian error bound, see \cite[Theorem~10]{Bauschke-Borwein-Li-1999}.\footnote{In the parlance of \cite{Bauschke-Borwein-Li-1999}, sets satisfying a Lipschitzian error bound are said to be \emph{boundedly linearly regular}.}

If ${\rm ri}({\cal K})\cap {\hat \cF}^\perp\neq \emptyset$, then $\cK$ and $\hat \cF^\perp$ satisfy a Lipschitzian error bound thanks to \cite[Corollary 3]{Bauschke-Borwein-Li-1999}. We now look at the case when ${\rm ri}({\cal K})\cap {\hat \cF}^\perp = \emptyset$, and we will proceed by invoking Theorem~\ref{the-error-bound} in an appropriate fashion.

Let $\cL \coloneqq \hat \cF^\perp$, ${\bm a} \coloneqq 0$.
Also, let $\cF_1 \coloneqq \cK$ and $\cF_2 \coloneqq \cK \cap \{\bz^*\}^\perp$.
We have
\begin{equation}\label{eq:ck_perp}
\cK \cap \cL = \cK \cap\hat \cF^\perp =  \cK \cap \{\bz^*\}^\perp,
\end{equation}
where the last equality follows from $\bz^* \in \ri(\hat \cF)$ and basic properties of the relative interior (e.g., see \cite[Theorem~6.1]{Rockafellar-1970}). Moreover, $\cF_2\subsetneq \cK$ because $\cF_2 = \cK \cap \hat \cF^\perp$ and ${\rm ri}({\cal K})\cap {\hat \cF}^\perp= \emptyset$.

We have that \eqref{eq:ck_perp} implies that $\cF_2 \subseteq \cL$ which leads to $\ri(\cF_2) \subseteq \cL$.
In particular, $\cL \cap \ri (\cF_2 ) \neq \emptyset$ holds, i.e.,
$\{\cF_2, \cL\}$ satisfy the PPS condition.
Let $\psi$ be a $\mathds{1}$-FRF for $\cK$ and  $\bz^*$, which, by the assumption that ${\cal K}$ is 1-EBR, we may assume that is H\"olderian of exponent $1$, i.e., $\psi(\epsilon,t) = \rho(t)\epsilon$ holds for all nonnegative $t$ and $\epsilon$, where $\rho$ is a nonnegative nondecreasing function.
We have thus satisfied all the assumptions of Theorem~\ref{the-error-bound} with $\ell = 2$ and $\varphi = \psi$.
Letting $\epsilon = \max\{d({\bm x},\mathcal{K}), {\rm d}({\bm x},\mathcal{L})\}$ in \eqref{eq:theo_eb}, we conclude that for every bounded set $B$ there exists $\bar{\kappa}_B > 0$ such that
\[
{\rm d}({\bm x},\cK \cap \cL)\leq \bar{\kappa}_B \max\{d({\bm x},\mathcal{K}), {\rm d}({\bm x},\mathcal{L})\}\quad \forall\, {\bm x}\in B,
\]
i.e., $\cK$ and $ \hat \cF^\perp$ satisfy  a Lipschitzian error bound.
\end{proof}

If we look at the $\iEBR$ of a cone, a curious pattern emerges.
In \cite{Lindstrom-Lourenco-Pong-2025}, asymptotically optimal $\mathds{1}$-FRFs were computed for $p$-cones and, in dimension at least $3$, every $p$-cone has cases where the optimal exponent for a certain ${\bm z}^*$ is $1/2$.
In fact, the discussion in \cite{Lindstrom-Lourenco-Pong-2025} implies that the $\iEBR$ for a $p$-cone in dimension at least three is $1/2$ if $p \in (1,2]$ and
$1/p$ if $p\in (2,\infty)$.
So, for any $\alpha$ in $(0,\frac{1}{2}]$, the $1/{\alpha}$-cone in three dimensions is an example having $\iEBR$ $\alpha$.
Similarly, for power cones, the discussion in \cite[Theorem~3.5, Remark~3.6]{Lin-et-2024} implies that their $\iEBR$ is always in $(0,1/2]$.

The summary so far is that polyhedral cones have $\iEBR$ $1$ but every other example of $\alpha$-EBR cone we know seems to have $\iEBR$ in $(0,1/2]$.
This motivates the following problem.
\begin{problem}\label{prob:cone}
For every $\alpha \in (1/2,1)$ construct a cone that has $\iEBR=\alpha$.	
\end{problem}
A difficulty in constructing  a cone as in Problem~\ref{prob:cone} is that certain obvious approaches using functions such as $x^{\alpha}, x^{1/\alpha}$ do not seem work because, as discussed previously, power cones and $p$-cones have $\iEBR$ in $(0,1/2]$.

\appendix

\section{Deducing error bounds for \eqref{Feas} using a $\mathds{1}$-FRF-based approach}\label{section-appedix-a1}

In this section, we review the $\mathds{1}$-FRF-based approach for analyzing the error bound for \eqref{Feas} with $(\mathcal{L}+{\bm a}) \cap \mathcal{K}\neq \emptyset$, for more details see \cite{Lindstrom-Lourenco-Pong-2023,Lindstrom-Lourenco-Pong-2025}.
We first recall a lemma that reduces the computation of a $\mathds{1}$-FRF to the computation of a specific type of error bound. 
\begin{lemma}[\!\!{\rm \cite[Lemma 3.9]{Lindstrom-Lourenco-Pong-2023}}]\label{lemma-frf}
Let $\mathcal{K}$ be a closed convex cone and let ${\bm z}\in \mathcal{K}^*$ be such that $\mathcal{F}\coloneqq\mathcal{K}\cap \{{\bm z}\}^\bot$ is a proper face of $\mathcal{K}$.
Let $\mathfrak{g}: \R_+\rightarrow \R_+$ be a nondecreasing  function with $\mathfrak{g}(0)=0$, and $\kappa_{{\bm z},s}$ be a finite  nondecreasing nonnegative function in $s\in\R_+$ such that
\begin{align}\label{lemma-frf-1}
{\rm d}({\bm q}, \mathcal{F})\leq \kappa_{{\bm z},\|{\bm q}\|} \mathfrak{g}({\rm d}({\bm q},\mathcal{K})) \quad \forall\, {\bm q}\in\{{\bm z}\}^\bot.
\end{align}
Then, a $\mathds{1}$-FRF $\psi_{\mathcal{K},{\bm z}}$ for $\mathcal{K}$ and ${\bm z}$ can be calculated as follows
\begin{align}\label{lemma-frf-2}
\psi_{\mathcal{K},{\bm z}}(\epsilon,t)\coloneqq \max\{\epsilon,\epsilon/\|{\bm z}\|\}+\kappa_{{\bm z},t} \mathfrak{g}(\epsilon+\max\{\epsilon,\epsilon/\|{\bm z}\|\}).
\end{align}
\end{lemma}
The following theorem gives a sufficient condition for  the error bound in  \eqref{lemma-frf-1}.
\begin{theorem}[\!\!{\rm \cite[Theorem 3.10]{Lindstrom-Lourenco-Pong-2023}}]\label{thm-frf}
Let $\mathcal{K}$ be a closed convex cone and let ${\bm z}\in \mathcal{K}^*$ be such that $\mathcal{F}\coloneqq\mathcal{K}\cap \{{\bm z}\}^\bot$ is a nontrivial exposed face of $\mathcal{K}$. Let $\eta\ge 0$, $\alpha\in (0,1]$, and $\mathfrak{g}: \R_+\to \R_+$ be a nondecreasing function such that $\mathfrak{g}(0)=0$ and ${\frak g}\geq |\cdot|^\alpha$. Define
\begin{align}\label{thm-frf-1}
\gamma_{{\bm z},\eta}:= \inf_{\bm v} \left\{\frac{\mathfrak{g}(\|{\bm w}-{\bm v}\|)}{\|{\bm w}-{\bm u}\|}: {\bm v}\in {\rm bdry}(\mathcal{K})\cap B(\eta)\backslash\mathcal{F}, {\bm w}={\rm P}_{\{{\bm z}\}^\bot} {\bm v}, {\bm u}={\rm P}_{\mathcal{F}} {\bm  w}, {\bm w}\neq {\bm u}\right\}.
\end{align}
If $\gamma_{{\bm z},\eta}\in(0,\infty]$, then it holds that
\begin{align}\label{thm-frf-2}
{\rm d}({\bm q},\mathcal{F})\leq \kappa_{{\bm z},\eta} \mathfrak{g}({\rm d}({\bm q},\mathcal{K})) \quad \forall\, {\bm q}\in\{{\bm z}\}^\bot  \cap B(\eta),
\end{align}
where $\kappa_{{\bm z},\eta}:=\max\{2\eta^{1-\alpha},2\gamma^{-1}_{{\bm z},\eta}\} < \infty$.
\end{theorem}
In order to use Theorem~\ref{thm-frf} we need to check that  $\gamma_{{\bm z},\eta}\in(0,\infty]$ for $\eta \geq 0$.
We always have $\gamma_{{\bm z},0} = \infty$, so we only need to consider $\eta > 0$.
The following lemma is useful for checking whether $\gamma_{{\bm z},\eta}\in(0,\infty]$ for $\eta > 0$.
\begin{lemma}[\!\!{\rm \cite[Lemma 3.12]{Lindstrom-Lourenco-Pong-2023}}]\label{lemma-gamma}
Let $\mathcal{K}$ be a closed convex cone and let ${\bm z}\in \mathcal{K}^*$ be such that $\mathcal{F}\coloneqq\mathcal{K}\cap \{{\bm z}\}^\bot$ is a nontrivial exposed face of $\mathcal{K}$.
Let $\eta >  0$,  $\mathfrak{g}$ and $\gamma_{{\bm z},\eta}$ be defined as in Theorem \ref{thm-frf}. If $\gamma_{{\bm z},\eta}=0$, then there exist $\bar{{\bm v}}\in\mathcal{F}$, and a sequence $\{{\bm v}^k\}\subset{\rm bdry}(\mathcal{K})\cap B(\eta)\backslash\mathcal{F}$ such that
\begin{align}
\lim_{k\rightarrow\infty} {\bm v}^k=\lim_{k\rightarrow\infty} {\bm w}^k=\bar{{\bm v}},\label{lemma-gamma-1}\\
{and}\ \ \lim_{k\rightarrow\infty} \frac{\mathfrak{g}(\|{\bm w}^k-{\bm v}^k\|)}{\|{\bm w}^k-{\bm u}^k\|}=0, \label{lemma-gamma-2}
\end{align}
where  ${\bm w}^k={\rm P}_{\{{\bm z}\}^\bot} {\bm v}^k, {\bm u}^k={\rm P}_{\mathcal{F}}{\bm w}^k$ and ${\bm w}^k\neq {\bm u}^k$.
\end{lemma}

\begin{lemma}[\!\!{\rm \cite[Lemma 2.4]{Lin-Lindstrom-et-2024}}]\label{lemma-bound-wu}
	Let $\mathcal{K}$ be a closed convex cone and let ${\bm z}\in {\rm bdry }(\mathcal{K}^*)\setminus \{\bm 0\}$ be such that $\mathcal{F}\coloneqq\mathcal{K}\cap \{{\bm z}\}^\bot$ is a nontrivial exposed face of $\mathcal{K}$.
		Let $\eta>0$, and let ${\bm v} \in {\rm bdry}(\mathcal{K})\cap B(\eta)\backslash\mathcal{F}$, ${\bm w}={\rm P}_{\{{\bm z}\}^\bot} {\bm v}, {\bm u}={\rm P}_{\mathcal{F}}{\bm w}$ and ${\bm w}\neq {\bm u}$.
	Then, it holds that
   \[
   \|{\bm w} - {\bm u}\|\leq  {\rm dist}({\bm v}, \mathcal{F}).
   \]
\end{lemma}


\begin{definition}[Miscellaneous definitions]\label{def-mis}
Suppose that \eqref{Feas} is feasible, i.e., $(\mathcal{L}+{\bm a}) \cap \mathcal{K}\neq \emptyset$.
	\begin{enumerate}[\rm (1)]
		\item  We say that $\{\mathcal{L}+{\bm a}, \mathcal{K}\}$ satisfies the partial-polyhedral Slater's (PPS) condition if one of the following is satisfied:
\begin{itemize}
	\item	 $(\mathcal{L}+{\bm a}) \cap {\rm ri}(\mathcal{K})\neq \emptyset$;
	\item  $(\mathcal{L}+{\bm a}) \cap \mathcal{K}\neq \emptyset$ and $\mathcal{K}$ is polyhedral;
	\item  up to a reordering of the coordinates,
	$(\mathcal{L}+{\bm a})$ and  $\mathcal{K}$ are such that $\mathcal{K}$ can be written as $\mathcal{K}=\mathcal{K}^{(1)}\times \mathcal{K}^{(2)}$ with polyhedral $\mathcal{K}^{(1)}$, and $(\mathcal{L}+{\bm a}) \cap (\mathcal{K}^{(1)}\times {\rm ri}(\mathcal{K}^{(2)}))\neq \emptyset$.

	\end{itemize}
	Put otherwise, $\mathcal{K}$ is a direct product of $m$ closed convex cones $\mathcal{K}^{(j)}$ and there exists ${\bm x} = ({\bm x}_1,\ldots, {\bm x}_m)\in (\mathcal{L}+{\bm a}) \cap \mathcal{K}$ such that ${\bm x}_{j} \in {\rm ri}(\mathcal{K}^{(j)})$ for every $j$ such that $\mathcal{K}^{(j)}$ is not polyhedral.
		\item  A finite collection of faces of $\mathcal{K}$ such that $\mathcal{F}_\ell\varsubsetneq\ldots \varsubsetneq\mathcal{F}_1 =\mathcal{K}$ is called a chain of faces with length $\ell$.
		We denote by $\ell_{\rm poly}(\mathcal{K})$ the {\it distance to polyhedrality of} $\mathcal{K}$ which is the length \emph{minus one} of a longest chain of faces of $\mathcal{K}$ such that \emph{only} the final face $\mathcal{F}_\ell$ is polyhedral.
		\item The distance to the PPS condition ${\rm d}_{\rm PPS}(\mathcal{K},\mathcal{L}+{\bm a})$ is defined as the length \emph{minus one} of a shortest chain of faces of $\mathcal{K}$ satisfying conditions (i)-(iii) in Proposition \ref{prop-facial-reduction} below.
		\end{enumerate}
	\end{definition}

In view of definition \ref{def-mis}(3),  ${\rm d}_{\rm PPS}(\mathcal{K},\mathcal{L}+{\bm a})=0$ if \eqref{Feas} satisfies the PPS condition.	
	If \eqref{Feas} satisfies the PPS condition, then a Lipschitzian error bound holds; see \cite[Proposition 2.3]{Lindstrom-Lourenco-Pong-2023} or \cite[Corollary 3]{Bauschke-Borwein-Li-1999}.
	If \eqref{Feas} do not satisfy the PPS condition, facial reduction algorithms \cite{Borwein-Wolkowicz-1981, Lourenco-Muramatsu-Tsuchiya-2018} can be applied to preprocess $\mathcal{K}$ in  such a way that  the PPS condition holds for ${\cal L}+{\bm a}$ and a face of ${\cal K}$ containing $({\cal L}+{\bm a})\cap {\cal K}$.

\begin{proposition}[\!\!{\rm \cite[Proposition 3.2]{Lindstrom-Lourenco-Pong-2023} and \cite[Proposition~8]{Lourenco-Muramatsu-Tsuchiya-2018}}]\label{prop-facial-reduction}
Let $\mathcal{K}=\mathcal{K}^{(1)}\times\cdots\times \mathcal{K}^{(m)}$ with each $\mathcal{K}^{(i)}$ being a closed convex cone. Suppose that \eqref{Feas} is feasible. Then there is a chain of faces with length $\ell$
\[
\mathcal{F}_\ell\varsubsetneq\ldots \varsubsetneq\mathcal{F}_1 =\mathcal{K}
\]
and vectors $({\bm z}^{(1)},\ldots, {\bm z}^{(\ell-1)})$ satisfying the following properties:
\begin{enumerate}[\rm (i)]
\item $\ell-1 \leq \sum_{i=1}^{m}\ell_{\rm poly}(\mathcal{K}^{(i)})\leq {\rm dim}\,\mathcal{K}$;
\item For all $i\in\{1,\ldots,\ell-1\}$, we have
${\bm z}^{(i)}\in \mathcal{F}_i^*\cap \mathcal{L}^\bot\cap\{{\bm a}\}^\bot$ and $\mathcal{F}_{i+1}=\mathcal{F}_i\cap\{{\bm z}^{(i)}\}^\bot$;
\item  $(\mathcal{L}+{\bm a}) \cap \mathcal{\mathcal{F}_\ell}= (\mathcal{L}+{\bm a}) \cap \mathcal{K}$ and $\{\mathcal{\mathcal{F}}_\ell,\mathcal{L}+{\bm a}\}$ satisfies
the PPS condition.
\end{enumerate}
\end{proposition}

	 Let $h_1, h_2:\R_+\times\R_+\rightarrow \R_+$.
	 Define the \emph{diamond composition} $h_1\diamond h_2$ as $(h_1\diamond h_2)(\epsilon,t)=h_1(\epsilon+h_2(\epsilon,t),t)$ for any $\epsilon,t\in \R_+$. For a facial residual function $\psi$,
we say that $\hat{\psi}$ is a \emph{positively rescaled shift} of $\psi$ if there exist positive constants $M_1$, $M_2$,  $M_3$ and nonnegative constant $M_4$ such that $\hat{\psi}(\epsilon,t)=M_3 \psi(M_1 \epsilon, M_2 t)+M_4 \epsilon$.

\begin{theorem}[\!\!{\rm \cite[Theorem 3.8]{Lindstrom-Lourenco-Pong-2023}}]\label{the-error-bound}
Suppose that \eqref{Feas} is feasible. Let
\[
\mathcal{F}_\ell\subsetneq\ldots \subsetneq\mathcal{F}_1=\mathcal{K}
\]
be a chain of faces of $\mathcal{K}$ with ${\bm z}^{(i)}\in \mathcal{F}^*_i\cap \mathcal{L}^\bot\cap \{{\bm a}\}^\bot$ such that $\{\mathcal{F}_\ell, \mathcal{L}+{\bm a}\}$ satisfies the PPS condition and $\mathcal{F}_{i+1}=\mathcal{F}_{i}\cap\{{\bm z^{(i)}}\}^\bot$ for every $i\in \{1,\dots,\ell-1\}$.
For $i\in \{1,\dots,\ell-1\}$, let $\psi_i$ be a $\mathds{1}$-FRF for $\mathcal{F}_i$ and ${\bm z}^{(i)}$.

Then, there is a suitable positively rescaled shift
of $\psi_i$ (still denoted as $\psi_i$ by an abuse of notation) such that for any bounded set $B$ there is a positive constant $\kappa_B$ (depending on $B, \mathcal{L}, {\bm a}, \mathcal{F}_\ell$) such that
\begin{equation}\label{eq:theo_eb}
{\bm x}\in B, d({\bm x},\mathcal{K})\leq \epsilon, {\rm d}({\bm x},\mathcal{L}+{\bm a})\leq\epsilon \Rightarrow {\rm d}({\bm x}, (\mathcal{L}+{\bm a}) \cap \mathcal{K})\leq \kappa_B(\epsilon+\varphi(\epsilon,M))),
\end{equation}
where $M=\sup_{{\bm x}\in B} \|{\bm x}\|$, $\varphi=\psi_{\ell-1}\diamond\ldots\diamond\psi_1$ if $\ell\geq 2$, and $\varphi(\epsilon,M)=\epsilon$ if $\ell=1$.
\end{theorem}

\section{Auxiliary lemmas and propositions}\label{section-appedix-a2}

We present several lemmas and propositions which will be useful for proving the main results of the paper.

\begin{lemma}\label{lemma-nonegative}
Let $h:\R\rightarrow \R_+\cup\{\infty\}$ be  proper, closed and convex. Then $h_\infty(d)\geq 0$ for all $d\in \R$.
\end{lemma}
\begin{proof}
Fix any $y\in {\rm dom}\,h$. For any $d\in \R$, we have
\[
h_\infty(d)=\lim_{t\rightarrow \infty}\frac{h(y+td)-h(y)}{t}=\lim_{t\rightarrow \infty}\frac{h(y+td)}{t}\geq 0,
\]
where the first equality follows from \eqref{def-f-infty}, the second equality holds because $y\in{\rm dom}\,h$, and the last inequality holds because $h$ is nonnegative.
\end{proof}

The following lemma concerns a representation of the recession function for functions on $\R$.
Note that it is slightly different from  \cite[Corollary~8.5.2]{Rockafellar-1970} because \eqref{lemma-infty-1} holds for all nonzero $x$ rather than just $x\in {\rm dom}\,h$.
\begin{lemma}\label{lemma-infty}
Let $h:\R\rightarrow \R\cup \{\infty\}$ be proper, closed and convex. Then for any $x\neq 0$ we have
\begin{align}\label{lemma-infty-1}
h_\infty(x)=\lim_{\lambda \downarrow 0} \lambda h(x/\lambda).
\end{align}
\begin{proof}
For $d\in \{-1,1\}$ and any fixed $y\in {\rm dom}\,h$, we have
\begin{equation}\label{hinfty}
\begin{aligned}
h_\infty(d) &= \lim_{t\rightarrow \infty}\frac{h(y+td)-h(y)}{t} = \lim_{t\rightarrow \infty}\frac{h(y+td)}{t}\\
& = \begin{cases}
  \lim_{t\rightarrow \infty}\frac{h(y+t)}{t} & {\rm if}\, d = 1,\\
  \lim_{t\rightarrow \infty}\frac{h(y-t)}{t} & {\rm if}\, d = -1,
\end{cases}
= \begin{cases}
  \lim_{s\rightarrow \infty}\frac{h(s)}{s} & {\rm if}\, d = 1,\\
  -\lim_{s\rightarrow -\infty}\frac{h(s)}{s} & {\rm if}\, d = -1,
\end{cases}
\end{aligned}
\end{equation}
where the first equality follows from \eqref{def-f-infty}, and the second equality holds because $y\in {\rm dom}\, h$.

Now, if $x>0$, we have
\[
\begin{aligned}
h_{\infty}(x)=x h_\infty(1) = x \lim_{s\to \infty} \frac{h(s)}{s} = x \lim_{\lambda\downarrow 0} \frac{h(x/\lambda)}{x/\lambda}=\lim_{\lambda \downarrow 0} \lambda h(x/\lambda),
\end{aligned}
\]
where the first equality follows from the sublinearity of $h_\infty$, the second inequality follows from \eqref{hinfty}.
 If $x<0$, we can similarly verify \eqref{lemma-infty-1} by using $h_\infty(x) = |x|h_\infty(-1)$ and \eqref{hinfty}.
\end{proof}
\end{lemma}

\begin{lemma}\label{lemma-1frf}
Let $\mathcal{K}$ be a two-dimensional closed convex cone and let ${\bm z}\in \mathcal{K}^*$. Then, there exists $\tau>0$ such that $\mathcal{K}$ and ${\bm z}$ admit a $\mathds{1}$-FRF $\psi_{\mathcal{K}, {\bm z}}:\R_+\times \R_+\rightarrow \R_+$ given by
\begin{align*}
\psi_{\mathcal{K}, {\bm z}}(\epsilon,t)=\tau \epsilon.
\end{align*}
\end{lemma}
\begin{proof}
Since $\mathcal{K}$ is a polyhedral cone, the conclusion follows directly from \cite[Example 3.6]{Lindstrom-Lourenco-Pong-2023}.
\end{proof}

\begin{proposition}\label{prop-duality}
If $f$ satisfies Assumption~\ref{assu-blanket}, then $f$ is strictly convex on ${\rm int}({\rm dom}\, f)$ and is a Legendre function.
\end{proposition}
\begin{proof}
 If ${\rm int}({\rm dom }\, f) \cap W=\emptyset$,  Assumption \ref{assu-blanket} states that $f$ is strictly convex on ${\rm int}({\rm dom }\, f)$.
Next, we consider the case that ${\rm int}({\rm dom }\, f) \cap W\neq \emptyset$.
Without loss generality, let
\begin{align*}
{\rm int}({\rm dom }\,f)\cap W=\{w_1,\ldots,w_n\} \quad \mbox{ with $w_1<w_2\cdots <w_n$}.
\end{align*}
Assumption \ref{assu-blanket} states that $f$ is strongly convex on every compact and convex subset of ${\rm int}({\rm dom}\,f)\setminus W$, which means that $f$ is strictly convex on each of the following open intervals
\begin{align}\label{prop-duality-1}
{\rm int}({\rm dom }\, f)\cap (-\infty, w_1),\,\, (w_1,w_2), \,\, \ldots, \,\, {\rm int}({\rm dom }\, f)\cap (w_n,\infty).
\end{align}
Thus the derivative $f'$ is increasing on each open intervals in \eqref{prop-duality-1}. In addition, the convexity of $f$ implies that $f'$ is nondecreasing on ${\rm int}({\rm dom }\,f)$. Then $f'$ is increasing on ${\rm int}({\rm dom }\,f)$.
From \cite[Theorem 2.1.5(iv)]{Zalinescu-2002}, we conclude that $f$ is strictly convex on ${\rm int}({\rm dom}\,f)$. Then $f$ is Legendre because $f$ is essentially smooth by assumption.
\end{proof}

\begin{proposition}\label{prop-pointed}
Suppose that $f$ satisfies Assumption~\ref{assu-blanket}.
Then the perspective cone  ${\rm epi}\, f^\pi$   given in \eqref{per-cone} is pointed and solid.
Moreover, the dual cone $({\rm epi}\, f^\pi)^*$ is also pointed and solid.
\end{proposition}
\begin{proof}
We first show that ${\rm epi}\, f^\pi$ is pointed.
Suppose to the contrary that ${\rm epi}\, f^\pi$ is not pointed.
 In view of the definition of ${\rm epi}\, f^\pi$ in \eqref{per-cone}, it holds that $\mathcal{K}^t\cap -\mathcal{K}^t =\emptyset$ and $\mathcal{K}^t\cap -\mathcal{K}^\theta =\emptyset$;
 thus, we must have the set $\mathcal{K}^\theta\cap -\mathcal{K}^\theta$ containing nonzero elements.
Also, it holds that  $f_\infty(d) \geq 0$ for any $d\in \R$ thanks to  Lemma \ref{lemma-nonegative} and the fact that $f$ is nonnegative.
Recall from \eqref{per-cone} that $\mathcal{K}^\theta={\{(x,0,s):\; (x,s)\in {\rm epi}\,f_\infty\}}$.
From these facts, we deduce that there exists $v_x\neq 0$ such that $(v_x,0)\in {\rm epi}\, f_\infty\cap -{\rm epi}\, f_\infty = {\rm epi}\, f_\infty\cap -({\rm epi}\, f)_\infty$.
According to \cite[Theorem 2.5.3]{Auslender-Alfred-Teboulle-2006}, this implies that
\begin{align*}
f(x+ t v_x)=f(x) \quad \forall\, x\in {\rm dom}\, f, t\in \R.
\end{align*}
 This contradicts the fact that  $f$ is strictly convex on ${\rm int}({\rm dom}\,f)$ (see Proposition \ref{prop-duality}).

 Next, we show that ${\rm epi}\, f^\pi$ is solid, i.e. ${\rm int}({\rm epi}\, f^\pi)\neq \emptyset$.
 According to Assumption \ref{assu-blanket}, we have that ${\rm int}({\rm dom}\, f)\neq \emptyset$.
 We can take some $(\hat{x},1,\hat{r})\in {\rm epi}\, f^\pi$ such that
 $\hat{x}\in {\rm int}({\rm dom}\, f)$ and $\hat{r}/1 > f(\hat{x}/1)$.
 Notice that $(a,b)\mapsto a/b$ is continuous around $(\hat{x},1)$ and $({\hat{r},1})$, and that
 $f$ is continuous around $\hat{x}$ (thanks to the fact that a convex function is continuous on the relative interior of its domain).
 We deduce that there exists a small $\delta>0$ such that  $(\hat{x},1,\hat{r})+B(\delta)\subset {\rm epi}\, f^\pi$. Then ${\rm int}({\rm epi}\, f^\pi)\neq \emptyset$.

 The claim for the dual cone follows directly from \cite[Proposition 2.4.3]{Facchinei-Pang-2003}.
\end{proof}

\begin{lemma}\label{lemma-dom-fstar}
Let $f$ satisfy Assumption~\ref{assu-blanket}.
Then the following statements hold.
\begin{enumerate}[\rm (i)]
\item ${\rm cl}({\rm dom}\, f^*)$ takes one of the following forms:
\begin{align}\label{lemma-dom-fstar-0}
[-a,\infty),\, (-\infty,b], \,[-a,b],  \, \R,
\end{align}
for some $a,b\geq 0$ with $b>-a$.
\item If  ${\rm bdry}({\rm dom}\, f^*)\cap {\rm dom}\, f^*\neq \emptyset$, then ${\rm dom}\,f^*$ takes one of the following forms:
\begin{align}\label{lemma-dom-fstar-1}
[-a,\infty),\, (-\infty, b],\, [-a, b),\, (-a,b],\, [-a,b],
\end{align}
for some $a,b\geq 0$ with $b>-a$.
\end{enumerate}
\end{lemma}
\begin{proof}
Since $\inf f\geq 0$, we see that $f^*(0)=\sup\{ 0\cdot x-f(x) : x\in \R \}=-\inf f \leq 0$, which implies that $0\in {\rm dom}\,f^*$.
In addition,  $f^*$ being Legendre means that ${\rm int}({\rm dom}\, f^*)\neq \emptyset$.
Combining these facts, we deduce item (i) and item (ii).
\end{proof}

\bibliographystyle{plain}
\bibliography{references}

\end{document}